\documentclass[12pt]{scrartcl}
\usepackage{a4}
\usepackage{amsthm}
\usepackage{amsmath}
\usepackage{amssymb}
\usepackage{amsfonts}
\usepackage{mathrsfs}
\usepackage{dsfont}
\usepackage{latexsym}
\usepackage{color}
\usepackage{bbm,exscale}
\definecolor{Myblue}{rgb}{0,0,0.6}
\usepackage[a4paper,colorlinks,citecolor=Myblue,linkcolor=Myblue,urlcolor=Myblue,pdfpagemode=None]{hyperref}
\usepackage[square,numbers,sort&compress]{natbib}
\usepackage[all,cmtip]{xy}
\usepackage{stmaryrd}
\usepackage{tikz}
\usetikzlibrary{decorations.pathmorphing}
\usetikzlibrary{calc}
\usetikzlibrary{decorations.markings}
\usetikzlibrary{fadings,decorations.pathreplacing}
\usetikzlibrary{matrix,arrows}
\usetikzlibrary{arrows,calc,decorations.pathreplacing,decorations.markings,shapes.geometric,shadows}

\tikzset{
    string/.style={draw=#1, postaction={decorate}, decoration={markings,mark=at position .51 with {\arrow[draw=#1]{>}}}},
    costring/.style={draw=#1, postaction={decorate}, decoration={markings,mark=at position .51 with {\arrow[draw=#1]{<}}}},
    ostring/.style={draw=#1, postaction={decorate}, decoration={markings,mark=at position .47 with {\arrow[draw=#1]{>}}}},
    ustring/.style={draw=#1, postaction={decorate}, decoration={markings,mark=at position .56 with {\arrow[draw=#1]{>}}}},
    oostring/.style={draw=#1, postaction={decorate}, decoration={markings,mark=at position .43 with {\arrow[draw=#1]{>}}}},
    uustring/.style={draw=#1, postaction={decorate}, decoration={markings,mark=at position .59 with {\arrow[draw=#1]{>}}}},
    directed/.style={string=blue!50!black}, 
    odirected/.style={ostring=blue!50!black}, 
    udirected/.style={ustring=blue!50!black}, 
    oodirected/.style={oostring=blue!50!black}, 
    uudirected/.style={uustring=blue!50!black},     
    redirected/.style={costring= blue!50!black},
    redirectedgreen/.style={costring= green!50!black},
    directedgreen/.style={string= green!50!black},
}

\tikzset{-dot-/.style={decoration={
  markings,
  mark=at position 0.5 with {\fill circle (2pt);}},postaction={decorate}}}

\tikzset{
	Fdot/.style={circle, draw, fill, inner sep=0pt}, 
	Odot/.style={circle, draw, inner sep=0.1pt, minimum size=0.1cm}
	}

  \vfuzz \hfuzz
\makeatletter
\newcommand{\raisemath}[1]{\mathpalette{\raisem@th{#1}}}
\newcommand{\raisem@th}[3]{\raisebox{#1}{$#2#3$}}
\makeatother

\newcommand{\C}{\mathds{C}}

\newcommand{\N}{\mathds{N}}

\newcommand{\R}{\mathds{R}}
\newcommand{\Z}{\mathds{Z}}
\def\1{\ifmmode\mathrm{1\!l}\else\mbox{\(\mathrm{1\!l}\)}\fi}
\newcommand{\one}{\mathbbm{1}}
\newcommand{\be}{\begin{equation}}
\newcommand{\ee}{\end{equation}}
\newcommand{\bes}{\begin{equation*}}
\newcommand{\ees}{\end{equation*}}

\newcommand{\id}{\text{id}}

\newcommand{\Hom}{\operatorname{Hom}}
\newcommand{\End}{\operatorname{End}}

\newcommand{\ev}{\operatorname{ev}}

\newcommand{\coev}{\operatorname{coev}}

\def\lra{\longrightarrow}
\def\lmt{\longmapsto}

\DeclareMathOperator{\Res}{Res}
\newcommand*{\longhookrightarrow}{\ensuremath{\lhook\joinrel\relbar\joinrel\rightarrow}}

\newcommand*{\longhookleftarrow}{\ensuremath{\leftarrow\joinrel\relbar\joinrel\rhook}}

\newcommand{\Bord}{\operatorname{Bord}}

\newcommand{\zz}{\mathcal{Z}}

\newcommand{\Obj}{\mathrm{Obj}}

\newcommand{\Vect}{\operatorname{Vect}}
\newcommand{\Vectk}{\operatorname{Vect}_\Bbbk}
\newcommand{\vectk}{\operatorname{vect}_\Bbbk}
\newcommand{\eps}{\varepsilon}

\def\cedille#1{\setbox0=\hbox{#1}\ifdim\ht0=1ex \accent'30 #1%
 \else{\ooalign{\hidewidth\char'30\hidewidth\crcr\unbox0}}\fi}

\newcommand\doi[2]        {\href{http://dx.doi.org/#1}{#2}}

\newcommand\tikzzbox[1]    {#1}%

\allowdisplaybreaks

\deffootnote[1em]{1em}{1em}
{\textsuperscript{\thefootnotemark}}

\theoremstyle{definition}
\newtheorem{definition}{Definition}
\newtheorem{proposition}[definition]{Proposition}
\newtheorem{theorem}[definition]{Theorem}
\newtheorem{lemma}[definition]{Lemma}
\newtheorem{corollary}[definition]{Corollary}
\newtheorem{remark}[definition]{Remark}

\newtheorem{examples}[definition]{Examples}

\numberwithin{equation}{section}
\numberwithin{definition}{section}
\numberwithin{figure}{section}

\newcommand\void[1]{}

\begin{document}

\title{Introductory lectures on \\ topological quantum field theory}

\author{%
\!\!\!\!\!\!\!Nils Carqueville$^*$ \quad
Ingo Runkel$^\#$ \quad
\\[0.5cm]
   \hspace{-1.8cm}  \normalsize{\texttt{\href{mailto:nils.carqueville@univie.ac.at}{nils.carqueville@univie.ac.at}}} \\  %
   \hspace{-1.8cm}  \normalsize{\texttt{\href{mailto:ingo.runkel@uni-hamburg.de}{ingo.runkel@uni-hamburg.de}}} \\
[0.1cm]
   \hspace{-1.2cm} {\normalsize\slshape $^*$Fakult\"at f\"ur Mathematik, Universit\"at Wien, Austria}\\[-0.1cm]
   \hspace{-1.2cm} {\normalsize\slshape $^\#$Fachbereich Mathematik, Universit\"{a}t Hamburg, Germany}\\[-0.1cm]
}

\date{}
\maketitle

\begin{abstract}
These notes offer a lightening introduction to topological quantum field theory in its functorial axiomatisation, assuming no or little prior exposure. 
We lay some emphasis on the connection between the path integral motivation and the definition in terms symmetric monoidal categories, and we highlight the algebraic formulation emerging from a formal generators-and-relations description. 
This allows one to understand (oriented, closed) 1- and 2-dimensional TQFTs in terms of a finite amount of algebraic data, while already the 3-dimensional case needs an infinite amount of data. 
We evade these complications by instead discussing some aspects of 3-dimensional extended TQFTs, and their relation to braided monoidal categories. 
\end{abstract}

\newpage

\tableofcontents

\section{Introduction}
\label{sec:intro}

Topological quantum field theories are a rewarding area of study in mathematical physics and pure mathematics.
They appear in the description of physical systems such as the fractional quantum Hall effect; they are used in topological quantum computing; they are important renormalisation group flow 
invariants of supersymmetric field theories obtained via ``twisting''; they have a clean mathematical axiomatisation; they give invariants of knots and of manifolds; they play vital roles in mirror symmetry and the Geometric Langlands Programme. 

\medskip

In these lecture notes (originally prepared for the ``Advanced School on Topological Quantum Field Theory'' in Warszawa, December 2015) we approach topological quantum field theory by motivating its axioms through path integral considerations. The resulting description in terms of symmetric monoidal functors from bordisms to vector spaces is introduced and some of its immediate consequences are expounded on, still in general dimension $n$.

We then turn to generators-and-relations descriptions of topological quantum field theories, which allow one to cast their study into algebraic language. The resulting algebraic structure is presented in some detail for dimensions one and two. In much less detail we describe the difficulties and resulting modifications of the formalism in dimension three.

\medskip

The entire content 
of these lectures can basically be found in the literature, for example in \cite{bdspv1509.06811,BakalovKirillov, Gawedzki:1999bq,Kockbook}.
Some features which maybe set these notes apart are:
\begin{itemize}
\item[-]
Some effort is spent on comparing the path integral motivation and the symmetric monoidal functor axiomatisation, which initially look quite different (Section \ref{sec:compare}).
\item[-]
Dimensional reduction is briefly discussed; the reduction from three to two dimensions and its relation to the Grothendieck ring is treated as an example (Sections~\ref{subsec:basicprop} and~\ref{subsec:dimred3to2}).
\item[-]
It is explained in some detail what it means for a symmetric monoidal category to be freely generated by some objects and morphisms, subject to some 
relations (Section \ref{sec:gen+rel_def}). 
\end{itemize}

To keep these notes at a reasonable length (or so we hope), we will sometimes skip details or entire proofs -- especially in dimension three. Nonetheless, we have tried to still be accurate or to point out when we gloss over details.
We hope that in this way we can convey some ideas about topological quantum field theory before the reader feels inclined to 
move on to the next interesting topic. 

\bigskip
\noindent
\textbf{Acknowledgements:} 
We are thankful to 
	Manuel B\"arenz, 
	Vanda Farsad, 
	Lorenz Hilfiker,
	Flavio Montiel Montoya, 
	Albert Georg Passegger
and
	Lorant Szegedy 
for comments on a draft of these lecture notes.
The work of N.\,C.~is partially supported by a grant from the Simons Foundation, and by the stand-alone project P\,27513-N27 of the Austrian Science Fund. 
I.\,R.~thanks Nils Carqueville, Piotr Su\l kowski and Rafa\l\ Suszek for organising the ``Advanced School on Topological Quantum Field Theory'' in Warszawa, December 2015, and for their dedicated hospitality during the event. 
We also acknowledge the support of the Faculty of Physics and the Heavy Ion Laboratory at the University of Warsaw, as well as Piotr Su\l kowski's ERC Starting Grant no.~335739 ``Quantum fields and knot homologies''.

\section{Axioms for TQFTs}
\label{sec:closedTQFTs}

In this section we introduce $n$-dimensional (closed, oriented) topological quantum field theories (TQFTs) for any $n\in \Z_+$, and we discuss some of their general properties. 
We start in Section~\ref{subsec:moti} with a list of desired properties which are motivated from heuristic path integral considerations. 
Making the desiderata into defining properties is the idea behind the axiomatic approach presented in Section~\ref{subsec:fun}. 
There we spell out various details of the functorial definition of TQFTs, and in Section~\ref{sec:compare} we make the connection between desired and defining properties precise and explicit. 
Then in Section~\ref{subsec:basicprop} we prove a few general consequences of these properties, and in Section~\ref{sec:compareTQFT} we explain why monoidal natural isomorphisms are a good notion to compare TQFTs.

\subsection{Path integrals as a motivation}
\label{subsec:moti}

While there are only few instances in which the path integral can be made precise, it is doing very well as a source of intuition about quantum field theory. In this subsection we would like to use this intuition to motivate the functorial description of TQFTs due to Atiyah and Segal \cite{Segal:1988,AtiyahTQFT}, given in Section~\ref{subsec:fun}. 
No attempt at mathematical rigour is made in this subsection, and nothing from here is needed in the rest of these notes. 

So, without further ado, here is a path integral:
\be
	\zz = \int D\Phi \, \textrm{e}^{-S[\Phi]} \, .
\ee
The ingredients in this example are as follows.
\begin{itemize}
\item 
$\Phi \colon M \to X$ is a smooth map between two Riemannian manifolds $M,X$. 
For example one can choose $X = \R$ to describe a single scalar field.
\item 
$S[\Phi]$ is called the action functional, typically it would depend on the value of $\Phi$ at each point of $M$ and its first derivative, 
\be 
	S[\Phi] = \int_M \, \mathcal{L}(\Phi,\partial_\mu \Phi)(x) \sqrt{\det(g)}\, \textrm{d}^n x \, ,
\ee 
where $\sqrt{\det(g)}\, \textrm{d}^n x$ is the volume form on $M$ and $\mathcal{L}$ is called Lagrangian density. In the example of a massless free scalar field $\phi\colon M \to \R$ we have 
$\mathcal{L} = \tfrac12 g^{\mu\nu} \, \partial_\mu \phi \, \partial_\nu \phi$.
\item
$\int D\Phi$, the ``integral over all such $\Phi$'' makes no sense.
(Well, for $M$ 1-dimensional, there is the Wiener measure \cite{SimonFI}, but this is really quantum mechanics, and in higher dimensions the measure theoretic approach needs modifications, if it can be salvaged at all.)
\end{itemize}

Next we want to use the path integral to formulate correlation functions. Let $\mathcal{O}_1, \dots , \mathcal{O}_n$ be some observables, that is, functions from the set of field configurations $\{\Phi\colon M \to X\}$ to the complex numbers. 
In the example of a scalar field~$\phi$ these could be $\mathcal{O} \in \{ \phi(x), \,\partial \phi(x), \,\phi(x)^2, \dots \}$. 
Note that these examples of observables only probe the field $\phi$ at a single point~$x$. 
One can also consider observables which are localised on lines, surfaces, etc.; some aspects of this are discussed in \cite{2dTQFT}. 
Wilson loop observables in gauge theory are another important example. 

The correlation function of the observables $\mathcal{O}_1, \dots, \mathcal{O}_n$ is
\be 
\big\langle \mathcal{O}_1 \cdots \mathcal{O}_n \big\rangle_g
~=~
\frac{1}{\zz} \int D\Phi \, \mathcal{O}_1 \cdots \mathcal{O}_n \, e^{-S[\Phi]} \, ,
\ee 
where $g$ indicates the metric on $M$. If this correlation function is independent of~$g$ we have a \textsl{topological quantum field theory}.
Note that metric independence implies diffeomorphism invariance: isometries $f\colon M \to M'$ should leave the path integral invariant (transporting the observables accordingly), but any diffeomorphism $f\colon M \to M'$ can be made into an isometry by modifying the metric on $M$ to the pull-back metric.
	
The most famous example of a topological quantum field theory is Chern-Simons theory \cite{wittenjones}.
We will not study Chern-Simons theory in these notes, but let us at least describe the action functional. We need
\begin{itemize}
\item a 3-dimensional compact oriented manifold $M$, 
\item a compact Lie group $G$, which we will take to be $SU(N)$, 
\item a principal $G$-bundle $P \to M$ (which, as $G=SU(N)$, can always be trivialised, and we choose such a trivialisation),
\item a connection 1-form $A \in \Omega^1(M,\mathfrak{g})$, for $\mathfrak{g}$ the Lie algebra of $G$.
\end{itemize}
Here $A$ plays the r\^ole of the field $\Phi$ from above, and
the action functional reads
\be 
	S[A] ~=~ \gamma \, \int_M
	\mathrm{Tr}\big( A \wedge dA + \tfrac23 A \wedge A \wedge A \big) \, ,
\ee 
where $\gamma \in \C^\times$ is a constant and $\mathrm{Tr}(-)$ is the matrix trace.
The first thing to notice about this action is the absence of any metric. 
Instead, the integrand is already a 3-form on~$M$.
Thus at least the action is certainly topological. 

There is much more to be said about Chern-Simons theory, but we will leave this example here and return to the general discussion of motivating Atiyah's axioms from path integrals.

\medskip

When trying to make enough sense of the path integral description to compute correlation functions, there are two very different approaches one can take:
\begin{itemize}
\item One can try to ``quantise the classical action''. For topological theories this is fiendishly difficult due to a large amount of gauge freedom which needs to be fixed. 
In the Chern-Simons example, this actually re-introduces a metric dependence \cite{wittenjones, Schwarz1978} (which however disappears again in the correlation functions).
\item One can try to fix a collection of properties one would expect of a sensible definition of the path integral, and then try to write down solutions to these, or even to classify all solutions.
\end{itemize}
We will take the second route. The properties we expect for an $n$-dimensional TQFT are as follows (the precise formulation will follow below in Sections~\ref{subsec:fun} and~\ref{sec:compare}):
\begin{enumerate}
\item 
For each oriented $(n-1)$-manifold~$E$ we would like to obtain a Hilbert space $\mathcal{H}_E$ of states on $E$. 

One can think of $\mathcal{H}_E$ as the ``space of functionals on the classical fields on~$E$''. 
(Though actually this space will turn out finite-dimensional in the case of TQFTs.)
\item
For each oriented $n$-manifold $M$ with boundary $\partial M = E$, we expect to obtain a vector $\zz(M) \in \mathcal{H}_E$.

Continuing with our analogy, we have to produce a functional on fields on~$E$ from the manifold~$M$. To do so, we will appeal to the path integral. 
Namely, let $\varphi$ be a field on~$E$. 
We can then consider the path integral over fields $\Phi$ on $M$ which restrict to $\varphi$ on the boundary,
\be
\label{eq:ZMphi}
	\zz(M)(\varphi) ~=~ \int_{\Phi \text{ on $M$ s.t. } \Phi|_E = \varphi}
	\hspace{-6em}
	 D\Phi ~ \textrm{e}^{-S[\Phi]} \hspace{2em} .
\ee
Then, if these were well-defined expressions, the right-hand side would indeed produce a number for each $\varphi$ one plugs in.

\item 
If the $(n-1)$-manifold $E$ is a disjoint union $E = E_1 \sqcup E_2$ we expect $\mathcal{H}_E = \mathcal{H}_{E_1} \otimes \mathcal{H}_{E_2}$.

In our analogy, the space $\mathcal{M}(E_1 \sqcup E_2)$ of maps from $E_1 \sqcup E_2$ to somewhere (the ``classical fields'') equals the Cartesian product $\mathcal{M}(E_1) \times \mathcal{M}(E_2)$. And if these spaces of maps were finite sets, the linear space $\mathcal{F}(E)$ of functionals $\mathcal{M}(E) \to \C $ would satisfy $\mathcal{F}(E_1 \sqcup E_2) = \mathcal{F}(E_1) \otimes \mathcal{F}(E_2)$.

Similarly, if we are given $n$-manifolds $M_1$, $M_2$ with $\partial M_1 = E_1$, $\partial M_2 = E_2$, then we expect $\zz(M_1 \sqcup M_2) = \zz(M_1) \otimes \zz(M_2)$. Indeed, we would like the action to be ``local enough'' so that for $\Phi_i = \Phi|_{M_i}$, $i\in\{1,2\}$, we have $S[\Phi] = S[\Phi_1] + S[\Phi_2]$ and the integral in (ii) decomposes into a product.

\item
For cylinders $M = E \times [0,1]$, the path integral produces an element $\zz(M) \in \mathcal{H}_{E} \otimes \mathcal{H}_{\overline E}$, where $\overline E$ is $E$ with reversed orientation. 
One would expect this element to be nondegenerate in the following sense. Pick two vectors $u,u' \in \mathcal{H}_{\overline E}$. If the contraction of $u$ into the tensor factor $\mathcal{H}_{\overline E}$ of  $\zz(M)$ via the scalar product $\langle-,-\rangle$ produces the same element of $\mathcal{H}_{E}$ as the contraction with~$u'$, then already $u=u'$. We obtain an injective anti-linear map $\mathcal{H}_{\overline E} \to \mathcal{H}_{E}$. 

In the path integral picture, this captures the idea that  on a very thin cylinder one can tell apart any two states placed on one boundary by  choosing an appropriate state to place on the other boundary. For, if this were not the case, the Hilbert spaces do not yet describe the ``physical states'' and one should instead work with an appropriate quotient. 
For example, if in a gauge theory one does not identify gauge equivalent field configurations, the corresponding states would be indistinguishable for arbitrarily thin cylinders and so describe the same physical state.
(Of course in a TQFT, ``very thin'' has no meaning, there is only one cylinder over a given $(n-1)$-manifold up to diffeomorphisms.)

\item
Consider a manifold $M$, possibly with non-empty boundary $E = \partial M$. Now embed a closed $(n-1)$-manifold $U$ into $M$ and cut $M$ open along $U$. This produces a new $n$-manifold $N$ with boundary $\partial N = E \sqcup U \sqcup \overline U$. Let $\{ e_i \}$ be an orthonormal basis of $\mathcal{H}_U$ and let $\overline e_i$ be the preimage of $e_i$ under the anti-linear map from item (iv) above (the existence of the preimage, i.\,e.\ surjectivity, is a property specific to TQFT). We demand
\be 
	\zz(M) = \sum_i \big\langle e_i \otimes \overline{e}_i \,,\, \zz(N) \big\rangle 
	\; \in \mathcal{H}_E \, ,
\ee 
where we think of $\langle e_i \otimes \overline{e}_i , - \rangle$ as a map $\mathcal{H}_E \otimes \mathcal{H}_U \otimes \mathcal{H}_{\overline U} \to \mathcal{H}_E$. 
(In accordance with the finite-dimensionality in (i), for TQFTs, the above sum will be finite.)

In our analogy, this property captures the ``sum over intermediate states''. 
For the path integral, it would mean that the integral \eqref{eq:ZMphi} over all $\Phi$ with $\Phi|_E = \varphi$ can be split into first integrating over all $\Phi$ such that $\Phi|_E = \varphi$ and in addition $\Phi|_U = \psi$ for some $\psi \in \mathcal{M}(U)$, and then integrating the result over $\psi$:
\be 
	\zz(M)(\varphi)
	~=~
	\int_{\psi \text{ on } U} \hspace{-2em} D\psi \hspace{1em}
	\int_{\Phi \text{ on $M$ s.t. } \Phi|_E = \varphi \,,\, \Phi|_U = \psi}
	\hspace{-8.5em} 
	D\Phi ~ \textrm{e}^{-S[\Phi]} \hspace{4.3em} .
\ee 
\end{enumerate}

These are the properties of which a close variant will now be turned into axioms.

\subsection{TQFTs as functors}
\label{subsec:fun}

We regard quantum field theory from an angle where it appears as a way of transporting the geometric and dynamical structure of spacetime into the algebraic description of physical states and observables. 
{}From 
the functorial perspective a QFT is a map 
\be 
\textrm{`geometry'} \lra \textrm{`algebra'}
\ee 
that preserves certain structure. 
It is a difficult and important problem to make this precise for general (and realistic) QFTs. 
For the special case when ``geometry'' basically only means ``topology'', the answer was given by Atiyah and Segal 
\cite{Segal:1988,AtiyahTQFT}: 

\begin{definition}
\label{def:closedTQFT}
An \textsl{$n$-dimensional oriented closed TQFT} (a \textsl{TQFT} for short) is a symmetric monoidal functor 
\be 
\zz\colon \Bord_n \lra \Vectk
\, . 
\ee 
\end{definition}

In this section we will unravel this statement by discussing the two categories $\Bord_n$ and $\Vectk$, highlighting the key properties of the functor~$\zz$, and indicating how it encodes the structure motivated in Section~\ref{subsec:moti}. 
For a neat technical review of symmetric monoidal categories, functors and their natural transformations we recommend \cite{BaezEveryone}, or a textbook such as \cite{EGNO}. 

\medskip

We begin with the category $\Vectk$. 
It is named after its objects which are $\Bbbk$-vector spaces for some field~$\Bbbk$,\footnote{In quantum physics, $\Bbbk = \C$.} and morphisms are $\Bbbk$-linear maps. 

By definition one can compose morphisms in any category, but $\Vectk$ is better than that: it is an example of a \textsl{monoidal} category, which means that objects, too, can be `composed' in the following sense: 
any $U, V \in \Vectk$ can be `multiplied' by forming the tensor product $U \otimes_\Bbbk V \in \Vectk$. 
This is a good product in the sense that it is associative up to natural isomorphism (satisfying the ``pentagon equations'' \cite[Def.\,6]{BaezEveryone} or \cite[Sect.\,2.1]{EGNO}), 
\be\label{eq:Vectass}
(U \otimes_\Bbbk V) \otimes_\Bbbk W 
\cong 
U \otimes_\Bbbk (V \otimes_\Bbbk W)
\, , 
\ee
and $\Bbbk \in \Vectk$ is a unit since
\be 
\Bbbk \otimes_\Bbbk V \cong V \cong V \otimes_\Bbbk \Bbbk
\, . 
\ee 
Better still, since one can also take the tensor product of linear maps, $\otimes_\Bbbk$ is a functor
\be 
\Vectk \times \Vectk \lra \Vectk 
\, . 
\ee 

The monoidal category $\Vectk$ also has a \textsl{symmetric} structure: for all $U, V \in \Vectk$ there are natural isomorphisms 
\be 
\beta_{U,V}: U \otimes_\Bbbk V \lra V \otimes_\Bbbk U
\ee 
(given by $u \otimes v \mapsto v \otimes u$) which are compatible with the isomorphism~\eqref{eq:Vectass} in the sense that they satisfy the ``hexagon equations'' \cite[Def.\,7]{BaezEveryone}, \cite[Sect.\,8.1]{EGNO}
and we have the symmetry property  
\be 
\beta_{U,V} = \beta_{V,U}^{-1} 
\, . 
\ee 

\medskip

The other symmetric monoidal category we need is $\Bord_n$, which is named after its morphisms. 
Objects in $\Bord_n$ are oriented closed $(n-1)$-dimensional real manifolds~$E$ for some fixed $n\in \Z_{\geqslant 1}$. 
We may think of~$E$ as a toy model of a spacial slice of an $n$-dimensional spacetime. 

A morphism $E\to F$ in $\Bord_n$ is an equivalence class of a bordism from~$E$ to~$F$. 
A \textsl{bordism} $E \to F$ is an oriented compact $n$-dimensional manifold~$M$ with boundary, together with smooth maps $\iota_{\textrm{in}}\colon  E \to M \leftarrow F \,\colon\!\! \iota_{\textrm{out}}$ with image in $\partial M$ such that 
\be 
\overline \iota_{\textrm{in}} \sqcup \iota_{\textrm{out}}\colon 
\overline{E} \sqcup F \lra \partial M
\ee 
is an orientation-preserving diffeomorphism, where $\overline{E}$ denotes~$E$ with the opposite orientation. 
Two bordisms $(M, \iota_{\textrm{in}}, \iota_{\textrm{out}}), (M', \iota'_{\textrm{in}}, \iota'_{\textrm{out}})\colon  E \to F$ are \textsl{equivalent} if there exists an orientation-preserving diffeomorphism $\psi\colon M \to M'$ such that
\be 
\begin{tikzpicture}[
			     baseline=(current bounding box.base), 
			     >=stealth,
			     descr/.style={fill=white,inner sep=3.5pt}, 
			     normal line/.style={->}
			     ] 
\matrix (m) [matrix of math nodes, row sep=1.5em, column sep=3.0em, text height=1.5ex, text depth=0.1ex] {%
& M &
\\
E & & F
\\
& M' &
\\
};
\path[font=\footnotesize] (m-2-1) edge[->] node[above] {$ \iota_{\textrm{in}} $} (m-1-2);
\path[font=\footnotesize] (m-2-1) edge[->] node[below] {$ \iota'_{\textrm{in}} $} (m-3-2);
\path[font=\footnotesize] (m-2-3) edge[->] node[above] {$ \iota_{\textrm{out}} $} (m-1-2);
\path[font=\footnotesize] (m-2-3) edge[->] node[below] {$ \iota'_{\textrm{out}} $} (m-3-2);
\path[font=\footnotesize] (m-1-2) edge[->] node[auto] {$ \psi $} (m-3-2);
\end{tikzpicture}
\ee 
commutes. 
This is how the smooth geometric structure is discarded in $\Bord_n$. 

Composition of morphisms $M_1\colon  E \to F$ and $M_2\colon  F \to G$ in $\Bord_n$ is given by ``gluing $M_1$ and $M_2$ along~$F$''.

\begin{remark}\label{rem:luck-with-gluing}
There is a subtlety in this definition of composition. The gluing along~$F$ fixes $M_2  \sqcup_F M_1$ a priori only as a topological space,  but we need a smooth structure as well. There are two ways to produce such a structure. One is to work with ``$(n-1)$-manifolds with collars'', so that the objects of the bordism category are of the form $E \times (-1,1)$, and the boundary parametrisations are now defined on $E \times (-1,0]$ and $E \times [0,1)$, respectively. 

The other way is to note that we are in luck in that even without the collars there are such smooth structures. 
Namely, we can appeal to the nontrivial fact that the colimit $M_2 \sqcup_F M_1$ exists in topological spaces and can be used to construct a smooth structure on \mbox{$M_2 \sqcup_F M_1$} which is unique up to (non-unique) diffeomorphism, see e.\,g.~\cite[Thm.\,1.3.12]{Kockbook}. 
Since morphisms in $\Bord_n$ are diffeomorphism classes of smooth $n$-manifolds, this is good enough for us. However, if one wants to discuss ``extended TQFTs'', as we will briefly do in Section~\ref{sec:extendedTQFTs}, ``unique up to diffeomorphism'' is no longer good enough and one has to work with collars.
\end{remark}

We note that in particular every diffeomorphism $\phi\colon E_1 \to E_2$ gives rise to an isomorphism $E_1 \cong E_2$ in $\Bord_n$ via the \textsl{cylinder construction}: 
the associated bordism is the cylinder
\be 
\label{eq:cylinderconstruction}
E_1 \cong E_1 \times \{ 0 \} \longhookrightarrow E_1 \times  [0,1] \longhookleftarrow
E_1\times \{ 1 \} \stackrel{\phi}{\cong} E_2 \times \{ 1 \} \cong E_2
\, .
\ee 

The bordism category $\Bord_n$ has a symmetric monoidal structure, too, with the tensor product given by disjoint union. 
The unit object is the empty set~$\emptyset$ viewed as an $(n-1)$-dimensional manifold. 
We obviously have $\emptyset \sqcup E = E = E \sqcup \emptyset$ for every $E \in \Bord_n$, and by definition (as a universal coproduct) taking disjoint unions is associative. 
Finally, the canonical diffeomorphism $E \sqcup F \to F \sqcup E$ induces (via the cylinder construction~\eqref{eq:cylinderconstruction}) a symmetric braiding 
\be\label{eq:Bordbraid}
\beta_{E,F}\colon E \sqcup F \lra F \sqcup E
\ee
on $\Bord_n$. 

\begin{remark}
\label{rem:manyBords}
There are several variants of bordism categories (and hence of functorial QFTs). 
Details and complexity may vary significantly, but the main distinguishing feature is the type of geometric or combinatorial 
structure one considers. 
For example, bordisms can come equipped with a metric, with a conformal structure,  with a spin structure,  with a framing, with boundaries, with embedded submanifolds, with homotopy classes of maps into some classifying space, etc. 
Correspondingly, functorial QFTs have many names (conformal QFT, spin / framed / open-closed / defect / homotopy TQFT, etc.). 
\end{remark}

TQFTs defined on compact oriented bordisms whose entire boundary is parametrised by the source and target objects are called ``oriented closed TQFTs''. As this is the only type of TQFT we will look at (up to Section \ref{sec:extendedTQFTs} anyway), we will just say ``TQFT''.
	
According to Definition \ref{def:closedTQFT}, a TQFT is a symmetric monoidal functor $\zz\colon \Bord_n \to \Vectk$, and we will now unpack this definition, following the same numbering (i)--(v) as in the path integral motivation at the end of Section \ref{subsec:moti}.
\begin{enumerate}
\item
The first part of data contained in a functor is that it assigns objects to objects. For $\zz$ this means that every $(n-1)$-dimensional manifold $E \in \Bord_n$ is assigned a $\Bbbk$-vector space $\zz(E)$. 

Comparing to the path integral motivation, $\zz(E)$ corresponds to the state space $\mathcal{H}_E$. 
In the present setting, however, we can work over an arbitrary field $\Bbbk$, and we do not need a scalar product on $\zz(E)$. 
As we will see in Section~\ref{subsec:basicprop} there still is a nondegenerate pairing between $\zz(E)$ and $\zz(\overline E)$.
\item
The second part of data contained in a functor is that it assigns morphisms to morphisms. That is, $\zz$ produces for every bordism $M\colon E \to E'$ a linear map 
\be 
\zz(M)\colon \zz(E) \lra \zz(E')
\ee 
which we can think of as describing the ``evolution along~$M$''. 

To relate this to the path integral motivation, think of an $n$-manifold~$M$ with boundary $\partial M = E$ as a bordism $\emptyset \xrightarrow{M} E$. 
We will learn in the next point that $\zz(\emptyset) = \Bbbk$, so that $\zz(M)\colon \Bbbk \to \zz(E)$, which (by taking the image of $1 \in \Bbbk$) is the same as giving an element in $\zz(E)$.

It may now seem that since in the path integral motivation we only consider the collection of elements in $\zz(E)$ coming from bordisms of the form $\emptyset \to E$, we have forgotten about all the linear maps $\zz(F) \to \zz(E)$ coming from bordisms $F \to E$, but this is actually not so, see Lemma~\ref{lem:sym-mon-fun_vs_path-int} below.

\item
A TQFT~$\zz$ also respects the symmetric monoidal structures of the source and target categories: 
it comes with isomorphisms 
\be 
\zz(\emptyset) \cong \Bbbk 
\, , \quad 
\zz( E \sqcup F ) \cong  \zz(E) \otimes_\Bbbk \zz(F) 
\ee 
which are compatible with associativity of the tensor products and with the braidings~$\beta$, cf.~\cite[Def.\,9\,\&\,10]{BaezEveryone}, \cite[Sect.\,2.4\,\&\,8.1]{EGNO}.
We may (and will) assume that $\zz(\emptyset)=\Bbbk$ (one may always replace $\zz$ by a monoidally equivalent functor with this property).
\item
Since $\zz$ is a functor, it maps unit morphisms to unit morphisms. That is, for an oriented $(n-1)$-manifold~$E$, we have $\zz(E \times[0,1]) = \id_{\zz(E)}$.

This looks again somewhat different from point (iv) in Section~\ref{subsec:moti}, where we could not speak about linear maps as we only obtained vectors in state spaces. Lemma~\ref{lem:sym-mon-fun_vs_path-int} will illustrate how these two descriptions of nondegeneracy are related.
\item
The most crucial property of a functor is that it is compatible with composition. For $\zz$ this means that
gluing of manifolds translates into composition of linear maps, 
\be 
\zz ( M_2 \sqcup_F M_1) = \zz(M_2) \circ \zz(M_1) \, . 
\ee 

Here, there is a qualitative difference to the description in the path integral motivation: the gluing property of functors always composes two \textsl{disjoint} $n$-manifolds, but the cutting property in the path integral motivation does not necessarily produce two disjoint pieces. It is thus in general not possible to cut a bordism $M$ along any embedded $(n-1)$-manifold into $M'$ and then use composition of bordisms to glue $M'$ back together to obtain $M$. Luckily, also this more severe looking difference is only apparent, as we will discuss in Section \ref{sec:compare}. 

Once this point is settled, the above composition property of functors translates into the sum over intermediates states in the path integral motivation.
\end{enumerate}

\subsection{Comparison to the path integral motivation}
\label{sec:compare}

We now address the apparent differences between (i)--(v) in Sections~\ref{subsec:moti} and~\ref{subsec:fun}. 
Let $\mathcal{Y}$ be the following collection of data:
\begin{itemize}
\item to each object $E \in \Bord_n$ a $\Bbbk$-vector space $\mathcal{Y}(E)$, such that $\mathcal{Y}(\emptyset) = \Bbbk$,
\item to each bordism $M\colon \emptyset \to E$ a linear map $\mathcal{Y}(M)\colon \Bbbk \to \mathcal{Y}(E)$,
\item isomorphisms $\mathcal{Y}(E \sqcup F) \to \mathcal{Y}(E) \otimes \mathcal{Y}(F)$ for all objects $E,F \in \Bord_n$. 
\end{itemize}
Note that this is the data motivated from the path integral in points (i)--(iii) in Section~\ref{subsec:moti}.

\begin{lemma}\label{lem:sym-mon-fun_vs_path-int}
The following are equivalent.
\begin{enumerate}
\item $\mathcal{Y}$ extends to a symmetric monoidal functor $\Bord_n \to \Vectk$.
\item $\mathcal{Y}$ satisfies:
\begin{enumerate}
\item 
The element $\mathcal{Y}( E \times [0,1] ) \in \mathcal{Y}(E \sqcup \overline E) \cong \mathcal{Y}(E) \otimes \mathcal{Y}(\overline E)$ is a nondegenerate copairing for all~$E$. 
Thus there is a (unique) dual pairing, which we denote by
$$
	d_E\colon \mathcal{Y}(\overline E) \otimes \mathcal{Y}(E)
	\lra \Bbbk 
	\, . 
$$ 
\item Let $U$ be a closed oriented $(n-1)$-manifold embedded in~$M$. Let~$M'$ be obtained by cutting~$M$ along $U$, i.\,e.\ $M'\colon \emptyset \to E \sqcup \overline U \sqcup U$. Then
\begin{align*}
\mathcal{Y}(M) ~=~ \Big[ \Bbbk &\xrightarrow{\mathcal{Y}(M')} \mathcal{Y}( E \sqcup \overline U \sqcup U)
\stackrel{\sim}{\lra} \mathcal{Y}( E ) \otimes \mathcal{Y}(\overline U) \otimes \mathcal{Y}(U)
\nonumber \\ &
\xrightarrow{1_{\mathcal Y(E)} \otimes d_U} \mathcal{Y}(E) \Big]
 \, .
\end{align*}
\item For bordisms $M\colon \emptyset \to E$ and $M'\colon \emptyset \to F$ we have
$$ 
	\mathcal{Y}(M) \otimes \mathcal{Y}(M')
	~=~
	\Big[ 
	\Bbbk \xrightarrow{\mathcal{Y}(M \sqcup M')}
	\mathcal{Y}(E \sqcup F)
	\stackrel{\sim}{\lra}
	\mathcal{Y}(E) \otimes \mathcal{Y}(F)
	\Big] \, .
$$ 
\item ``$\mathcal{Y}$ is compatible with the symmetric braiding and the coherence isomorphisms of $\Bord_n$ and $\Vectk$''
(see Appendix \ref{app:alt-sym-mon-fun} for the precise formulation of this condition in a more general setting).
\end{enumerate}
\end{enumerate}
If $\mathcal{Y}$ satisfies the conditions in (ii), the extension to a symmetric monoidal functor in (i) is unique.
\end{lemma}

Comparing to (i)--(v) in Section \ref{subsec:moti}, we see that a) amounts to the nondegeneracy condition (iv), b) is the sum over intermediate states in (v), c) is the condition for disjoint unions of $n$-manifolds in (iii), and d) does not appear as in Section~\ref{subsec:moti} we treated coherences as equalities and did not worry about permuting tensor factors.

\medskip

The above lemma is actually a general statement about symmetric monoidal functors between any two symmetric monoidal categories. We will give the detailed formulation and a sketch of the proof of this statement in Appendix \ref{app:alt-sym-mon-fun}.
For now we only remark that the direction (i)$\,\Rightarrow\,$(ii) is obvious, and that for the other direction the ansatz for $\zz$ is, for objects $E,F \in \Bord_n$ and a bordism $M\colon E \to F$,
\begin{align}
	\zz(E) ~=~& \mathcal{Y}(E) \, ,
	\nonumber \\
	\zz(M) 
	~=~&
	\Big[\,
	\mathcal{Y}(E)
	\xrightarrow{\mathcal{Y}(\widetilde M) \otimes \id}
	\mathcal{Y}(F \sqcup \overline E) \otimes \mathcal{Y}(E)
	\nonumber
	\\  &
	\qquad \stackrel{\sim}{\lra}
	\mathcal{Y}(F) \otimes \mathcal{Y}(\overline E) \otimes \mathcal{Y}(E)
	\xrightarrow{\id \otimes d_E}
	\mathcal{Y}(F)
	\,\Big] \, ,
\end{align}
where $\widetilde M\colon \emptyset \to F \sqcup \overline E$ is $M$, but with ingoing boundary $E$ interpreted as outgoing boundary.

\begin{remark}
The above lemma thus provides an alternative definition of a TQFT. In fact, historically it is the other way around, as the original formulation of the TQFT axioms in \cite{AtiyahTQFT} does not use symmetric monoidal functors, but instead is given in terms of the data and conditions for $\mathcal{Y}$ (except that in condition b), in \cite{AtiyahTQFT} it is assumed that  the bordism is cut into disjoint pieces).
\end{remark}

\subsection[Basic properties of $n$-dimensional TQFTs]{Basic properties of $\boldsymbol{n}$-dimensional TQFTs}
\label{subsec:basicprop}

TQFTs become rapidly more
difficult with increasing dimension. 
Nonetheless, there are some basic properties which one can easily deduce from the definition and which hold for any dimension~$n$. 
In this section we will list some of them.

\medskip

We will start with the most crucial property, namely that all state spaces of a TQFT as in Definition \ref{def:closedTQFT} are necessarily finite-dimensional (recall that the target category was that of all $\Bbbk$-vector spaces, not only finite-dimensional ones).

\begin{proposition}
\label{prop:findim}
Let $\zz\colon \Bord_n \to \Vectk$ be a TQFT. 
Then $\zz(E)$ is finite-dimensional for every $E\in \Bord_n$, and $\zz(\overline E) \cong \zz(E)^*$.
\end{proposition}

\begin{proof}
The origin of this finiteness property is duality. 
To see this, let us set $U := \zz(E)$ and $V :=  \zz(\overline E)$ for the vector space associated to the manifold with opposite orientation. 
Next we consider the cylinder $E \times [0,1]$, but viewed as a morphism
\be 
\label{eq:bendcyl1}
\begin{tikzpicture}[very thick,scale=1.0,color=blue!50!black, baseline=0.14cm]
\coordinate (p1) at (-0.55,0);
\coordinate (p2) at (-0.2,0);
\coordinate (p3) at (0.2,0);
\coordinate (p4) at (0.55,0);
%
\fill [orange!23] 
(p1) -- (p2)
-- (p2) .. controls +(0,0.35) and +(0,0.35) ..  (p3)
-- (p3) -- (p4)
-- (p4) .. controls +(0,0.9) and +(0,0.9) ..  (p1)
;
\fill [orange!38] 
(p1) .. controls +(0,-0.15) and +(0,-0.15) ..  (p2)
-- (p2) .. controls +(0,0.15) and +(0,0.15) ..  (p1)
;
\fill [orange!38] 
(p3) .. controls +(0,-0.15) and +(0,-0.15) ..  (p4)
-- (p4) .. controls +(0,0.15) and +(0,0.15) ..  (p3)
;
%
%
\draw[thick, black] (p2) .. controls +(0,0.35) and +(0,0.35) ..  (p3); 
\draw[thick, black] (p4) .. controls +(0,0.9) and +(0,0.9) ..  (p1);
\draw[very thick, red!80!black] (p1) .. controls +(0,0.15) and +(0,0.15) ..  (p2); 
\draw[very thick, red!80!black] (p1) .. controls +(0,-0.15) and +(0,-0.15) ..  (p2); 
\draw[very thick, red!80!black] (p3) .. controls +(0,0.15) and +(0,0.15) ..  (p4); 
\draw[very thick, red!80!black] (p3) .. controls +(0,-0.15) and +(0,-0.15) ..  (p4); 
%
\fill[color=black] (-0.375,0) circle (0pt) node[below] {{\footnotesize$\overline E$}};
\fill[color=black] (0.375,0) circle (0pt) node[below] {{\footnotesize$E\vphantom{\overline E}$}};
\end{tikzpicture}
: 
\overline E \sqcup E \lra \emptyset \, . 
\ee 
Similarly, we can view the cylinder as a map 
\be\label{eq:bendcyl2}
\begin{tikzpicture}[very thick,scale=1.0,color=blue!50!black, baseline=-0.22cm, rotate=180]
\coordinate (p1) at (-0.55,0);
\coordinate (p2) at (-0.2,0);
\coordinate (p3) at (0.2,0);
\coordinate (p4) at (0.55,0);
%
\fill [orange!23] 
(p1) -- (p2)
-- (p2) .. controls +(0,0.35) and +(0,0.35) ..  (p3)
-- (p3) -- (p4)
-- (p4) .. controls +(0,0.9) and +(0,0.9) ..  (p1)
;
\fill [orange!38] 
(p1) .. controls +(0,-0.15) and +(0,-0.15) ..  (p2)
-- (p2) .. controls +(0,0.15) and +(0,0.15) ..  (p1)
;
\fill [orange!38] 
(p3) .. controls +(0,-0.15) and +(0,-0.15) ..  (p4)
-- (p4) .. controls +(0,0.15) and +(0,0.15) ..  (p3)
;
%
%
\draw[thick, black] (p2) .. controls +(0,0.35) and +(0,0.35) ..  (p3); 
\draw[thick, black] (p4) .. controls +(0,0.9) and +(0,0.9) ..  (p1);
\draw[very thick, red!80!black] (p1) .. controls +(0,0.15) and +(0,0.15) ..  (p2); 
\draw[very thick, red!80!black] (p1) .. controls +(0,-0.15) and +(0,-0.15) ..  (p2); 
\draw[very thick, red!80!black] (p3) .. controls +(0,0.15) and +(0,0.15) ..  (p4); 
\draw[very thick, red!80!black] (p3) .. controls +(0,-0.15) and +(0,-0.15) ..  (p4); 
%
\fill[color=black] (-0.375,-0.68) circle (0pt) node[below] {{\footnotesize$\overline E$}};
\fill[color=black] (0.375,-0.68) circle (0pt) node[below] {{\footnotesize$E\vphantom{\overline E}$}};
\end{tikzpicture}
: \emptyset \lra E \sqcup \overline E
\, . 
\ee
By diffeomorphism invariance these two maps are related to the identity $1_E$ 
as follows: 
\be\label{eq:tubeZorro}
\begin{tikzpicture}[thick, scale=1.0, color=black, baseline=0.55cm]
\fill [orange!23] 
(-0.175,-0.2) .. controls +(0,-0.1) and +(0,-0.1) .. (0.2,-0.2)
-- (0.2,-0.2) -- (0.2,0.75)
-- (0.2,0.75).. controls +(0,0.2) and +(0,0.2) .. (0.5,0.75)
-- (0.5,0.75) -- (0.5,0.2)
-- (0.5,0.2).. controls +(0,-0.45) and +(0,-0.45) .. (1.4,0.2)
-- (1.4,0.2) -- (1.4,1.5)
-- (1.4,1.5) .. controls +(0,0.1) and +(0,0.1) .. (1,1.5)
-- (1,1.5) -- (1,0.5)
-- (1,0.5).. controls +(0,-0.2) and +(0,-0.2) .. (0.75,0.5)
-- (0.75,0.5) -- (0.75,1)
-- (0.75,1).. controls +(0,0.45) and +(0,0.45) .. (-0.175,1)
-- (-0.175,1) -- (-0.175,-0.2)
;
%
\fill [orange!38] 
(-0.175,-0.2) .. controls +(0,0.1) and +(0,0.1) ..  (0.2,-0.2) -- 
(0.2,-0.2) .. controls +(0,-0.1) and +(0,-0.1) ..  (-0.175,-0.2); 
\draw (0.2,-0.2) -- (0.2,0.75);
\draw (0.2,0.75).. controls +(0,0.2) and +(0,0.2) .. (0.5,0.75);
\draw (0.5,0.75) -- (0.5,0.2);
\draw (0.5,0.2).. controls +(0,-0.45) and +(0,-0.45) .. (1.4,0.2);
\draw (1.4,0.2) -- (1.4,1.5);
\draw (1,1.5) -- (1,0.5);
\draw (1,0.5).. controls +(0,-0.2) and +(0,-0.2) .. (0.75,0.5);
\draw (0.75,0.5) -- (0.75,1);
\draw (0.75,1).. controls +(0,0.45) and +(0,0.45) .. (-0.175,1);
\draw (-0.175,1) -- (-0.175,-0.2);
\draw[very thick, color=red!80!black] (-0.175,-0.2) .. controls +(0,0.1) and +(0,0.1) ..  (0.2,-0.2); 
\draw[very thick, color=red!80!black] (-0.175,-0.2) .. controls +(0,-0.1) and +(0,-0.1) ..  (0.2,-0.2); 
\draw[very thick, color=red!80!black] (1.4,1.5) .. controls +(0,0.1) and +(0,0.1) ..  (1,1.5); 
\draw[very thick, color=red!80!black, opacity=0.2] (1.4,1.5) .. controls +(0,-0.1) and +(0,-0.1) ..  (1,1.5); 
\end{tikzpicture}
= 
\begin{tikzpicture}[thick, scale=1.0, color=black, baseline=0.55cm]
%
\fill [orange!23] 
(0,-0.2) -- (0.35,-0.2)
-- (0.35,-0.2) -- (0.35,1.5)
-- (0.35,1.5) .. controls +(0,0.1) and +(0,0.1) .. (0,1.5)
-- (0,1.5) -- (0,-0.2)
;
\fill [orange!38] 
(0,-0.2) .. controls +(0,0.1) and +(0,0.1) ..  (0.35,-0.2) -- 
(0.35,-0.2) .. controls +(0,-0.1) and +(0,-0.1) ..  (0,-0.2); 
\draw (0.35,-0.2) -- (0.35,1.5);
\draw (0,1.5) -- (0,-0.2);
\draw[very thick, color=red!80!black] (0,-0.2) .. controls +(0,0.1) and +(0,0.1) ..  (0.35,-0.2); 
\draw[very thick, color=red!80!black] (0,-0.2) .. controls +(0,-0.1) and +(0,-0.1) ..  (0.35,-0.2); 
\draw[very thick, color=red!80!black] (0.35,1.5) .. controls +(0,0.1) and +(0,0.1) ..  (0,1.5); 
\draw[very thick, color=red!80!black, opacity=0.2] (0.35,1.5) .. controls +(0,-0.1) and +(0,-0.1) ..  (0,1.5); 
\end{tikzpicture}
\ee
Applying~$\zz$, we obtain a pairing $\langle-,-\rangle := \zz(
\begin{tikzpicture}[very thick,scale=0.45,color=blue!50!black, baseline=0.01cm]
\coordinate (p1) at (-0.55,0);
\coordinate (p2) at (-0.2,0);
\coordinate (p3) at (0.2,0);
\coordinate (p4) at (0.55,0);
%
\fill [orange!23] 
(p1) -- (p2)
-- (p2) .. controls +(0,0.35) and +(0,0.35) ..  (p3)
-- (p3) -- (p4)
-- (p4) .. controls +(0,0.9) and +(0,0.9) ..  (p1)
;
\fill [orange!38] 
(p1) .. controls +(0,-0.15) and +(0,-0.15) ..  (p2)
-- (p2) .. controls +(0,0.15) and +(0,0.15) ..  (p1)
;
\fill [orange!38] 
(p3) .. controls +(0,-0.15) and +(0,-0.15) ..  (p4)
-- (p4) .. controls +(0,0.15) and +(0,0.15) ..  (p3)
;
%
%
\draw[thick, black] (p2) .. controls +(0,0.35) and +(0,0.35) ..  (p3); 
\draw[thick, black] (p4) .. controls +(0,0.9) and +(0,0.9) ..  (p1);
\draw[very thick, red!80!black] (p1) .. controls +(0,0.15) and +(0,0.15) ..  (p2); 
\draw[very thick, red!80!black] (p1) .. controls +(0,-0.15) and +(0,-0.15) ..  (p2); 
\draw[very thick, red!80!black] (p3) .. controls +(0,0.15) and +(0,0.15) ..  (p4); 
\draw[very thick, red!80!black] (p3) .. controls +(0,-0.15) and +(0,-0.15) ..  (p4); 
\end{tikzpicture}
)\colon V \otimes_\Bbbk U \to \Bbbk$, a copairing $\gamma := \zz(
\begin{tikzpicture}[very thick,scale=0.45,color=blue!50!black, baseline=-0.22cm, rotate=180]
\coordinate (p1) at (-0.55,0);
\coordinate (p2) at (-0.2,0);
\coordinate (p3) at (0.2,0);
\coordinate (p4) at (0.55,0);
%
\fill [orange!23] 
(p1) -- (p2)
-- (p2) .. controls +(0,0.35) and +(0,0.35) ..  (p3)
-- (p3) -- (p4)
-- (p4) .. controls +(0,0.9) and +(0,0.9) ..  (p1)
;
\fill [orange!38] 
(p1) .. controls +(0,-0.15) and +(0,-0.15) ..  (p2)
-- (p2) .. controls +(0,0.15) and +(0,0.15) ..  (p1)
;
\fill [orange!38] 
(p3) .. controls +(0,-0.15) and +(0,-0.15) ..  (p4)
-- (p4) .. controls +(0,0.15) and +(0,0.15) ..  (p3)
;
%
%
\draw[thick, black] (p2) .. controls +(0,0.35) and +(0,0.35) ..  (p3); 
\draw[thick, black] (p4) .. controls +(0,0.9) and +(0,0.9) ..  (p1);
\draw[very thick, red!80!black] (p1) .. controls +(0,0.15) and +(0,0.15) ..  (p2); 
\draw[very thick, red!80!black] (p1) .. controls +(0,-0.15) and +(0,-0.15) ..  (p2); 
\draw[very thick, red!80!black] (p3) .. controls +(0,0.15) and +(0,0.15) ..  (p4); 
\draw[very thick, red!80!black] (p3) .. controls +(0,-0.15) and +(0,-0.15) ..  (p4); 
\end{tikzpicture}
)\colon \Bbbk \to U \otimes_\Bbbk V$, and~\eqref{eq:tubeZorro} translates into the identity
\be\label{eq:tubeZorroVect}
\big(  \langle-,-\rangle \otimes \id_V  \big) \circ \big(   \id_V \otimes \gamma  \big)
= \id_V
\, . 
\ee
We may choose finitely many $u_i \in U$ and $v_i \in V$ such that $\gamma(1) = \sum_i u_i \otimes v_i$ (as every element of $U�\otimes_\Bbbk V$ is of this form). 
Using this in~\eqref{eq:tubeZorroVect} we find that for every $v \in V$, 
\be 
v = \sum_i \langle v, u_i \rangle \cdot v_i
\ee 
which proves that the finite set $\{ v_i \}$ spans~$V$, so $\zz(E)$ is indeed finite-dimensional. 
Furthermore, one checks that $V \to U^*$, $v \mapsto \langle v,- \rangle$ is an isomorphism. 
\end{proof}

\begin{remark}
Proposition~\ref{prop:findim} is the main reason why TQFTs are
 comparably manageable. 
While finite-dimensional vector spaces are sufficient for the purposes of quantum computing, other realistic applications need infinite dimensions. 
There are several ways to accommodate this in the functorial approach. For example one may not allow all bordisms in the source category, or one can  endow bordisms with appropriate geometric structures such as volume dependence  (or even a metric as one would have in a non-topological QFT).
\end{remark}

As a consequence of Proposition~\ref{prop:findim}, or rather its proof, we find that taking products with a circle computes dimensions of state spaces:

\begin{corollary}
For $E$ an object in $\Bord_n$ we have $\zz(E \times S^1) = \dim_\Bbbk(\zz(E))$. 
\end{corollary}

\begin{proof}
$E \times S^1$ is diffeomorphic to composing the symmetric braiding $E \sqcup \overline E \to \overline E \sqcup E$ with the two ``bent'' cylinders in~\eqref{eq:bendcyl1} and~\eqref{eq:bendcyl2} on either side. 
Using the expression for the bent cylinders in terms of a pairing and its copairing produces the trace over the identity map on $\zz(E)$.
\end{proof}

The next result exemplifies why TQFTs in higher dimensions must be much more complicated. Namely, an $n$-dimensional TQFT can produce lots and lots of lower-dimensional TQFTs via a process called \textsl{dimensional reduction}:

\begin{proposition}
\label{prop:dimred}
Let $\zz\colon \Bord_n \to \Vectk$ be a TQFT and let $X$ be a closed, compact, oriented $r$-manifold with $r<n$. Then
\be 
	\zz^{\mathrm{red}}\colon \Bord_{n-r} \lra \Vectk
	\, , \quad
	\zz^{\mathrm{red}} \Big(E \stackrel{M}{\lra} F\Big)
	:= \zz\Big(E \times X \xrightarrow{M \times X} F \times X\Big)
\ee 
is again a TQFT.	
\end{proposition}

There is actually nothing to prove here except to quickly verify that $(-) \times X$ is a symmetric monoidal functor $\Bord_{n-r} \to \Bord_n$, so that $\zz^{\mathrm{red}}$ is just the composition of two symmetric monoidal functors.

\medskip

Finally we give a property which does not apply to all $n$-dimensional TQFTs, but only to those whose state space on the $(n-1)$-dimensional sphere is 1-dimensional. 
Three-dimensional Chern-Simons theory for a compact Lie group has this property, but the 2-dimensional examples in Section~\ref{subsec:2dTQFTs} typically do not.

Given two connected $n$-manifolds $M,N$, possibly with non-empty boundary, one can produce a new connected $n$-manifold by taking their \textsl{connected sum} $M \# N$. 
To define this in more detail, let us take $M,N$ to be connected bordisms $M\colon \emptyset \to E$, $N\colon F \to \emptyset$. 
Now cut out $n$-balls $B^n$ from~$M$ and~$N$, i.\,e.\ write $M$ and $N$ as compositions
\be
	M = \Big[ \emptyset \stackrel{B^n}{\lra} S^{n-1} \stackrel{\widetilde M}{\lra} E \Big]
	\, , \quad
	N = \Big[ F \stackrel{\widetilde N}{\lra}
	S^{n-1} \stackrel{B^n}{\lra} \emptyset
	\Big] \, .
\ee
The connected sum is now defined as
\be
	M \# N := \widetilde M \circ \widetilde N \, .
\ee
Note that as $M$ and $N$ are connected, choosing different ways to cut out the $n$-balls will give diffeomorphic manifolds after composition. 
Thus the bordism $\widetilde M \circ \widetilde N$ does not depend on this choice.

We can now state:

\begin{proposition}
\label{prop:consum}
Let $\zz\colon \Bord_n \to \Vectk$ be a TQFT with $\dim_{\Bbbk}\zz(S^{n-1}) = 1$. 
Then $\zz(S^n) \neq 0$ and for any connected $M\colon \emptyset \to E$, $N\colon F \to \emptyset$ in $\Bord_n$ we have
\be 
	\zz(M \# N) ~=~ \tfrac{1}{\zz(S^n)} \, \zz(M) \circ \zz(N)
	:~ \zz(F) \lra \zz(E) \, .
\ee 
\end{proposition}

\begin{proof}
Using the notation from the definition of the connected sum, let us abbreviate
\begin{align}
\psi &= \zz\Big(S^{n-1}\stackrel{\widetilde M}{\lra}E\Big)  \, , &
\chi &= \zz\Big(F\stackrel{\widetilde N}{\lra}S^{n-1}\Big) \, ,
\nonumber \\
v &= \zz\Big(\emptyset \stackrel{B^n}{\lra} S^{n-1}\Big) \, , &
v' &= \zz\Big(S^{n-1} \stackrel{B^n}{\lra} \emptyset\Big) \, .
\end{align}
We have pairings $(\psi,\chi) = \psi \circ \chi: \zz(F) \to \zz(E)$, $(v',v) \in \Bbbk$, etc.
Now we appeal to our assumption that $\zz(S^{n-1})$ is 1-dimensional, namely, using \cite[P.\,393]{wittenjones} 
``the wonderful fact of one dimensional linear algebra
$(\psi,\chi) \cdot (v',v) = (\psi,v) \cdot (v',\chi)$,'' we arrive at 
\be
\zz(S^n)
	\cdot	
	\zz(M \# N) ~=~  \zz(M) \circ \zz(N) \, .
\ee	

Suppose now that $\zz(S^n)$ were zero. Take $M$ to be \eqref{eq:bendcyl2} and $N$ to be \eqref{eq:bendcyl1}, both for $E = S^{n-1}$. Then the above equation forces $\zz(M) \circ \zz(N)=0$. However, we can find bordisms $X\colon E \sqcup E \sqcup \overline E \to E$ and $Y\colon E \to E \sqcup \overline E \sqcup E$ such that
\be
	X \circ (1_E \sqcup (M \circ N)) \circ Y ~=~ 1_E \, ,
\ee
where $1_E$ is the cylinder over $E=S^{n-1}$. Applying $\zz$ to the left-hand side gives zero, but applying it to the right-hand side gives the identity map on $\zz(S^{n-1})$, which is not zero as by assumption $\zz(S^{n-1})$ is 1-dimensional. This is a contradiction and hence $\zz(S^n) \neq 0$.
\end{proof}

\subsection{Comparing TQFTs via monoidal natural transformations}\label{sec:compareTQFT}

Of course, once one has seen Definition \ref{def:closedTQFT}, the answer to the question ``How should we compare two $n$-dimensional TQFTs?'' 
naturally is 
``by monoidal natural transformations'', allowing one to display one's category-savviness by not adding the qualifier ``symmetric'' in front of ``monoidal natural transformation'' (cf.~\cite[Def.\,11]{BaezEveryone} or \cite[Sect.\,2.4]{EGNO}).
	
But as we did the path integral detour before stating that definition, let us at least briefly pause and see how one would compare TQFTs from that perspective.

\medskip

So, suppose we are given two $n$-dimensional TQFTs in the sense of the data and conditions in (i)--(v) of Section \ref{subsec:moti}. 
Write $\mathcal{H}_E$ and $\mathcal{H}'_E$ for their state spaces on $(n-1)$-manifolds $E$, and $\zz(M)$ and $\zz'(M)$ for the vectors assigned to an $n$-manifold~$M$ with $\partial M = E$ by the two TQFTs. 

Let us address the slightly simpler question of what it could mean for the two TQFTs to be ``the same'' (or better: equivalent). Surely, we want the state spaces to be isomorphic. So we demand that there is, for each $(n-1)$-manifold~$E$, a linear isomorphism $\chi_E\colon \mathcal{H}_E \to \mathcal{H}'_E$. 
There is an obvious compatibility condition with $\zz$ and $\zz'$. 
Namely, for~$M$ with $\partial M = E$ we require $\zz'(M) = \chi_E(\zz(M))$. 

But what if $M = M_1 \sqcup M_2$, with $\partial M_1 = E_1$, $\partial M_2 = E_2$? 
Then $\mathcal{H}_E = \mathcal{H}_{E_1} \otimes \mathcal{H}_{E_2}$, and dito for $\mathcal{H}'_E$. 
But now we have two choices for the isomorphism, we could take $\chi_E$, or we could take $\chi_{E_1} \otimes \chi_{E_2}$. 
However, we are not free to choose due to our compatibility condition with $\zz$,
\begin{align}
	\chi_E(\zz(M))
	&=
	\zz'(M) 
	= 
	\zz'(M_1 \sqcup M_2)
	= 
	\zz'(M_1) \otimes \zz'(M_2)
	\nonumber \\
	&=
	\chi_{E_1}(\zz(M_1)) \otimes \chi_{E_2}(\zz(M_2)) \, ,
\end{align}
that is, $\chi_E(v) = (\chi_{E_1} \otimes \chi_{E_2})(v)$ for all $v$ of the form $\zz(M_1) \otimes \zz(M_2)$. 
The easiest way to satisfy this clearly is to initially define $\chi_E$ only on connected $(n-1)$-manifolds and then extend it to all $(n-1)$-manifolds by taking tensor products,
\be\label{eq:chi-nonconnected-mfld}
	\chi_{E_1 \sqcup \cdots \sqcup E_k} :=
	\chi_{E_1} \otimes \cdots \otimes \chi_{E_k}
\ee
for $E_1,\dots,E_k$ connected.

\medskip

Using the formalisation of conditions (i)--(v) for $\mathcal{H}$ and $\zz$ from Section \ref{subsec:moti} in terms of $\mathcal{Y}$ as given in Section \ref{sec:compare}, we arrive at the following notion of equivalence:

\begin{definition}
Two sets of data $\mathcal{Y}$ and $\mathcal{Y}'$ as in Section \ref{sec:compare} and satisfying the conditions of Lemma \ref{lem:sym-mon-fun_vs_path-int}\,(ii) are \textsl{equivalent} if there is a family of isomorphisms $\chi_E\colon \mathcal{Y}(E) \to \mathcal{Y}'(E)$, where $E \in \Bord_n$ is connected, such that for all~$E$ and $M\colon \emptyset \to E$ we have 
\be
	\chi_E(\mathcal{Y}(M))
	=
	\mathcal{Y}'(M) \, .
\ee
\end{definition}

\begin{lemma}
Let $\mathcal{Y}$ and $\mathcal{Y}'$ satisfy the conditions of Lemma \ref{lem:sym-mon-fun_vs_path-int}\,(ii), and let $\chi_E\colon \mathcal{Y}(E) \to \mathcal{Y}'(E)$ be a family of isomorphisms for connected~$E$. The following are equivalent.
\begin{enumerate}
\item
$\chi_E$ extends to a natural monoidal isomorphism between the symmetric monoidal functors corresponding to $\mathcal{Y}$ and $\mathcal{Y}'$ via Lemma \ref{lem:sym-mon-fun_vs_path-int}.
\item
$\mathcal{Y}$ and $\mathcal{Y}'$ are equivalent via $\chi_E$.
\end{enumerate}
If the $\chi_E$ satisfy the condition in (ii), the extension to a monoidal natural isomorphism in (i) is unique.
\end{lemma}

As for the relation between $\mathcal{Y}$ and symmetric monoidal functors, this lemma is best proved in its general categorical context, which we do in Lemma \ref{lem:sym-mon-iso-via-Y}. We note that condition b) in Lemma \ref{lem:sym-mon-iso-via-Y} is automatic from the definition in \eqref{eq:chi-nonconnected-mfld} of $\chi$ for non-connected manifolds.

\medskip

At this point we feel confident that in terms of Definition~\ref{def:closedTQFT}, a natural monoidal isomorphism is the correct notion of equivalence between TQFTs. 
But what about other natural monoidal transformations? 
Here we meet what may qualify 
as a little surprise when analysing ways to compare TQFTs: 

\begin{lemma}
Let $\zz,\zz'$ be $n$-dimensional TQFTs and let 
$\eta\colon \zz \to \zz'$ be a natural monoidal transformation. Then $\eta$ is an isomorphism. 
\end{lemma}

This is actually a general fact about monoidal natural transformations between monoidal functors whose source category has duals, and we recall the argument in Appendix \ref{app:mon-nat-x-groupoid}.
The duality morphisms in $\Bord_n$ are those in \eqref{eq:bendcyl1} and \eqref{eq:bendcyl2} -- for more on dualities in monoidal categories see \cite[Sect.\,2.10]{EGNO}.

\medskip

All in all we have learned that $n$-dimensional TQFTs form a groupoid.

\section{Algebraic description of TQFTs in dimensions 1,2,3}

Apart from studying concrete examples of TQFTs, it is an important question to what extent one can control all TQFTs of a given dimension. One might formulate this goal as ``classification of TQFTs'', but in a sense this is a bad term to use. For, if someone hands you a paper saying it contains the classification of TQFTs in some dimension $n$, you might hope for some sort of list, e.\,g.\ one that says that there are such and such infinite families parametrised by this and that set, plus a bunch of exceptionals. 
However, for TQFTs beyond dimension $n=1$ this is impossible, much in the same way that there cannot be a list of ``all finite groups'' (but there is a list of all simple finite groups).

Instead, a fruitful question has been:
\begin{quote}
Can one describe a TQFT in terms of a finite amount of data which is subject to a finite number of conditions?
\end{quote}
The answer in this case turns out to be, at least in the way we described TQFTs up to now, ``yes'' in dimensions 1 and 2, and ``no'' from dimension three onwards, for a simple reason we will get to in Section~\ref{sec:extendedTQFTs}.

\medskip

The roadmap for this section is as follows: 
In Section~\ref{sec:one-dim-class} we discuss at length how 1-dimensional TQFTs are basically finite-dimensional vector spaces -- only to rephrase and refine this discussion in Section~\ref{sec:gen+rel_def}, where we introduce the notion of freely generated symmetric monoidal categories. 
Section~\ref{subsec:2dTQFTs} then examines 2-dimensional TQFTs through this lens, finding how such TQFTs are in one-to-one correspondence to commutative Frobenius algebras. 
A similarly exhaustive characterisation in three dimensions is much harder, and in Sections~\ref{sec:extendedTQFTs} and~\ref{subsec:dimred3to2} we only offer a brief tour to some of the highlights of 3-dimensional ``extended'' TQFTs.

\subsection{One-dimensional TQFTs}\label{sec:one-dim-class}

We start with the case of a 1-dimensional TQFT
\be 
\zz\colon \Bord_1 \lra \Vectk \, . 
\ee 
We first note that $\Bord_1$ is ``tensor-generated'' by just two objects: the positively and negatively oriented points~$\bullet_+$ and~$\bullet_-$, respectively. 
This is simply another way of saying that every 0-dimensional compact oriented closed manifold is a disjoint union of finitely many (including zero!) copies of~$\bullet_+$ and~$\bullet_-$. 
Hence the objects of $\Bord_1$ look like 
\be 
\emptyset
\, , \quad
\bullet_+
\, , \quad
\bullet_- \sqcup \bullet_-
\, , \quad
\bullet_+ \sqcup \bullet_+ \sqcup \bullet_- \sqcup \dots \sqcup \bullet_+
\, .  
\ee 

Morphisms in $\Bord_1$ are diffeomorphism classes of oriented lines connecting such points. 
For example, one particular morphism from $\bullet_+ \sqcup \bullet_- \sqcup \bullet_+ \sqcup \bullet_-$ to $\bullet_+ \sqcup \bullet_-$ is represented by
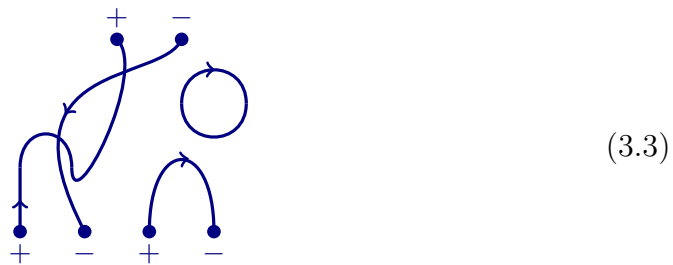
\be\label{eq:lineex}
\begin{tikzpicture}[very thick,scale=0.85,color=blue!50!black, baseline=1cm]
\coordinate (b1) at (0,0);
\coordinate (b2) at (1,0);
\coordinate (b3) at (2,0);
\coordinate (b4) at (3,0);
\coordinate (t1) at (1.5,3);
\coordinate (t2) at (2.5,3);
\draw[
	decoration={markings, mark=at position 0.5 with {\arrow{>}}}, postaction={decorate}
	]
 (b1) -- (0,1); 
\draw (0,1) .. controls +(0,0.7) and +(0,0.7) .. (0.8,1);
\draw (0.8,1) .. controls +(0,-0.9) and +(0.5,-0.5) .. (t1);
\draw[
	decoration={markings, mark=at position 0.5 with {\arrow{<}}}, postaction={decorate}
	]
 (b2) .. controls +(-1.4,2.7) and +(-0.3,-0.7) .. (t2); 
\draw[redirected] (b4) .. controls +(0,1.5) and +(0,1.5) .. (b3);
\draw[directed] (2.5,2) .. controls +(0,0.7) and +(0,0.7) .. (3.5,2);
\draw (3.5,2) .. controls +(0,-0.7) and +(0,-0.7) .. (2.5,2);
\fill (b1) circle (3.0pt) node[below] {{\small $+$}};
\fill (b2) circle (3.0pt) node[below] {{\small $-$}};
\fill (b3) circle (3.0pt) node[below] {{\small $+$}};
\fill (b4) circle (3.0pt) node[below] {{\small $-$}};
\fill (t1) circle (3.0pt) node[above] {{\small $+$}};
\fill (t2) circle (3.0pt) node[above] {{\small $-$}};
\end{tikzpicture}
\ee
It is an unsurprising fact (which can be proven using Morse theory) that every morphism in $\Bord_1$ can be built by composing (both via gluing and taking disjoint unions) the identity $1_{\bullet_+} = 
\begin{tikzpicture}[very thick,scale=0.2,color=blue!50!black, baseline=-0.1cm]
\draw[
	decoration={markings, mark=at position 0.5 with {\arrow{>}}}, postaction={decorate}
	]
(0,-1.25) -- (0,1.25); 
\end{tikzpicture}
\,
$ 
and the \textsl{generators}
\be\label{eq:1dgens}
\begin{tikzpicture}[very thick,scale=1.0,color=blue!50!black, baseline=.4cm]
\draw[line width=0pt] 
(3,0) node[line width=0pt] (D) {}
(2,0) node[line width=0pt] (s) {}; 
\draw[directed] (D) .. controls +(0,1) and +(0,1) .. (s);
\end{tikzpicture}
\, , \quad
\begin{tikzpicture}[very thick,scale=1.0,color=blue!50!black, baseline=-.4cm,rotate=180]
\draw[line width=0pt] 
(3,0) node[line width=0pt] (D) {}
(2,0) node[line width=0pt] (s) {}; 
\draw[redirected] (D) .. controls +(0,1) and +(0,1) .. (s);
\end{tikzpicture}
\, , \quad
\begin{tikzpicture}[very thick,scale=1.0,color=blue!50!black, baseline=.4cm]
\draw[line width=0pt] 
(3,0) node[line width=0pt] (D) {}
(2,0) node[line width=0pt] (s) {}; 
\draw[redirected] (D) .. controls +(0,1) and +(0,1) .. (s);
\end{tikzpicture}
\, , \quad
\begin{tikzpicture}[very thick,scale=1.0,color=blue!50!black, baseline=-.4cm,rotate=180]
\draw[line width=0pt] 
(3,0) node[line width=0pt] (D) {}
(2,0) node[line width=0pt] (s) {}; 
\draw[directed] (D) .. controls +(0,1) and +(0,1) .. (s);
\end{tikzpicture}
\, , \quad
\begin{tikzpicture}[very thick,scale=0.85,color=blue!50!black, baseline=0cm]
\draw[line width=0] 
(1,0.7) node[line width=0pt] (A) {}
(0,-0.7) node[line width=0pt] (A2) {}; 
\draw[line width=0] 
(0,0.7) node[line width=0pt] (B) {}
(1,-0.7) node[line width=0pt] (B2) {}; 
\draw[
	decoration={markings, mark=at position 0.3 with {\arrow{>}}}, postaction={decorate}
	]
 (A2) -- (A); 
\draw[
	decoration={markings, mark=at position 0.3 with {\arrow{>}}}, postaction={decorate}
	]
 (B2) -- (B); 
\end{tikzpicture}
\ee
where the last diagram represents the braiding bordism $\beta_{\bullet_+,\bullet_+}$ of~\eqref{eq:Bordbraid}. 
Note that typically we do not show oriented endpoints in such diagrams. 
The generators~\eqref{eq:1dgens} are subject to the \textsl{relations}
\be\label{eq:1rel}
\tikzzbox{
\begin{tikzpicture}[very thick,scale=0.85,color=blue!50!black, baseline=0cm]
\draw[line width=0] 
(-1,1.25) node[line width=0pt] (A) {}
(1,-1.25) node[line width=0pt] (A2) {}; 
\draw[directed] (0,0) .. controls +(0,-1) and +(0,-1) .. (-1,0);
\draw[directed] (1,0) .. controls +(0,1) and +(0,1) .. (0,0);
\draw (-1,0) -- (A); 
\draw (1,0) -- (A2); 
\end{tikzpicture}
}
=
\tikzzbox{
\begin{tikzpicture}[very thick,scale=0.85,color=blue!50!black, baseline=0cm]
\draw[line width=0] 
(0,1.25) node[line width=0pt] (A) {}
(0,-1.25) node[line width=0pt] (A2) {}; 
\draw[
	decoration={markings, mark=at position 0.5 with {\arrow{>}}}, postaction={decorate}
	]
 (A2) -- (A); 
\end{tikzpicture}
}
= 
\begin{tikzpicture}[very thick,scale=0.85,color=blue!50!black, baseline=0cm]
\draw[line width=0] 
(1,1.25) node[line width=0pt] (A) {}
(-1,-1.25) node[line width=0pt] (A2) {}; 
\draw[redirected] (0,0) .. controls +(0,1) and +(0,1) .. (-1,0);
\draw[redirected] (1,0) .. controls +(0,-1) and +(0,-1) .. (0,0);
\draw (-1,0) -- (A2); 
\draw (1,0) -- (A); 
\end{tikzpicture}
\, , \quad
\tikzzbox{
\begin{tikzpicture}[very thick,scale=0.85,color=blue!50!black, baseline=0cm]
\draw[line width=0] 
(1,1.25) node[line width=0pt] (A) {}
(-1,-1.25) node[line width=0pt] (A2) {}; 
\draw[directed] (0,0) .. controls +(0,1) and +(0,1) .. (-1,0);
\draw[directed] (1,0) .. controls +(0,-1) and +(0,-1) .. (0,0);
\draw (-1,0) -- (A2); 
\draw (1,0) -- (A); 
\end{tikzpicture}
}
=
\tikzzbox{
\begin{tikzpicture}[very thick,scale=0.85,color=blue!50!black, baseline=0cm]
\draw[line width=0] 
(0,1.25) node[line width=0pt] (A) {}
(0,-1.25) node[line width=0pt] (A2) {}; 
\draw[
	decoration={markings, mark=at position 0.5 with {\arrow{>}}}, postaction={decorate}
	]
(A) -- (A2); 
\end{tikzpicture}
}
= 
\begin{tikzpicture}[very thick,scale=0.85,color=blue!50!black, baseline=0cm]
\draw[line width=0] 
(-1,1.25) node[line width=0pt] (A) {}
(1,-1.25) node[line width=0pt] (A2) {}; 
\draw[redirected] (0,0) .. controls +(0,-1) and +(0,-1) .. (-1,0);
\draw[redirected] (1,0) .. controls +(0,1) and +(0,1) .. (0,0);
\draw (-1,0) -- (A); 
\draw (1,0) -- (A2); 
\end{tikzpicture}
\, ,
\ee
as well as further relations involving the braiding, namely the hexagon equations and a relation that expresses the naturality of $\beta_{\bullet_+,\bullet_+}$. 
For example, these relations tell us that the morphism~\eqref{eq:lineex} is equivalently represented by the bordism
\be 
\begin{tikzpicture}[very thick,scale=0.85,color=blue!50!black, baseline=1.22cm]
\coordinate (b1) at (0,0);
\coordinate (b2) at (1,0);
\coordinate (b3) at (2,0);
\coordinate (b4) at (3,0);
\coordinate (t1) at (0,3);
\coordinate (t2) at (1,3);
\draw[
	decoration={markings, mark=at position 0.5 with {\arrow{<}}}, postaction={decorate}
	]
 (t1) -- (b1); 
\draw[
	decoration={markings, mark=at position 0.5 with {\arrow{<}}}, postaction={decorate}
	]
 (b2) -- (t2); 
\draw[redirected] (b4) .. controls +(0,1) and +(0,1) .. (b3);
\draw[directed] (2,2) .. controls +(0,0.7) and +(0,0.7) .. (3,2);
\draw (3,2) .. controls +(0,-0.7) and +(0,-0.7) .. (2,2);
\fill (b1) circle (3.0pt) node[below] {{\small $+$}};
\fill (b2) circle (3.0pt) node[below] {{\small $-$}};
\fill (b3) circle (3.0pt) node[below] {{\small $+$}};
\fill (b4) circle (3.0pt) node[below] {{\small $-$}};
\fill (t1) circle (3.0pt) node[above] {{\small $+$}};
\fill (t2) circle (3.0pt) node[above] {{\small $-$}};
\end{tikzpicture}
\, . 
\ee 

Since with the above generators and relations $\Bord_1$ is completely under control, one can give a preliminary ``classification of'' 1-dimensional TQFTs: 

\begin{theorem}
\label{thm:1dTQFTprelim}
There is a 1-to-1 correspondence between 1-dimensional TQFTs $\zz\colon \Bord_1 \to \Vectk$ and finite-dimensional vector spaces, given by $\zz \mapsto \zz(\bullet_+)$. 
\end{theorem}

It follows from Proposition~\ref{prop:findim} that $\zz(\bullet_+)$ is indeed finite-dimensional. 
Conversely, for every finite-dimensional vector space~$V$ we construct a symmetric monoidal functor $\zz\colon \Bord_1 \to \Vectk$ as follows: set 
$\zz(\bullet_+) = V$ and $\zz(\bullet_-) = V^*$, and more generally 
\be 
\zz\big( \bullet_+^{\sqcup m} \sqcup \bullet_-^{\sqcup n} \big) = V^{\otimes m} \otimes_\Bbbk (V^*)^{\otimes n} \, . 
\ee 
To define~$\zz$ on generators we pick a basis $\{ e_i \}$ of~$V$ (with dual basis $\{ e_i^* \}$) and set 
\begin{align}
& 
\zz\Big(
\begin{tikzpicture}[very thick,scale=1.0,color=blue!50!black, baseline=.3cm]
\draw[line width=0pt] 
(3,0) node[line width=0pt] (D) {}
(2,0) node[line width=0pt] (s) {}; 
\draw[directed] (D) .. controls +(0,1) and +(0,1) .. (s);
\end{tikzpicture}
\Big)\colon V^* \otimes_\Bbbk V \lra \Bbbk
\, , \quad
\varphi \otimes v \lmt \varphi (v)
\, , 
\nonumber\\
&
\zz\Big(
\begin{tikzpicture}[very thick,scale=1.0,color=blue!50!black, baseline=-.5cm,rotate=180]
\draw[line width=0pt] 
(3,0) node[line width=0pt] (D) {}
(2,0) node[line width=0pt] (s) {}; 
\draw[redirected] (D) .. controls +(0,1) and +(0,1) .. (s);
\end{tikzpicture}
\Big)\colon \Bbbk \lra V \otimes_\Bbbk V^*
\, , \quad
\lambda \lmt \sum_i \lambda \cdot e_i \otimes e_i^*
\, , 
\nonumber\\
& 
\zz\Big(
\begin{tikzpicture}[very thick,scale=1.0,color=blue!50!black, baseline=.3cm]
\draw[line width=0pt] 
(3,0) node[line width=0pt] (D) {}
(2,0) node[line width=0pt] (s) {}; 
\draw[redirected] (D) .. controls +(0,1) and +(0,1) .. (s);
\end{tikzpicture}
\Big)\colon V \otimes_\Bbbk V^* \lra \Bbbk
\, , \quad
v \otimes \varphi \lmt \varphi (v)
\, , 
\nonumber\\
&
\zz\Big(
\begin{tikzpicture}[very thick,scale=1.0,color=blue!50!black, baseline=-.5cm,rotate=180]
\draw[line width=0pt] 
(3,0) node[line width=0pt] (D) {}
(2,0) node[line width=0pt] (s) {}; 
\draw[directed] (D) .. controls +(0,1) and +(0,1) .. (s);
\end{tikzpicture}
\Big)\colon \Bbbk \lra V^* \otimes_\Bbbk V
\, , \quad
\lambda \lmt \sum_i \lambda \cdot e_i^* \otimes e_i
\, , 
\nonumber\\
&
\zz\Big(
\begin{tikzpicture}[very thick,scale=0.85,color=blue!50!black, baseline=-0.1cm]
\draw[line width=0] 
(1,0.7) node[line width=0pt] (A) {}
(0,-0.7) node[line width=0pt] (A2) {}; 
\draw[line width=0] 
(0,0.7) node[line width=0pt] (B) {}
(1,-0.7) node[line width=0pt] (B2) {}; 
\draw[
	decoration={markings, mark=at position 0.3 with {\arrow{>}}}, postaction={decorate}
	]
 (A2) -- (A); 
\draw[
	decoration={markings, mark=at position 0.3 with {\arrow{>}}}, postaction={decorate}
	]
 (B2) -- (B); 
\end{tikzpicture}
\Big)\colon V \otimes_\Bbbk V \lra V \otimes_\Bbbk V
\, , \quad
u \otimes v \lmt v \otimes u 
\, . 
\end{align}
It is straightforward to verify that~$\zz$ really is a symmetric monoidal functor, i.\,e.~that it respects the relations~\eqref{eq:1rel}: for the second identity, say, we compute 
\begin{align}
\zz\Big(
\begin{tikzpicture}[very thick,scale=0.55,color=blue!50!black, baseline=-0.1cm]
\draw[line width=0] 
(1,1.25) node[line width=0pt] (A) {}
(-1,-1.25) node[line width=0pt] (A2) {}; 
\draw[redirected] (0,0) .. controls +(0,1) and +(0,1) .. (-1,0);
\draw[redirected] (1,0) .. controls +(0,-1) and +(0,-1) .. (0,0);
\draw (-1,0) -- (A2); 
\draw (1,0) -- (A); 
\end{tikzpicture}
\Big)
& = \Big( v \lmt \big( \zz\Big(
\begin{tikzpicture}[very thick,scale=1.0,color=blue!50!black, baseline=.3cm]
\draw[line width=0pt] 
(3,0) node[line width=0pt] (D) {}
(2,0) node[line width=0pt] (s) {}; 
\draw[redirected] (D) .. controls +(0,1) and +(0,1) .. (s);
\end{tikzpicture}
\Big)
\otimes \id \big) \circ \big( \id \otimes \zz\Big(
\begin{tikzpicture}[very thick,scale=1.0,color=blue!50!black, baseline=-.5cm,rotate=180]
\draw[line width=0pt] 
(3,0) node[line width=0pt] (D) {}
(2,0) node[line width=0pt] (s) {}; 
\draw[directed] (D) .. controls +(0,1) and +(0,1) .. (s);
\end{tikzpicture}
\Big)
\big)
(v \otimes 1)
\Big)
\nonumber\\
& = \Big( v \lmt \big( \zz\Big(
\begin{tikzpicture}[very thick,scale=1.0,color=blue!50!black, baseline=.3cm]
\draw[line width=0pt] 
(3,0) node[line width=0pt] (D) {}
(2,0) node[line width=0pt] (s) {}; 
\draw[redirected] (D) .. controls +(0,1) and +(0,1) .. (s);
\end{tikzpicture}
\Big)
\otimes \id \big) \circ \big( v \otimes \sum_i e^*_i \otimes e_i \big)  \Big)
\nonumber\\
& = \Big( v \lmt \sum_i e_i^*(v) \cdot e_i \Big)
\nonumber\\
& = 
\zz\Big(
\begin{tikzpicture}[very thick,scale=0.55,color=blue!50!black, baseline=-0.1cm]
\draw[line width=0] 
(0,1) node[line width=0pt] (A) {}
(0,-1.0) node[line width=0pt] (A2) {}; 
\draw[
	decoration={markings, mark=at position 0.5 with {\arrow{>}}}, postaction={decorate}
	]
 (A2) -- (A); 
\end{tikzpicture}
\Big)
\, , 
\end{align}
and the other relations are checked similarly. 
This establishes the 1-to-1 correspondence of Theorem~\ref{thm:1dTQFTprelim}. 

The upshot so far is that 1-dimensional TQFTs are boring: finite-dimensional vector spaces with no further structure. 
So they are basically natural numbers. 

\medskip

The reason why the above way of stating Theorem \ref{thm:1dTQFTprelim} is preliminary is that ``1-to-1 correspondence'' is not really a term one should use when comparing categories. Indeed, we have learned in Section \ref{sec:compareTQFT} that TQFTs of a given dimension form a groupoid. So really we would like to have a statement like this:
$$
\begin{tikzpicture}[
			     baseline=(current bounding box.base), 
			     >=stealth,
			     descr/.style={fill=white,inner sep=4.5pt}, 
			     normal line/.style={->}
			     ] 
\matrix (m) [matrix of math nodes, row sep=3.5em, column sep=4.5em, text height=1.1ex, text depth=0.1ex] {%
\Big(
\begin{minipage}{10em}
\begin{center}
groupoid of\\ 
$n$-dimensional TQFTs
\end{center}
\end{minipage}
\Big)
& 
\Big(\begin{minipage}{13em}
\begin{center}
some algebraic structure\\
which also forms a groupoid
\end{center}
\end{minipage}
\Big)
\\
};
\path[font=\footnotesize] (m-1-1) edge[->] node[auto] {functorial} (m-1-2);
\path[font=\footnotesize] (m-1-1) edge[->] node[below] {equivalence} (m-1-2);
\end{tikzpicture}
$$ 
The algebraic structure suggested by our preliminary theorem is finite-dimensional vector spaces. However, the natural way of comparing vector spaces are linear maps, and these do not form a groupoid. Of course, one can just throw out all non-invertible linear maps, but the structural purist 
will insist that firstly, one would then have failed to identify the correct algebraic structure, and secondly, there is no inherent reason to give preference to $\bullet_+$ over $\bullet_-$. 

So, let us instead describe the category $\mathcal{DP}_\Bbbk$ of \textsl{dual pairs}: 
\begin{itemize}
\item \textsl{Objects:} tuples $(U,V,b,d)$, where $U,V$ are $\Bbbk$-vector spaces and $b\colon \Bbbk \to U \otimes V$ and $d\colon V \otimes U \to \Bbbk$ (``birth'' and ``death'') are two linear maps which satisfy the 
	\textsl{Zorro moves} 
	$
	(d\otimes \id_V) \circ (\id_V \otimes b) = \id_V
	$
	and 
	$
	(\id_U\otimes d) \circ (b \otimes \id_U) = \id_U
	$. 
That is, the vector spaces~$U$ and~$V$ are dual to each other, and the duality is exhibited by $b,d$.
\item \textsl{Morphisms:} A morphism $(U,V,b,d) \to (U',V',b',d')$ is a pair $(f,g)$ of linear maps, where $f\colon U \to U'$ and $g\colon V \to V'$ and
\be 
	d 
	= 
	d' \circ (g \otimes f) 
	\, , \quad
	b \circ (f \otimes g)
	=
	b'
 \, .
\ee 
\end{itemize}
There are two important points to note about  $\mathcal{DP}_\Bbbk$. 
Firstly, the vector spaces $U,V$ in a dual pair are necessarily finite-dimensional and of the same dimension (cf.\ Proposition \ref{prop:findim}). 
Secondly, morphisms between dual pairs are automatically invertible. The argument is essentially the same as the one for monoidal natural transformations given in Appendix \ref{app:mon-nat-x-groupoid}. Even better, it is a direct consequence of Lemma \ref{lem:sym-mon-inv} as we will point out in the next section.

Having said all this, let us state, without proof, the final version of Theorem~\ref{thm:1dTQFTprelim}:

\begin{theorem}
\label{thm:1dTQFT}
The functor 
$
{
\zz \mapsto \big( \zz(\bullet_+) , \zz(\bullet_-) , \zz(
\begin{tikzpicture}[very thick, scale=0.39, color=blue!50!black, baseline=-0.25cm, rotate=180, >=stealth]
\draw[-{stealth[scale width=1.1]}] 
(2,0) .. controls +(0,1) and +(0,1) .. (3,0);
\end{tikzpicture} 
),
\zz(\,
\begin{tikzpicture}[very thick, scale=0.39, color=blue!50!black, baseline=0.05cm, >=stealth]
\draw[-{stealth[scale width=1.1]}] 
(3,0) .. controls +(0,1) and +(0,1) .. (2,0);
\end{tikzpicture} 
) 
 \big)
 }
$ 
is an equivalence of groupoids between 1-dimensional TQFTs and $\mathcal{DP}_\Bbbk$.
\end{theorem}

\subsection{Interlude: Generators and relations}
\label{sec:gen+rel_def}

Before we continue to two dimensions, we spend a little effort to make precise the phrase ``freely generated as a symmetric monoidal category by (something)''. 
As our run-along example we take the 1-dimensional TQFTs from above. 
In particular, we will see 
that $\Bord_1$ is freely generated as a symmetric monoidal category  by the objects 
\be\label{eq:bord1-G0}
\bullet_+ \, , \quad \bullet_-
\ee
and the morphisms
\be\label{eq:genrelex-mor}
\begin{tikzpicture}[very thick,scale=1.0,color=blue!50!black, baseline=.4cm]
\draw[line width=0pt] 
(3,0) node[line width=0pt] (D) {}
(2,0) node[line width=0pt] (s) {}; 
\draw[directed] (D) .. controls +(0,1) and +(0,1) .. (s);
\end{tikzpicture}
\, , \quad
\begin{tikzpicture}[very thick,scale=1.0,color=blue!50!black, baseline=-.5cm,rotate=180]
\draw[line width=0pt] 
(3,0) node[line width=0pt] (D) {}
(2,0) node[line width=0pt] (s) {}; 
\draw[redirected] (D) .. controls +(0,1) and +(0,1) .. (s);
\end{tikzpicture}
\ee
subject to the relations
\be\label{eq:genrelex-rel}
\tikzzbox{
\begin{tikzpicture}[very thick,scale=0.85,color=blue!50!black, baseline=0cm]
\draw[line width=0] 
(-1,1.25) node[line width=0pt] (A) {}
(1,-1.25) node[line width=0pt] (A2) {}; 
\draw[directed] (0,0) .. controls +(0,-1) and +(0,-1) .. (-1,0);
\draw[directed] (1,0) .. controls +(0,1) and +(0,1) .. (0,0);
\draw (-1,0) -- (A); 
\draw (1,0) -- (A2); 
\end{tikzpicture}
}
=
\tikzzbox{
\begin{tikzpicture}[very thick,scale=0.85,color=blue!50!black, baseline=0cm]
\draw[line width=0] 
(0,1.25) node[line width=0pt] (A) {}
(0,-1.25) node[line width=0pt] (A2) {}; 
\draw[
	decoration={markings, mark=at position 0.5 with {\arrow{>}}}, postaction={decorate}
	]
 (A2) -- (A); 
\end{tikzpicture}
}
\, , \quad 
\tikzzbox{
\begin{tikzpicture}[very thick,scale=0.85,color=blue!50!black, baseline=0cm]
\draw[line width=0] 
(1,1.25) node[line width=0pt] (A) {}
(-1,-1.25) node[line width=0pt] (A2) {}; 
\draw[directed] (0,0) .. controls +(0,1) and +(0,1) .. (-1,0);
\draw[directed] (1,0) .. controls +(0,-1) and +(0,-1) .. (0,0);
\draw (-1,0) -- (A2); 
\draw (1,0) -- (A); 
\end{tikzpicture}
}
=
\tikzzbox{
\begin{tikzpicture}[very thick,scale=0.85,color=blue!50!black, baseline=0cm]
\draw[line width=0] 
(0,1.25) node[line width=0pt] (A) {}
(0,-1.25) node[line width=0pt] (A2) {}; 
\draw[
	decoration={markings, mark=at position 0.5 with {\arrow{>}}}, postaction={decorate}
	]
(A) -- (A2); 
\end{tikzpicture}
}
\, . 
\ee

A good way to look at ``free somethings'' is via a universal property they satisfy (structural purist: ``Just say they are left adjoint to a forgetful functor'').

We should warn the reader that this section is the most dry and technical of these notes. Please feel free to just read the quick summary below and then jump ahead to Section \ref{subsec:2dTQFTs}.

\subsubsection*{A quick summary of the universal property}

Given a set of objects $G_0$, a set of morphisms $G_1$, and a set of relations $G_2$, e.\,g.\ as in \eqref{eq:bord1-G0}--\eqref{eq:genrelex-rel}, 
the outcome of the three-step construction below will be a symmetric monoidal category $\mathcal{F}(G_0,G_1,G_2)$ with the following universal property: 
\begin{quote}
Let $\mathcal{C}$ be a symmetric monoidal category. A symmetric monoidal functor $\mathcal{F}(G_0,G_1,G_2) \to \mathcal{C}$ is characterised uniquely (up to monoidal isomorphism) by 
choosing an object in $\mathcal{C}$ for each element in $G_0$ and a morphism  in $\mathcal{C}$ for each element in $G_1$ (with correct source and target) such that the relations posited in $G_2$ are satisfied.
\end{quote}
We will refer to $\mathcal{F}(G_0,G_1,G_2)$ as the \textsl{symmetric monoidal category freely generated by $G_0,G_1,G_2$}. It is unique up to natural monoidal isomorphism.

This concept becomes useful to TQFT if one can verify that a bordism category fulfills the above universal property for some $G_0,G_1,G_2$, i.\,e.\ if one may set $\Bord_n = \mathcal{F}(G_0,G_1,G_2)$. In this case one may specify symmetric monoidal functors $\Bord_n  \to \Vectk$ in the explicit way described above.

Indeed -- as we will see in more detail below -- Section \ref{sec:one-dim-class} shows that for $G_0,G_1,G_2$ as in \eqref{eq:bord1-G0}--\eqref{eq:genrelex-rel} we have $\Bord_1 = \mathcal{F}(G_0,G_1,G_2)$. A similar presentation with finite sets $G_0,G_1,G_2$ is possible for $\Bord_2$ (cf.~Section \ref{subsec:2dTQFTs}), but for $\Bord_n$ with $n >2$ these sets become infinite.

\subsubsection*{Step 1: Objects only}

We start with the free symmetric monoidal category $\mathcal{F}(G_0)$ whose objects are freely generated by a set $G_0$. 
We will think of $G_0$ interchangeably as a set or as a category whose objects are $G_0$ and which only has identity morphisms. 
Accordingly, we sometimes speak of functions out of $G_0$ and sometimes of functors.
	
Let $\mathcal{C}$ be an arbitrary symmetric monoidal category. 
Denote by $F^{G_0}(\mathcal{C})$ the category of functors from $G_0$ to $\mathcal{C}$. Explicitly, a functor $\Phi \in F^{G_0}(\mathcal{C})$ is just a function $G_0 \to \Obj(\mathcal{C})$ since for morphisms there is no freedom -- identity morphisms must be mapped to identity morphisms.

Given another symmetric monoidal category $\mathcal{C}'$, we can look at  symmetric monoidal functors $\Psi\colon \mathcal{C} \to \mathcal{C}'$ that are compatible with a choice of $\Phi \in F^{G_0}(\mathcal{C})$ and $\Phi' \in F^{G_0}(\mathcal{C}')$ in the sense that
\be
	\xymatrix{ & G_0 \ar[dl]_\Phi \ar[dr]^{\Phi'} & \\
	\mathcal{C} \ar[rr]_\Psi && \mathcal{C}'
	}
\ee
commutes up to natural isomorphism. In this setup, we may ask if there is a symmetric monoidal category $\mathcal{F}(G_0)$ together with a functor $I\colon G_0 \to \mathcal{F}(G_0)$ such that the pair $(\mathcal{F}(G_0),I)$ is universal in the following sense: 
for all symmetric monoidal categories $\mathcal{C}$, precomposing with $I$ induces an equivalence of categories
\be\label{eq:FG0-univ-prop}
\mathrm{Fun}_{\otimes,\mathrm{sym}}(\mathcal{F}(G_0),\mathcal{C})
\xrightarrow{(-) \circ I}
F^{G_0}(\mathcal{C})
\; . 
\ee
We call the pair $(\mathcal{F}(G_0),I)$ the \textsl{free symmetric monoidal category generated by~$G_0$}.

An equivalent way of stating the universal property \eqref{eq:FG0-univ-prop}  is to impose the following two conditions on the pair $(\mathcal{F}(G_0),I)$, which have to hold for every symmetric monoidal category $\mathcal{C}$: 
\begin{enumerate}
\item 
We require that for each function $\Phi\colon G_0 \to \Obj(\mathcal{C})$
there exists a unique-up-to-natural-monoidal-isomorphism symmetric monoidal functor $\widetilde \Phi\colon \mathcal{F}(G_0) \to \mathcal{C}$ such that
\be
	\xymatrix{ & G_0 \ar[dl]_{I} \ar[dr]^\Phi & \\
	\mathcal{F}(G_0)\ar@{-->}[rr]_{\exists! \widetilde \Phi} && \mathcal{C}
	}
\ee
commutes up to natural isomorphism. 
\item
Let $\Phi,\Psi\colon G_0 \to \Obj(\mathcal{C})$ be two functions. Note that a natural transformation $\phi\colon \Phi \to \Psi$ is given by a collection $(\phi_x)_{x \in G_0}$ of morphisms $\phi_x\colon \Phi(x) \to \Psi(x)$ in $\mathcal{C}$ with no further conditions imposed. 

We require that for each collection of morphisms $(\phi_x\colon \Phi(x) \to \Psi(x))_{x \in G_0}$ in $\mathcal{C}$ there exists a unique natural monoidal transformations $\widetilde\phi\colon \widetilde \Phi \to \widetilde \Psi$ such that $\phi_x = \widetilde \phi_{I(x)}$ for all $x \in G_0$.
\end{enumerate}
In particular, condition (i) and (ii) imply that functors out of the free symmetric monoidal category $\mathcal{F}(G_0)$, as well as natural monoidal transformations between them, are uniquely determined by what they do on generating objects. 

Via the usual argument, any other pair $(\mathcal{F}',I')$ satisfying the universal property will be equivalent to 
$(\mathcal{F}(G_0),I)$ as a symmetric monoidal category via a unique-up-to-natural-monoidal-isomorphism equivalence compatible with $I,I'$.

\medskip 

One can easily 
write down a realisation of $\mathcal{F}(G_0)$ and $I$. Namely, as objects of $\mathcal{F}(G_0)$ take finite ordered lists $\boldsymbol{x} = (x_1,x_2,\dots, x_m)$ of elements of $G_0$, including the empty list. 
The functor $I$ maps $x \in G_0$ to the one-element list $(x)$. 
The tensor product is given on objects by concatenation of lists, and the tensor unit is the empty list. 
There are no morphisms between lists of different length. For two lists $\boldsymbol{x} = (x_1,x_2,\dots, x_m)$ and $\boldsymbol{y} = (y_1,y_2,\dots, y_m)$ of length $m$, the morphisms $\boldsymbol{x} \to \boldsymbol{y}$ are given by all permutations $\pi$ of $m$ elements such that $y_{\pi(i)} = x_i$ for all $i \in \{1,\dots,m\}$. 
The empty list only has the identity morphism. Composition of morphisms is composition of permutations.

\medskip

For the generators and relations description of 1-dimensional TQFT that we are after, one can take $G_0 = \{ \bullet_+ , \bullet_- \}$. Then, for example, there is only one morphism $(\bullet_+,\bullet_-) \to (\bullet_-,\bullet_+)$ in $\mathcal{F}(G_0)$, but there are two morphisms  $(\bullet_+,\bullet_+) \to (\bullet_+,\bullet_+)$.
By the universal property \eqref{eq:FG0-univ-prop}, giving a symmetric monoidal functor $\mathcal{F}(G_0) \to \Vectk$ amounts to picking vector spaces $U$ for $\bullet_+$ and $V$ for $\bullet_-$.

\subsubsection*{Step 2: Objects and morphisms} 

Having the symmetric monoidal category $\mathcal{F}(G_0)$ at our disposal, we can try to add a set $G_1$ of extra morphisms to $\mathcal{F}(G_0)$. More formally, pick a set $G_1$ and maps $s, t\colon G_1 \to \Obj(\mathcal{F}(G_0))$ (``source'' and ``target''). 

Let $\mathcal{C}$ be an arbitrary symmetric monoidal category. 
We embellish the functor category 
$\mathrm{Fun}_{\otimes,\mathrm{sym}}(\mathcal{F}(G_0),\mathcal{C})$ from above by including a choice of a morphism for each element of $G_1$. That is, we define a category $F^{G_0,G_1}(\mathcal{C})$ with
\begin{itemize}
\item \textsl{objects:} pairs $(\Phi,H)$, where $\Phi \in \mathrm{Fun}_{\otimes,\mathrm{sym}}(\mathcal{F}(G_0),\mathcal{C})$ and
$H$ is a map from $G_1$ into $\mathrm{Mor}(\mathcal{C})$ such that for $f \in G_1$, $H(f)$ has source $\Phi(s(f))$ and target $\Phi(t(f))$.
\item \textsl{morphisms $(\Phi,H) \to (\Phi',H')$:} monoidal transformations $\phi\colon \Phi \to \Phi'$ which make
\be\label{eq:freesmc-morph}
\xymatrix{
\Phi(s(f))  \ar[d]_{\phi_{s(f)}} \ar[r]^{H(f)} & \Phi(t(f)) \ar[d]^{\phi_{t(f)}}
\\
\Phi'(s(f)) \ar[r]_{H'(f)} & \Phi'(t(f)) 
}
\ee 
commute for each $f \in G_1$.
\end{itemize}

The \textsl{free symmetric monoidal category generated by $G_0$ and $G_1$} is the universal such category in the following sense: it is a symmetric monoidal category $\mathcal{F}(G_0,G_1)$ together with a pair $(J,j)$ of a symmetric monoidal functor $J\colon \mathcal{F}(G_0) \to \mathcal{F}(G_0,G_1)$ and a function $j\colon G_1 \to \mathrm{Mor}(\mathcal{F}(G_0,G_1))$ as above, such that for each symmetric monoidal category the functor
\be
	\mathrm{Fun}_{\otimes,\mathrm{sym}}\big( \,\mathcal{F}(G_0,G_1) \, , \,\mathcal{C} \,\big)
	\lra F^{G_0,G_1}(\mathcal{C})	
	\, , \quad
	\Phi \lmt (\Phi \circ J, \Phi \circ j)
\ee
is an equivalence of categories.

In particular, picking a symmetric monoidal functor out of $\mathcal{F}(G_0,G_1)$ is the same as picking an object in $\mathcal{C}$ for each element of $G_0$ and a morphism of $\mathcal{C}$ with the correct source and target for each element of $G_1$.
A monoidal natural transformation between two such functors $P,P'$ is a collection of morphisms $\phi_x\colon P(x) \to P'(x)$ in $\mathcal{C}$ such that \eqref{eq:freesmc-morph} (with $\Phi,H$ both replaced by $P$ and $\Phi',H'$ replaced by $P'$) commutes for all $f \in G_1$.

One can show that $\mathcal{F}(G_0,G_1)$ exists by writing down an explicit realisation, but we will not do this and content ourselves with being able to describe functors out of $\mathcal{F}(G_0,G_1)$.

\medskip

In our example of 1-dimensional TQFTs, $G_1$ is as in \eqref{eq:genrelex-mor} above, with $s,t$ given by 
{
$
s(
\begin{tikzpicture}[very thick, scale=0.39, color=blue!50!black, baseline=-0.25cm, rotate=180, >=stealth]
\draw[-{stealth[scale width=1.1]}] 
(2,0) .. controls +(0,1) and +(0,1) .. (3,0);
\end{tikzpicture} 
) 
= 
()
$, 
$t(
\begin{tikzpicture}[very thick, scale=0.39, color=blue!50!black, baseline=-0.25cm, rotate=180, >=stealth]
\draw[-{stealth[scale width=1.1]}] 
(2,0) .. controls +(0,1) and +(0,1) .. (3,0);
\end{tikzpicture} 
) 
= (\bullet_+,\bullet_-)$}, etc.
To give a symmetric monoidal functor $\mathcal{F}(G_0,G_1) \to \Vectk$ amounts to picking $U,V$ as in step 1, together with  morphisms $b\colon \Bbbk \to U \otimes V$ and $d\colon V \otimes U \to \Bbbk$. 
A monoidal natural transformation to another such functor with data $(U',V',b',d')$ amounts to choosing linear maps $f\colon U \to U'$, $g\colon V \to V'$ such that $(f \otimes g) \circ b = b'$ and $d' \circ (g \otimes f) = d$.
Note that at this point, $U,V$ do not have to be finite-dimensional. 
This will be enforced only by the relations we turn to next.

\subsubsection*{Step 3: Objects, morphisms and relations}
 
We follow the pattern set above. We already have $\mathcal{F}(G_0,G_1)$ at our disposal. 
Then the relations $G_2$ are a set of diagrams in $\mathcal{F}(G_0,G_1)$ which we would like to commute. 
We formalise this by saying that an element of $G_2$ is a pair $(f_1,f_2)$, where $f_1, f_2\colon \boldsymbol{x} \to \boldsymbol{y}$ are morphisms in $\mathcal{F}(G_0,G_1)$.

Let again $\mathcal{C}$ be an arbitrary symmetric monoidal category. We define the category $F^{G_0,G_1,G_2}(\mathcal{C})$ to be the full subcategory of 
$\mathrm{Fun}_{\otimes,\mathrm{sym}}\big( \mathcal{F}(G_0,G_1)  , \,\mathcal{C} \big)$ whose objects are those symmetric monoidal functors $F$
which satisfy $F(f_1) = F(f_2)$ for each pair $(f_1,f_2) \in G_2$. Finally, we can state:

\begin{definition}
\label{def:freesymmon}
A \textsl{symmetric monoidal category freely generated by objects $G_0$, morphisms $G_1$ and relations $G_2$} is 
\begin{itemize}
\item
a symmetric monoidal category $\mathcal{F}$, 
\item
a symmetric monoidal functor $S\colon  \mathcal{F}(G_0,G_1)  \to \mathcal{F}$ such that $S(f_1) = S(f_2)$ for each pair $(f_1,f_2) \in G_2$,
\end{itemize}
such that for each symmetric monoidal category $\mathcal{C}$, the functor
\be
	\mathrm{Fun}_{\otimes,\mathrm{sym}}( \mathcal{F}  , \mathcal{C} )
	\xrightarrow{~(-) \circ S~} F^{G_0,G_1,G_2}(\mathcal{C})
\ee
is an equivalence of categories.
We denote such a category by $\mathcal{F}(G_0,G_1,G_2) := \mathcal{F}$.
\end{definition}

In particular, a symmetric monoidal functor out of $\mathcal{F}(G_0,G_1,G_2)$ is determined uniquely up to natural monoidal isomorphism by a choice of object for each element of $G_0$ and a choice of morphism with correct source and target for each element of $G_1$, all subject to the condition that the diagrams from $G_2$ commute.
The data determining a monoidal natural transformation between two such functors is the same as in step 2.

\medskip

In our example of 1-dimensional TQFTs, $G_2$ is the two-element set from \eqref{eq:genrelex-rel} above. 
To give a symmetric monoidal functor $\mathcal{F}(G_0,G_1,G_2) \to \Vectk$ then amounts to picking $(U,V,b,d)$ as above, but now subject to the relations
\be 
	(d \otimes \id) \circ (\id \otimes b) = \id
	\, , \quad
	(\id \otimes d) \circ (b \otimes \id)  = \id \, .
\ee 
We note that this precisely describes an object of $\mathcal{DP}_\Bbbk$. 
In step 2 we already saw that a natural monoidal transformation between two such functors is the same as giving a morphism in $\mathcal{DP}_\Bbbk$.
We obtain an equivalence of categories
\be\label{eq:1dim-genrel-dualpair}
\mathrm{Fun}_{\otimes,\mathrm{sym}}\big( \,\mathcal{F}(G_0,G_1,G_2) \, , \,\Vectk \big) \lra \mathcal{DP}_\Bbbk \, .
\ee
This equivalence also re-establishes that $\mathcal{DP}_\Bbbk$ is a groupoid: $\bullet_+$ is left dual to $\bullet_-$ in $\mathcal{F}(G_0,G_1,G_2)$, hence by Lemma~\ref{lem:sym-mon-inv} the left-hand side of the above equivalence is a groupoid, and thus also the right-hand side.

\medskip

A systematic approach to this type of description of freely generated categories with relations which generalises to higher categories can be found in \cite[Sect.\,2.10]{spthesis}.

\subsubsection*{Algebraic description of TQFTs via generators and relations}

With the help our new toy -- freely generated symmetric monoidal categories -- we can say more precisely what we mean by a description of TQFTs via generators and relations.

Namely, suppose we found sets $G_0,G_1,G_2$ such that $\Bord_n$ satisfies the universal property of $\mathcal{F}(G_0,G_1,G_2)$. In other words, we may just take $\mathcal{F}(G_0,G_1,G_2) = \Bord_n$. 
Then our task is to find a nice algebraic structure whose category is equivalent to $\mathrm{Fun}_{\otimes,\mathrm{sym}}( \mathcal{F}(G_0,G_1,G_2) , \Vectk)$. This is then an algebraic description of $n$-dimensional TQFTs. Of course, the generators and relations description is of most use if all three sets $G_0$, $G_1$, $G_2$ are finite.

\medskip

In our 1-dimensional example, the discussion in Section \ref{sec:one-dim-class} can be reformulated as follows. First one rewrites the proof of the preliminary version of Theorem \ref{thm:1dTQFTprelim} as:

\begin{theorem}
$\Bord_1$ is freely generated as a symmetric monoidal category by objects $\bullet_+$, $\bullet_-$, morphisms \eqref{eq:genrelex-mor} and relations \eqref{eq:genrelex-rel}.
\end{theorem}

Then the observation \eqref{eq:1dim-genrel-dualpair} amounts to the statement of the final version of Theorem \ref{thm:1dTQFTprelim}.

\subsection{Two-dimensional TQFTs}
\label{subsec:2dTQFTs}

We now move to 2-dimensional TQFTs. 
Here the analogue of Theorem~\ref{thm:1dTQFT} reads as follows: 
\begin{theorem}
\label{thm:2dTQFT}
There is an equivalence of groupoids 
\be 
\big\{ \textrm{TQFTs } \Bord_2 \lra \Vectk \!\big\}
\stackrel{\sim}{\lra}
\textrm{comFrob}_\Bbbk
\ee 
given on objects by $\zz \mapsto \zz(S^1)$. 
\end{theorem}

We will solve the mystery of what $\textrm{comFrob}_\Bbbk$ is only while proving the above classification result. 
As in the 1-dimensional case we start with the generators of the bordism category. 
By definition, an object in $\Bord_2$ 
is orientation-preserving diffeomorphic to a finite disjoint union of (oriented) $S^1$s. 
Put differently, 
every object in $\Bord_2$ is isomorphic to a finite tensor product of $S^1$s, 
so in the notation of Section~\ref{sec:gen+rel_def} we have 
\be
G_0= \big\{ S^1 \big\} \, . 
\ee

It is a classical result (which may be proven with or without Morse theory \cite[Sect.\,1.4]{Kockbook}) 
that the morphisms of $\Bord_2$ can be obtained by composing and tensoring the elementary bordisms
\be\label{eq:Bord2gen}
\tikzzbox
{
\begin{tikzpicture}[thick,scale=1.0,color=black, baseline=0.3cm]
\coordinate (p1) at (-0.55,0);
\coordinate (p2) at (-0.2,0);
\coordinate (p3) at (-0.55,0.8);
\coordinate (p4) at (-0.2,0.8);
%
\fill [orange!23] 
(p1) -- (p2)
-- (p4)
.. controls +(0,0.1) and +(0,0.1) .. (p3)
-- (p1)
;
%
\fill [orange!38] 
(p1) .. controls +(0,0.1) and +(0,0.1) ..  (p2) -- 
(p2) .. controls +(0,-0.1) and +(0,-0.1) ..  (p1)
;
\draw (p2) -- (p4); 
\draw (p3) -- (p1); 
\draw[very thick, color=red!80!black] (p3) .. controls +(0,0.1) and +(0,0.1) ..  (p4); 
\draw[very thick, color=red!80!black, opacity=0.2] (p3) .. controls +(0,-0.1) and +(0,-0.1) ..  (p4); 
\draw[very thick, color=red!80!black] (p1) .. controls +(0,0.1) and +(0,0.1) ..  (p2); 
\draw[very thick, color=red!80!black] (p1) .. controls +(0,-0.1) and +(0,-0.1) ..  (p2); 
\end{tikzpicture}
} 
\, , \quad
\tikzzbox
{
\begin{tikzpicture}[thick,scale=1.0,color=black, baseline=0.3cm]
\coordinate (p1) at (-0.55,0);
\coordinate (p2) at (-0.2,0);
\coordinate (p3) at (0.2,0);
\coordinate (p4) at (0.55,0);
\coordinate (p5) at (0.175,0.8);
\coordinate (p6) at (-0.175,0.8);
%
\fill [orange!23] 
(p1) .. controls +(0,-0.15) and +(0,-0.15) ..  (p2)
-- (p2) .. controls +(0,0.35) and +(0,0.35) ..  (p3)
-- (p3) .. controls +(0,-0.15) and +(0,-0.15) ..  (p4)
-- (p4) .. controls +(0,0.5) and +(0,-0.5) ..  (p5)
-- (p5) .. controls +(0,0.15) and +(0,0.15) ..  (p6)
-- (p6) .. controls +(0,-0.5) and +(0,0.5) ..  (p1)
;
\fill [orange!38] 
(p1) .. controls +(0,-0.15) and +(0,-0.15) ..  (p2)
-- (p2) .. controls +(0,0.15) and +(0,0.15) ..  (p1)
;
\fill [orange!38] 
(p3) .. controls +(0,-0.15) and +(0,-0.15) ..  (p4)
-- (p4) .. controls +(0,0.15) and +(0,0.15) ..  (p3)
;
\draw (p2) .. controls +(0,0.35) and +(0,0.35) ..  (p3); 
\draw (p4) .. controls +(0,0.5) and +(0,-0.5) ..  (p5); 
\draw (p6) .. controls +(0,-0.5) and +(0,0.5) ..  (p1); 
\draw[very thick, red!80!black] (p1) .. controls +(0,0.15) and +(0,0.15) ..  (p2); 
\draw[very thick, red!80!black] (p1) .. controls +(0,-0.15) and +(0,-0.15) ..  (p2); 
\draw[very thick, red!80!black] (p3) .. controls +(0,0.15) and +(0,0.15) ..  (p4); 
\draw[very thick, red!80!black] (p3) .. controls +(0,-0.15) and +(0,-0.15) ..  (p4); 
\draw[very thick, red!80!black] (p5) .. controls +(0,0.15) and +(0,0.15) ..  (p6); 
\draw[very thick, red!80!black, opacity=0.2] (p5) .. controls +(0,-0.15) and +(0,-0.15) ..  (p6); 
\end{tikzpicture}
} 
\, , \quad
\tikzzbox
{
\begin{tikzpicture}[thick,scale=1.0,color=black, rotate=180, baseline=-0.5cm]
\coordinate (p1) at (-0.55,0);
\coordinate (p2) at (-0.2,0);
\coordinate (p3) at (0.2,0);
\coordinate (p4) at (0.55,0);
\coordinate (p5) at (0.175,0.8);
\coordinate (p6) at (-0.175,0.8);
%
\fill [orange!23] 
(p1) .. controls +(0,-0.15) and +(0,-0.15) ..  (p2)
-- (p2) .. controls +(0,0.35) and +(0,0.35) ..  (p3)
-- (p3) .. controls +(0,-0.15) and +(0,-0.15) ..  (p4)
-- (p4) .. controls +(0,0.5) and +(0,-0.5) ..  (p5)
-- (p5) .. controls +(0,0.15) and +(0,0.15) ..  (p6)
-- (p6) .. controls +(0,-0.5) and +(0,0.5) ..  (p1)
;
\fill [orange!38] 
(p6) .. controls +(0,-0.15) and +(0,-0.15) ..  (p5)
-- (p5) .. controls +(0,0.15) and +(0,0.15) ..  (p6)
;
\draw (p2) .. controls +(0,0.35) and +(0,0.35) ..  (p3); 
\draw (p4) .. controls +(0,0.5) and +(0,-0.5) ..  (p5); 
\draw (p6) .. controls +(0,-0.5) and +(0,0.5) ..  (p1); 
\draw[very thick, red!80!black, opacity=0.2] (p1) .. controls +(0,0.15) and +(0,0.15) ..  (p2); 
\draw[very thick, red!80!black] (p1) .. controls +(0,-0.15) and +(0,-0.15) ..  (p2); 
\draw[very thick, red!80!black, opacity=0.2] (p3) .. controls +(0,0.15) and +(0,0.15) ..  (p4); 
\draw[very thick, red!80!black] (p3) .. controls +(0,-0.15) and +(0,-0.15) ..  (p4); 
\draw[very thick, red!80!black] (p5) .. controls +(0,0.15) and +(0,0.15) ..  (p6); 
\draw[very thick, red!80!black] (p5) .. controls +(0,-0.15) and +(0,-0.15) ..  (p6); 
\end{tikzpicture}
} 
\, , \quad 
\tikzzbox
{
\begin{tikzpicture}[thick,scale=1.0,color=black, baseline=-0.15cm]
\coordinate (p1) at (-0.55,0);
\coordinate (p2) at (-0.2,0);
%
\fill [orange!23] 
(p1) .. controls +(0,0.15) and +(0,0.15) ..  (p2)
-- (p2) .. controls +(0,-0.4) and +(0,-0.4) ..  (p1)
;
\draw (p1) .. controls +(0,-0.4) and +(0,-0.4) ..  (p2); 
\draw[very thick, red!80!black] (p1) .. controls +(0,0.15) and +(0,0.15) ..  (p2); 
\draw[very thick, red!80!black, opacity=0.2] (p1) .. controls +(0,-0.15) and +(0,-0.15) ..  (p2); 
\end{tikzpicture}
} 
\, , \quad
\tikzzbox
{
\begin{tikzpicture}[thick,scale=1.0,color=black, rotate=180, baseline=0.1cm]
\coordinate (p1) at (-0.55,0);
\coordinate (p2) at (-0.2,0);
%
\fill [orange!23] 
(p1) .. controls +(0,0.15) and +(0,0.15) ..  (p2)
-- (p2) .. controls +(0,-0.4) and +(0,-0.4) ..  (p1)
;
\fill [orange!38] 
(p1) .. controls +(0,-0.15) and +(0,-0.15) ..  (p2)
-- (p2) .. controls +(0,0.15) and +(0,0.15) ..  (p1)
;
\draw (p1) .. controls +(0,-0.4) and +(0,-0.4) ..  (p2); 
\draw[very thick, red!80!black] (p1) .. controls +(0,0.15) and +(0,0.15) ..  (p2); 
\draw[very thick, red!80!black] (p1) .. controls +(0,-0.15) and +(0,-0.15) ..  (p2); 
\end{tikzpicture}
} 
\, , \quad
\tikzzbox
{
\begin{tikzpicture}[thick,scale=1.0,color=black, baseline=0.3cm]
\coordinate (p1) at (-0.55,0);
\coordinate (p2) at (-0.2,0);
\coordinate (p3) at (0.2,0);
\coordinate (p4) at (0.55,0);
\coordinate (p5) at (-0.55,0.8);
\coordinate (p6) at (-0.2,0.8);
\coordinate (p7) at (0.2,0.8);
\coordinate (p8) at (0.55,0.8);
%
\fill [orange!23] 
(p1) .. controls +(0,-0.15) and +(0,-0.15) ..  (p2)
-- (p2) -- (p8)
-- (p8) .. controls +(0,0.15) and +(0,0.15) ..  (p7)
-- (p7) -- (p1)
;
\fill [orange!23] 
(p3) .. controls +(0,-0.15) and +(0,-0.15) ..  (p4)
-- (p4) -- (p6)
-- (p6) .. controls +(0,0.15) and +(0,0.15) ..  (p5)
-- (p5) -- (p3)
;
\fill [orange!38] 
(p1) .. controls +(0,-0.15) and +(0,-0.15) ..  (p2)
-- (p2) .. controls +(0,0.15) and +(0,0.15) ..  (p1)
;
\fill [orange!38] 
(p3) .. controls +(0,-0.15) and +(0,-0.15) ..  (p4)
-- (p4) .. controls +(0,0.15) and +(0,0.15) ..  (p3)
;
\draw (p1) -- (p7);
\draw (p2) -- (p8);
\draw (p3) -- (p5);
\draw (p4) -- (p6);
\draw[very thick, red!80!black] (p1) .. controls +(0,0.15) and +(0,0.15) ..  (p2); 
\draw[very thick, red!80!black] (p1) .. controls +(0,-0.15) and +(0,-0.15) ..  (p2); 
\draw[very thick, red!80!black] (p3) .. controls +(0,0.15) and +(0,0.15) ..  (p4); 
\draw[very thick, red!80!black] (p3) .. controls +(0,-0.15) and +(0,-0.15) ..  (p4); 
\draw[very thick, red!80!black] (p5) .. controls +(0,0.15) and +(0,0.15) ..  (p6); 
\draw[very thick, red!80!black, opacity=0.2] (p5) .. controls +(0,-0.15) and +(0,-0.15) ..  (p6); 
\draw[very thick, red!80!black] (p7) .. controls +(0,0.15) and +(0,0.15) ..  (p8); 
\draw[very thick, red!80!black, opacity=0.2] (p7) .. controls +(0,-0.15) and +(0,-0.15) ..  (p8); 
\end{tikzpicture}
\, , 
} 
\ee
where 
$
\tikzzbox
{
\begin{tikzpicture}[thick,scale=0.45,color=black, baseline=0cm]
\coordinate (p1) at (-0.55,0);
\coordinate (p2) at (-0.2,0);
\coordinate (p3) at (-0.55,0.8);
\coordinate (p4) at (-0.2,0.8);
%
\fill [orange!23] 
(p1) -- (p2)
-- (p4)
.. controls +(0,0.1) and +(0,0.1) .. (p3)
-- (p1)
;
%
\fill [orange!38] 
(p1) .. controls +(0,0.1) and +(0,0.1) ..  (p2) -- 
(p2) .. controls +(0,-0.1) and +(0,-0.1) ..  (p1)
;
\draw (p2) -- (p4); 
\draw (p3) -- (p1); 
\draw[very thick, color=red!80!black] (p3) .. controls +(0,0.1) and +(0,0.1) ..  (p4); 
\draw[very thick, color=red!80!black, opacity=0.2] (p3) .. controls +(0,-0.1) and +(0,-0.1) ..  (p4); 
\draw[very thick, color=red!80!black] (p1) .. controls +(0,0.1) and +(0,0.1) ..  (p2); 
\draw[very thick, color=red!80!black] (p1) .. controls +(0,-0.1) and +(0,-0.1) ..  (p2); 
\end{tikzpicture}
} 
$
represents the unit morphism $1_{S^1}$ and
$
\begin{tikzpicture}[thick,scale=0.45,color=black, baseline=-0cm]
\coordinate (p1) at (-0.55,0);
\coordinate (p2) at (-0.2,0);
\coordinate (p3) at (0.2,0);
\coordinate (p4) at (0.55,0);
\coordinate (p5) at (-0.55,0.8);
\coordinate (p6) at (-0.2,0.8);
\coordinate (p7) at (0.2,0.8);
\coordinate (p8) at (0.55,0.8);
%
\fill [orange!23] 
(p1) .. controls +(0,-0.15) and +(0,-0.15) ..  (p2)
-- (p2) -- (p8)
-- (p8) .. controls +(0,0.15) and +(0,0.15) ..  (p7)
-- (p7) -- (p1)
;
\fill [orange!23] 
(p3) .. controls +(0,-0.15) and +(0,-0.15) ..  (p4)
-- (p4) -- (p6)
-- (p6) .. controls +(0,0.15) and +(0,0.15) ..  (p5)
-- (p5) -- (p3)
;
\fill [orange!38] 
(p1) .. controls +(0,-0.15) and +(0,-0.15) ..  (p2)
-- (p2) .. controls +(0,0.15) and +(0,0.15) ..  (p1)
;
\fill [orange!38] 
(p3) .. controls +(0,-0.15) and +(0,-0.15) ..  (p4)
-- (p4) .. controls +(0,0.15) and +(0,0.15) ..  (p3)
;
\draw (p1) -- (p7);
\draw (p2) -- (p8);
\draw (p3) -- (p5);
\draw (p4) -- (p6);
\draw[very thick, red!80!black] (p1) .. controls +(0,0.15) and +(0,0.15) ..  (p2); 
\draw[very thick, red!80!black] (p1) .. controls +(0,-0.15) and +(0,-0.15) ..  (p2); 
\draw[very thick, red!80!black] (p3) .. controls +(0,0.15) and +(0,0.15) ..  (p4); 
\draw[very thick, red!80!black] (p3) .. controls +(0,-0.15) and +(0,-0.15) ..  (p4); 
\draw[very thick, red!80!black] (p5) .. controls +(0,0.15) and +(0,0.15) ..  (p6); 
\draw[very thick, red!80!black, opacity=0.2] (p5) .. controls +(0,-0.15) and +(0,-0.15) ..  (p6); 
\draw[very thick, red!80!black] (p7) .. controls +(0,0.15) and +(0,0.15) ..  (p8); 
\draw[very thick, red!80!black, opacity=0.2] (p7) .. controls +(0,-0.15) and +(0,-0.15) ..  (p8); 
\end{tikzpicture}
$
represents the 
braiding bordism $\beta_{S^1,S^1}$ of~\eqref{eq:Bordbraid}. 
Every surface can be chopped into the above pieces. 

Since $\Bord_2$ is a \textsl{symmetric} monoidal category which we hope to describe as freely generated by some generators and relations, we can drop the unit and braiding bordisms from the set of morphism generators. 
Hence we set 
\be
\label{eq:Bord2gen2}
G_1 
= 
\bigg\{ \;
\tikzzbox
{
\begin{tikzpicture}[thick,scale=1.0,color=black, baseline=0.3cm]
\coordinate (p1) at (-0.55,0);
\coordinate (p2) at (-0.2,0);
\coordinate (p3) at (0.2,0);
\coordinate (p4) at (0.55,0);
\coordinate (p5) at (0.175,0.8);
\coordinate (p6) at (-0.175,0.8);
%
\fill [orange!23] 
(p1) .. controls +(0,-0.15) and +(0,-0.15) ..  (p2)
-- (p2) .. controls +(0,0.35) and +(0,0.35) ..  (p3)
-- (p3) .. controls +(0,-0.15) and +(0,-0.15) ..  (p4)
-- (p4) .. controls +(0,0.5) and +(0,-0.5) ..  (p5)
-- (p5) .. controls +(0,0.15) and +(0,0.15) ..  (p6)
-- (p6) .. controls +(0,-0.5) and +(0,0.5) ..  (p1)
;
\fill [orange!38] 
(p1) .. controls +(0,-0.15) and +(0,-0.15) ..  (p2)
-- (p2) .. controls +(0,0.15) and +(0,0.15) ..  (p1)
;
\fill [orange!38] 
(p3) .. controls +(0,-0.15) and +(0,-0.15) ..  (p4)
-- (p4) .. controls +(0,0.15) and +(0,0.15) ..  (p3)
;
\draw (p2) .. controls +(0,0.35) and +(0,0.35) ..  (p3); 
\draw (p4) .. controls +(0,0.5) and +(0,-0.5) ..  (p5); 
\draw (p6) .. controls +(0,-0.5) and +(0,0.5) ..  (p1); 
\draw[very thick, red!80!black] (p1) .. controls +(0,0.15) and +(0,0.15) ..  (p2); 
\draw[very thick, red!80!black] (p1) .. controls +(0,-0.15) and +(0,-0.15) ..  (p2); 
\draw[very thick, red!80!black] (p3) .. controls +(0,0.15) and +(0,0.15) ..  (p4); 
\draw[very thick, red!80!black] (p3) .. controls +(0,-0.15) and +(0,-0.15) ..  (p4); 
\draw[very thick, red!80!black] (p5) .. controls +(0,0.15) and +(0,0.15) ..  (p6); 
\draw[very thick, red!80!black, opacity=0.2] (p5) .. controls +(0,-0.15) and +(0,-0.15) ..  (p6); 
\end{tikzpicture}
} 
\, , \quad
\tikzzbox
{
\begin{tikzpicture}[thick,scale=1.0,color=black, rotate=180, baseline=-0.5cm]
\coordinate (p1) at (-0.55,0);
\coordinate (p2) at (-0.2,0);
\coordinate (p3) at (0.2,0);
\coordinate (p4) at (0.55,0);
\coordinate (p5) at (0.175,0.8);
\coordinate (p6) at (-0.175,0.8);
%
\fill [orange!23] 
(p1) .. controls +(0,-0.15) and +(0,-0.15) ..  (p2)
-- (p2) .. controls +(0,0.35) and +(0,0.35) ..  (p3)
-- (p3) .. controls +(0,-0.15) and +(0,-0.15) ..  (p4)
-- (p4) .. controls +(0,0.5) and +(0,-0.5) ..  (p5)
-- (p5) .. controls +(0,0.15) and +(0,0.15) ..  (p6)
-- (p6) .. controls +(0,-0.5) and +(0,0.5) ..  (p1)
;
\fill [orange!38] 
(p6) .. controls +(0,-0.15) and +(0,-0.15) ..  (p5)
-- (p5) .. controls +(0,0.15) and +(0,0.15) ..  (p6)
;
\draw (p2) .. controls +(0,0.35) and +(0,0.35) ..  (p3); 
\draw (p4) .. controls +(0,0.5) and +(0,-0.5) ..  (p5); 
\draw (p6) .. controls +(0,-0.5) and +(0,0.5) ..  (p1); 
\draw[very thick, red!80!black, opacity=0.2] (p1) .. controls +(0,0.15) and +(0,0.15) ..  (p2); 
\draw[very thick, red!80!black] (p1) .. controls +(0,-0.15) and +(0,-0.15) ..  (p2); 
\draw[very thick, red!80!black, opacity=0.2] (p3) .. controls +(0,0.15) and +(0,0.15) ..  (p4); 
\draw[very thick, red!80!black] (p3) .. controls +(0,-0.15) and +(0,-0.15) ..  (p4); 
\draw[very thick, red!80!black] (p5) .. controls +(0,0.15) and +(0,0.15) ..  (p6); 
\draw[very thick, red!80!black] (p5) .. controls +(0,-0.15) and +(0,-0.15) ..  (p6); 
\end{tikzpicture}
} 
\, , \quad 
\tikzzbox
{
\begin{tikzpicture}[thick,scale=1.0,color=black, baseline=-0.15cm]
\coordinate (p1) at (-0.55,0);
\coordinate (p2) at (-0.2,0);
%
\fill [orange!23] 
(p1) .. controls +(0,0.15) and +(0,0.15) ..  (p2)
-- (p2) .. controls +(0,-0.4) and +(0,-0.4) ..  (p1)
;
\draw (p1) .. controls +(0,-0.4) and +(0,-0.4) ..  (p2); 
\draw[very thick, red!80!black] (p1) .. controls +(0,0.15) and +(0,0.15) ..  (p2); 
\draw[very thick, red!80!black, opacity=0.2] (p1) .. controls +(0,-0.15) and +(0,-0.15) ..  (p2); 
\end{tikzpicture}
} 
\, , \quad
\tikzzbox
{
\begin{tikzpicture}[thick,scale=1.0,color=black, rotate=180, baseline=0.1cm]
\coordinate (p1) at (-0.55,0);
\coordinate (p2) at (-0.2,0);
%
\fill [orange!23] 
(p1) .. controls +(0,0.15) and +(0,0.15) ..  (p2)
-- (p2) .. controls +(0,-0.4) and +(0,-0.4) ..  (p1)
;
\fill [orange!38] 
(p1) .. controls +(0,-0.15) and +(0,-0.15) ..  (p2)
-- (p2) .. controls +(0,0.15) and +(0,0.15) ..  (p1)
;
\draw (p1) .. controls +(0,-0.4) and +(0,-0.4) ..  (p2); 
\draw[very thick, red!80!black] (p1) .. controls +(0,0.15) and +(0,0.15) ..  (p2); 
\draw[very thick, red!80!black] (p1) .. controls +(0,-0.15) and +(0,-0.15) ..  (p2); 
\end{tikzpicture}
} 
\;
\bigg\} 
\, . 
\ee
Not having to explicitly deal with the braiding -- being baked in via the description as a freely generated symmetric monoidal category -- also cuts down the relations~$G_2$ that we have to keep track of. 

Indeed, using Morse theory, one can further identify a sufficient set of relations among the bordisms in $G_1$:
\begin{align}
& 
\tikzzbox
{
\begin{tikzpicture}[thick,scale=1.0,color=black, baseline=0.75cm]
\coordinate (p1) at (-0.55,0);
\coordinate (p2) at (-0.2,0);
\coordinate (p3) at (0.2,0);
\coordinate (p4) at (0.55,0);
\coordinate (p5) at (0.175,0.8);
\coordinate (p6) at (-0.175,0.8);
\coordinate (p7) at (0.95,0);
\coordinate (p8) at (1.3,0);
\coordinate (p9) at (0.2,1.6);
\coordinate (p10) at (0.55,1.6);
%
\fill [orange!23] 
(p1) .. controls +(0,0.1) and +(0,0.1) ..  (p2)
-- (p2) .. controls +(0,0.65) and +(0,0.65) ..  (p3)
-- (p3) -- (p4)
-- (p4) .. controls +(-0.55,1.95) and +(-0.1,0.5) ..  (p7)
-- (p7) -- (p8)
-- (p8) .. controls +(0,0.5) and +(0,-0.5) ..  (p10)
-- (p10) .. controls +(0,0.1) and +(0,0.1) ..  (p9)
-- (p9) .. controls +(0,-0.5) and +(0,0.75) ..  (p1)
;
%
\fill [orange!38] 
(p1) .. controls +(0,0.1) and +(0,0.1) ..  (p2) -- 
(p2) .. controls +(0,-0.1) and +(0,-0.1) ..  (p1)
;
%
\fill [orange!38] 
(p3) .. controls +(0,0.1) and +(0,0.1) ..  (p4) -- 
(p4) .. controls +(0,-0.1) and +(0,-0.1) ..  (p3)
;
%
\fill [orange!38] 
(p7) .. controls +(0,0.1) and +(0,0.1) ..  (p8) -- 
(p8) .. controls +(0,-0.1) and +(0,-0.1) ..  (p7)
;
\draw (p4) .. controls +(-0.55,1.95) and +(-0.1,0.5) ..  (p7);
\draw (p2) .. controls +(0,0.65) and +(0,0.65) ..  (p3); 
\draw (p8) .. controls +(0,0.5) and +(0,-0.5) ..  (p10); 
\draw (p9) .. controls +(0,-0.5) and +(0,0.75) ..  (p1); 
\draw[very thick, color=red!80!black] (p1) .. controls +(0,0.1) and +(0,0.1) ..  (p2); 
\draw[very thick, color=red!80!black] (p2) .. controls +(0,-0.1) and +(0,-0.1) ..  (p1); 
\draw[very thick, color=red!80!black] (p3) .. controls +(0,0.1) and +(0,0.1) .. (p4); 
\draw[very thick, color=red!80!black] (p4) .. controls +(0,-0.1) and +(0,-0.1) ..  (p3); 
\draw[very thick, color=red!80!black] (p7) .. controls +(0,0.1) and +(0,0.1) .. (p8); 
\draw[very thick, color=red!80!black] (p8) .. controls +(0,-0.1) and +(0,-0.1) ..  (p7); 
\draw[very thick, color=red!80!black] (p9) .. controls +(0,0.1) and +(0,0.1) .. (p10); 
\draw[very thick, color=red!80!black, opacity=0.2] (p10) .. controls +(0,-0.1) and +(0,-0.1) ..  (p9); 
\end{tikzpicture}
} 
= 
\tikzzbox
{
\begin{tikzpicture}[thick,scale=1.0,color=black, baseline=0.75cm]
\coordinate (p1) at (-0.55,0);
\coordinate (p2) at (-0.2,0);
\coordinate (p3) at (0.2,0);
\coordinate (p4) at (0.55,0);
\coordinate (p5) at (0.175,0.8);
\coordinate (p6) at (-0.175,0.8);
\coordinate (p7) at (-0.95,0);
\coordinate (p8) at (-1.3,0);
\coordinate (p9) at (-0.2,1.6);
\coordinate (p10) at (-0.55,1.6);
%
\fill[orange!23] 
(p8) -- (p7)
-- (p7) .. controls +(0.1,0.5) and +(0.55,1.95) ..  (p1)
-- (p1) -- (p2)
-- (p2) .. controls +(0,0.65) and +(0,0.65) ..  (p3)
-- (p3) -- (p4)
-- (p4) .. controls +(0,0.75) and +(0,-0.5) ..  (p9)
-- (p9) .. controls +(0,0.1) and +(0,0.1) .. (p10)
-- (p10) .. controls +(0,-0.5) and +(0,0.5) ..  (p8)
;
%
%
\fill [orange!38] 
(p1) .. controls +(0,0.1) and +(0,0.1) ..  (p2) -- 
(p2) .. controls +(0,-0.1) and +(0,-0.1) ..  (p1)
;
%
\fill [orange!38] 
(p3) .. controls +(0,0.1) and +(0,0.1) ..  (p4) -- 
(p4) .. controls +(0,-0.1) and +(0,-0.1) ..  (p3)
;
%
\fill [orange!38] 
(p7) .. controls +(0,0.1) and +(0,0.1) ..  (p8) -- 
(p8) .. controls +(0,-0.1) and +(0,-0.1) ..  (p7)
;
\draw (p7) .. controls +(0.1,0.5) and +(0.55,1.95) ..  (p1); 
\draw (p2) .. controls +(0,0.65) and +(0,0.65) ..  (p3); 
\draw (p4) .. controls +(0,0.75) and +(0,-0.5) ..  (p9); 
\draw (p10) .. controls +(0,-0.5) and +(0,0.5) ..  (p8); 
\draw[very thick, color=red!80!black] (p1) .. controls +(0,0.1) and +(0,0.1) ..  (p2); 
\draw[very thick, color=red!80!black] (p2) .. controls +(0,-0.1) and +(0,-0.1) ..  (p1); 
\draw[very thick, color=red!80!black] (p3) .. controls +(0,0.1) and +(0,0.1) .. (p4); 
\draw[very thick, color=red!80!black] (p4) .. controls +(0,-0.1) and +(0,-0.1) ..  (p3); 
\draw[very thick, color=red!80!black] (p7) .. controls +(0,0.1) and +(0,0.1) .. (p8); 
\draw[very thick, color=red!80!black] (p8) .. controls +(0,-0.1) and +(0,-0.1) ..  (p7); 
\draw[very thick, color=red!80!black] (p9) .. controls +(0,0.1) and +(0,0.1) .. (p10); 
\draw[very thick, color=red!80!black, opacity=0.2] (p10) .. controls +(0,-0.1) and +(0,-0.1) ..  (p9); 
\end{tikzpicture}
} 
\, , \quad 
\tikzzbox
{
\begin{tikzpicture}[thick,scale=1.0,color=black, rotate=180, baseline=-0.85cm]
\coordinate (p1) at (-0.55,0);
\coordinate (p2) at (-0.2,0);
\coordinate (p3) at (0.2,0);
\coordinate (p4) at (0.55,0);
\coordinate (p5) at (0.175,0.8);
\coordinate (p6) at (-0.725,0.8);
\coordinate (p7) at (0.95,0);
\coordinate (p8) at (1.3,0);
\coordinate (p9) at (0.2,1.6);
\coordinate (p10) at (0.55,1.6);
%
\fill [orange!23] 
(p1) .. controls +(0,-0.1) and +(0,-0.1) ..  (p2)
-- (p2) .. controls +(0,0.65) and +(0,0.65) ..  (p3)
-- (p3) .. controls +(0,-0.1) and +(0,-0.1) ..  (p4)
-- (p4) .. controls +(-0.55,1.95) and +(-0.1,0.5) ..  (p7)
-- (p7) .. controls +(0,-0.1) and +(0,-0.1) ..  (p8)
-- (p8) .. controls +(0,0.5) and +(0,-0.5) ..  (p10)
-- (p10) .. controls +(0,0.1) and +(0,0.1) ..  (p9)
-- (p9) .. controls +(0,-0.5) and +(0,0.75) ..  (p1)
;
%
\fill [orange!38] 
(p9) .. controls +(0,0.1) and +(0,0.1) ..  (p10) -- 
(p10) .. controls +(0,-0.1) and +(0,-0.1) ..  (p9)
;
\draw (p4) .. controls +(-0.55,1.95) and +(-0.1,0.5) ..  (p7);
\draw (p2) .. controls +(0,0.65) and +(0,0.65) ..  (p3); 
\draw (p8) .. controls +(0,0.5) and +(0,-0.5) ..  (p10); 
\draw (p9) .. controls +(0,-0.5) and +(0,0.75) ..  (p1); 
\draw[very thick, color=red!80!black, opacity=0.2] (p1) .. controls +(0,0.1) and +(0,0.1) ..  (p2); 
\draw[very thick, color=red!80!black] (p2) .. controls +(0,-0.1) and +(0,-0.1) ..  (p1); 
\draw[very thick, color=red!80!black, opacity=0.2] (p3) .. controls +(0,0.1) and +(0,0.1) .. (p4); 
\draw[very thick, color=red!80!black] (p4) .. controls +(0,-0.1) and +(0,-0.1) ..  (p3); 
\draw[very thick, color=red!80!black, opacity=0.2] (p7) .. controls +(0,0.1) and +(0,0.1) .. (p8); 
\draw[very thick, color=red!80!black] (p8) .. controls +(0,-0.1) and +(0,-0.1) ..  (p7); 
\draw[very thick, color=red!80!black] (p9) .. controls +(0,0.1) and +(0,0.1) .. (p10); 
\draw[very thick, color=red!80!black] (p10) .. controls +(0,-0.1) and +(0,-0.1) ..  (p9); 
\end{tikzpicture}
} 
= 
\tikzzbox
{
\begin{tikzpicture}[thick,scale=1.0,color=black, rotate=180, baseline=-0.85cm]
\coordinate (p1) at (-0.55,0);
\coordinate (p2) at (-0.2,0);
\coordinate (p3) at (0.2,0);
\coordinate (p4) at (0.55,0);
\coordinate (p5) at (0.175,0.8);
\coordinate (p6) at (-0.175,0.8);
\coordinate (p7) at (-0.95,0);
\coordinate (p8) at (-1.3,0);
\coordinate (p9) at (-0.2,1.6);
\coordinate (p10) at (-0.55,1.6);
%
\fill[orange!23] 
(p8) .. controls +(0,-0.1) and +(0,-0.1) ..  (p7)
-- (p7) .. controls +(0.1,0.5) and +(0.55,1.95) ..  (p1)
-- (p1) .. controls +(0,-0.1) and +(0,-0.1) ..  (p2)
-- (p2) .. controls +(0,0.65) and +(0,0.65) ..  (p3)
-- (p3) .. controls +(0,-0.1) and +(0,-0.1) ..  (p4)
-- (p4) .. controls +(0,0.75) and +(0,-0.5) ..  (p9)
-- (p9) .. controls +(0,0.1) and +(0,0.1) .. (p10)
-- (p10) .. controls +(0,-0.5) and +(0,0.5) ..  (p8)
;
%
\fill [orange!38] 
(p9) .. controls +(0,0.1) and +(0,0.1) ..  (p10) -- 
(p10) .. controls +(0,-0.1) and +(0,-0.1) ..  (p9)
;
\draw (p7) .. controls +(0.1,0.5) and +(0.55,1.95) ..  (p1); 
\draw (p2) .. controls +(0,0.65) and +(0,0.65) ..  (p3); 
\draw (p4) .. controls +(0,0.75) and +(0,-0.5) ..  (p9); 
\draw (p10) .. controls +(0,-0.5) and +(0,0.5) ..  (p8); 
\draw[very thick, color=red!80!black, opacity=0.2] (p1) .. controls +(0,0.1) and +(0,0.1) ..  (p2); 
\draw[very thick, color=red!80!black] (p2) .. controls +(0,-0.1) and +(0,-0.1) ..  (p1); 
\draw[very thick, color=red!80!black, opacity=0.2] (p3) .. controls +(0,0.1) and +(0,0.1) .. (p4); 
\draw[very thick, color=red!80!black] (p4) .. controls +(0,-0.1) and +(0,-0.1) ..  (p3); 
\draw[very thick, color=red!80!black, opacity=0.2] (p7) .. controls +(0,0.1) and +(0,0.1) .. (p8); 
\draw[very thick, color=red!80!black] (p8) .. controls +(0,-0.1) and +(0,-0.1) ..  (p7); 
\draw[very thick, color=red!80!black] (p9) .. controls +(0,0.1) and +(0,0.1) .. (p10); 
\draw[very thick, color=red!80!black] (p10) .. controls +(0,-0.1) and +(0,-0.1) ..  (p9); 
\end{tikzpicture}
}
\, , 
\label{eq:R1}
\\
& 
\tikzzbox
{
\begin{tikzpicture}[thick,scale=1.0,color=black, baseline=0.3cm]
\coordinate (p1) at (-0.55,0);
\coordinate (p2) at (-0.2,0);
\coordinate (p3) at (0.2,0);
\coordinate (p4) at (0.55,0);
\coordinate (p5) at (0.175,0.8);
\coordinate (p6) at (-0.175,0.8);
%
\fill [orange!23] 
(p1) .. controls +(0,-0.4) and +(0,-0.4) ..  (p2)
-- (p2) .. controls +(0,0.35) and +(0,0.35) ..  (p3)
-- (p3) -- (p4)
-- (p4) .. controls +(0,0.5) and +(0,-0.5) ..  (p5)
-- (p5) .. controls +(0,0.1) and +(0,0.1) .. (p6)
-- (p6) .. controls +(0,-0.5) and +(0,0.5) ..  (p1)
;
%
\fill [orange!38] 
(p3) .. controls +(0,0.1) and +(0,0.1) ..  (p4) -- 
(p4) .. controls +(0,-0.1) and +(0,-0.1) ..  (p3)
;
\draw (p1) .. controls +(0,-0.4) and +(0,-0.4) ..  (p2); 
\draw (p2) .. controls +(0,0.35) and +(0,0.35) ..  (p3); 
\draw (p4) .. controls +(0,0.5) and +(0,-0.5) ..  (p5); 
\draw (p6) .. controls +(0,-0.5) and +(0,0.5) ..  (p1); 
\draw[very thick, color=red!80!black] (p3) .. controls +(0,0.1) and +(0,0.1) ..  (p4); 
\draw[very thick, color=red!80!black] (p3) .. controls +(0,-0.1) and +(0,-0.1) ..  (p4); 
\draw[very thick, color=red!80!black] (p5) .. controls +(0,0.1) and +(0,0.1) ..  (p6); 
\draw[very thick, color=red!80!black, opacity=0.2] (p5) .. controls +(0,-0.1) and +(0,-0.1) ..  (p6); 
\end{tikzpicture}
} 
= 
\tikzzbox
{
\begin{tikzpicture}[thick,scale=1.0,color=black, baseline=0.3cm]
\coordinate (p1) at (-0.55,0);
\coordinate (p2) at (-0.2,0);
\coordinate (p3) at (-0.55,0.8);
\coordinate (p4) at (-0.2,0.8);
%
\fill [orange!23] 
(p1) -- (p2)
-- (p4)
.. controls +(0,0.1) and +(0,0.1) .. (p3)
-- (p1)
;
%
\fill [orange!38] 
(p1) .. controls +(0,0.1) and +(0,0.1) ..  (p2) -- 
(p2) .. controls +(0,-0.1) and +(0,-0.1) ..  (p1)
;
\draw (p2) -- (p4); 
\draw (p3) -- (p1); 
\draw[very thick, color=red!80!black] (p3) .. controls +(0,0.1) and +(0,0.1) ..  (p4); 
\draw[very thick, color=red!80!black, opacity=0.2] (p3) .. controls +(0,-0.1) and +(0,-0.1) ..  (p4); 
\draw[very thick, color=red!80!black] (p1) .. controls +(0,0.1) and +(0,0.1) ..  (p2); 
\draw[very thick, color=red!80!black] (p1) .. controls +(0,-0.1) and +(0,-0.1) ..  (p2); 
\end{tikzpicture}
} 
=
\tikzzbox
{
\begin{tikzpicture}[thick,scale=1.0,color=black, baseline=0.3cm]
\coordinate (p1) at (-0.55,0);
\coordinate (p2) at (-0.2,0);
\coordinate (p3) at (0.2,0);
\coordinate (p4) at (0.55,0);
\coordinate (p5) at (0.175,0.8);
\coordinate (p6) at (-0.175,0.8);
%
\fill [orange!23] 
(p3) -- (p4)
-- (p2) .. controls +(0,0.35) and +(0,0.35) ..  (p3)
-- (p3) .. controls +(0,-0.4) and +(0,-0.4) ..  (p4)
-- (p4) .. controls +(0,0.5) and +(0,-0.5) ..  (p5)
-- (p5) .. controls +(0,0.1) and +(0,0.1) ..  (p6)
-- (p6) .. controls +(0,-0.5) and +(0,0.5) ..  (p1)
;
\fill [orange!38] 
(p1) .. controls +(0,0.1) and +(0,0.1) ..  (p2) -- 
(p2) .. controls +(0,-0.1) and +(0,-0.1) ..  (p1)
;
\draw (p2) .. controls +(0,0.35) and +(0,0.35) ..  (p3); 
\draw (p3) .. controls +(0,-0.4) and +(0,-0.4) ..  (p4); 
\draw (p4) .. controls +(0,0.5) and +(0,-0.5) ..  (p5); 
\draw (p6) .. controls +(0,-0.5) and +(0,0.5) ..  (p1); 
\draw[very thick, color=red!80!black] (p1) .. controls +(0,0.1) and +(0,0.1) ..  (p2); 
\draw[very thick, color=red!80!black] (p1) .. controls +(0,-0.1) and +(0,-0.1) ..  (p2); 
\draw[very thick, color=red!80!black] (p5) .. controls +(0,0.1) and +(0,0.1) ..  (p6); 
\draw[very thick, color=red!80!black, opacity=0.2] (p5) .. controls +(0,-0.1) and +(0,-0.1) ..  (p6); 
\end{tikzpicture}
} 
\, , \quad 
\tikzzbox
{
\begin{tikzpicture}[thick,scale=1.0,color=black, rotate=180, baseline=-0.5cm]
\coordinate (p1) at (-0.55,0);
\coordinate (p2) at (-0.2,0);
\coordinate (p3) at (0.2,0);
\coordinate (p4) at (0.55,0);
\coordinate (p5) at (0.175,0.8);
\coordinate (p6) at (-0.175,0.8);
%
\fill [orange!23] 
(p1) .. controls +(0,-0.4) and +(0,-0.4) ..  (p2)
-- (p2) .. controls +(0,0.35) and +(0,0.35) ..  (p3)
-- (p3) .. controls +(0,-0.1) and +(0,-0.1) .. (p4)
-- (p4) .. controls +(0,0.5) and +(0,-0.5) ..  (p5)
-- (p5) .. controls +(0,0.1) and +(0,0.1) .. (p6)
-- (p6) .. controls +(0,-0.5) and +(0,0.5) ..  (p1)
;
%
\fill [orange!38] 
(p5) .. controls +(0,0.1) and +(0,0.1) ..  (p6) -- 
(p6) .. controls +(0,-0.1) and +(0,-0.1) ..  (p5)
;
\draw (p1) .. controls +(0,-0.4) and +(0,-0.4) ..  (p2); 
\draw (p2) .. controls +(0,0.35) and +(0,0.35) ..  (p3); 
\draw (p4) .. controls +(0,0.5) and +(0,-0.5) ..  (p5); 
\draw (p6) .. controls +(0,-0.5) and +(0,0.5) ..  (p1); 
\draw[very thick, color=red!80!black, opacity=0.2] (p3) .. controls +(0,0.1) and +(0,0.1) ..  (p4); 
\draw[very thick, color=red!80!black] (p3) .. controls +(0,-0.1) and +(0,-0.1) ..  (p4); 
\draw[very thick, color=red!80!black] (p5) .. controls +(0,0.1) and +(0,0.1) ..  (p6); 
\draw[very thick, color=red!80!black] (p5) .. controls +(0,-0.1) and +(0,-0.1) ..  (p6); 
\end{tikzpicture}
} 
= 
\tikzzbox
{
\begin{tikzpicture}[thick,scale=1.0,color=black, baseline=0.3cm]
\coordinate (p1) at (-0.55,0);
\coordinate (p2) at (-0.2,0);
\coordinate (p3) at (-0.55,0.8);
\coordinate (p4) at (-0.2,0.8);
%
\fill [orange!23] 
(p1) -- (p2)
-- (p4)
.. controls +(0,0.1) and +(0,0.1) .. (p3)
-- (p1)
;
%
\fill [orange!38] 
(p1) .. controls +(0,0.1) and +(0,0.1) ..  (p2) -- 
(p2) .. controls +(0,-0.1) and +(0,-0.1) ..  (p1)
;
\draw (p2) -- (p4); 
\draw (p3) -- (p1); 
\draw[very thick, color=red!80!black] (p3) .. controls +(0,0.1) and +(0,0.1) ..  (p4); 
\draw[very thick, color=red!80!black, opacity=0.2] (p3) .. controls +(0,-0.1) and +(0,-0.1) ..  (p4); 
\draw[very thick, color=red!80!black] (p1) .. controls +(0,0.1) and +(0,0.1) ..  (p2); 
\draw[very thick, color=red!80!black] (p1) .. controls +(0,-0.1) and +(0,-0.1) ..  (p2); 
\end{tikzpicture}
} 
=
\tikzzbox
{
\begin{tikzpicture}[thick,scale=1.0,color=black, rotate=180, baseline=-0.5cm]
\coordinate (p1) at (-0.55,0);
\coordinate (p2) at (-0.2,0);
\coordinate (p3) at (0.2,0);
\coordinate (p4) at (0.55,0);
\coordinate (p5) at (0.175,0.8);
\coordinate (p6) at (-0.175,0.8);
%
\fill [orange!23] 
(p3) -- (p4)
-- (p2) .. controls +(0,0.35) and +(0,0.35) ..  (p3)
-- (p3) .. controls +(0,-0.4) and +(0,-0.4) ..  (p4)
-- (p4) .. controls +(0,0.5) and +(0,-0.5) ..  (p5)
-- (p5) .. controls +(0,0.1) and +(0,0.1) ..  (p6)
-- (p6) .. controls +(0,-0.5) and +(0,0.5) ..  (p1)
;
\fill [orange!38] 
(p5) .. controls +(0,0.1) and +(0,0.1) ..  (p6) -- 
(p6) .. controls +(0,-0.1) and +(0,-0.1) ..  (p5)
;
%
\fill [orange!23] 
(p1) .. controls +(0,0.1) and +(0,0.1) ..  (p2) -- 
(p2) .. controls +(0,-0.1) and +(0,-0.1) ..  (p1)
;
\draw (p2) .. controls +(0,0.35) and +(0,0.35) ..  (p3); 
\draw (p3) .. controls +(0,-0.4) and +(0,-0.4) ..  (p4); 
\draw (p4) .. controls +(0,0.5) and +(0,-0.5) ..  (p5); 
\draw (p6) .. controls +(0,-0.5) and +(0,0.5) ..  (p1); 
\draw[very thick, color=red!80!black, opacity=0.2] (p1) .. controls +(0,0.1) and +(0,0.1) ..  (p2); 
\draw[very thick, color=red!80!black] (p1) .. controls +(0,-0.1) and +(0,-0.1) ..  (p2); 
\draw[very thick, color=red!80!black] (p5) .. controls +(0,0.1) and +(0,0.1) ..  (p6); 
\draw[very thick, color=red!80!black] (p5) .. controls +(0,-0.1) and +(0,-0.1) ..  (p6); 
\end{tikzpicture}
} 
\, , \quad
\label{eq:R2}
\\
& 
\tikzzbox
{
\begin{tikzpicture}[thick,scale=1.0, color=black, baseline=0.7cm]
\coordinate (p1) at (-0.55,0);
\coordinate (p2) at (-0.2,0);
\coordinate (p3) at (0.45,0);
\coordinate (p4) at (0.8,0);
\coordinate (q1) at (0.2,1.6);
\coordinate (q2) at (0.55,1.6);
\coordinate (q3) at (1.2,1.6);
\coordinate (q4) at (1.55,1.6);
%
\fill [orange!23] 
(p1) -- (p2)
-- (p2) .. controls +(0.7,1.55) and +(-0.2,0.85) ..  (p3)
-- (p3) -- (p4)
-- (p4) .. controls +(0,0.5) and +(0,-0.5) .. (q4)
-- (q4) .. controls +(0,0.1) and +(0,0.1) .. (q3)
-- (q3) .. controls +(-0.7,-1.55) and +(0.2,-0.85) .. (q2)
-- (q2) .. controls +(0,0.1) and +(0,0.1) .. (q1)
-- (q1) .. controls +(0,-0.5) and +(0,0.5) ..  (p1)
;
%
\fill [orange!38] 
(p1) .. controls +(0,0.1) and +(0,0.1) ..  (p2) -- 
(p2) .. controls +(0,-0.1) and +(0,-0.1) ..  (p1)
;
%
\fill [orange!38] 
(p3) .. controls +(0,0.1) and +(0,0.1) ..  (p4) -- 
(p4) .. controls +(0,-0.1) and +(0,-0.1) ..  (p3)
;
\draw (p2) .. controls +(0.7,1.55) and +(-0.2,0.85) ..  (p3); 
\draw (p4) .. controls +(0,0.5) and +(0,-0.5) .. (q4); 
\draw (q3) .. controls +(-0.7,-1.55) and +(0.2,-0.85) .. (q2); 
\draw (q1) .. controls +(0,-0.5) and +(0,0.5) ..  (p1); 
\draw[very thick, color=red!80!black] (p1) .. controls +(0,0.1) and +(0,0.1) ..  (p2); 
\draw[very thick, color=red!80!black] (p2) .. controls +(0,-0.1) and +(0,-0.1) ..  (p1); 
\draw[very thick, color=red!80!black] (p3) .. controls +(0,0.1) and +(0,0.1) .. (p4); 
\draw[very thick, color=red!80!black] (p4) .. controls +(0,-0.1) and +(0,-0.1) ..  (p3); 
\draw[very thick, color=red!80!black] (q1) .. controls +(0,0.1) and +(0,0.1) ..  (q2); 
\draw[very thick, color=red!80!black, opacity=0.2] (q2) .. controls +(0,-0.1) and +(0,-0.1) ..  (q1); 
\draw[very thick, color=red!80!black] (q3) .. controls +(0,0.1) and +(0,0.1) .. (q4); 
\draw[very thick, color=red!80!black, opacity=0.2] (q4) .. controls +(0,-0.1) and +(0,-0.1) ..  (q3); 
\end{tikzpicture}
} 
= 
\tikzzbox
{
\begin{tikzpicture}[thick, scale=1.0, color=black, baseline=0.7cm]
\coordinate (p1) at (-0.55,0);
\coordinate (p2) at (-0.2,0);
\coordinate (p3) at (0.2,0);
\coordinate (p4) at (0.55,0);
\coordinate (p5) at (0.175,0.8);
\coordinate (p6) at (-0.175,0.8);
\coordinate (q1) at (-0.55,1.6);
\coordinate (q2) at (-0.2,1.6);
\coordinate (q3) at (0.2,1.6);
\coordinate (q4) at (0.55,1.6);
%
\fill [orange!23] 
(p1) -- (p2)
-- (p2) .. controls +(0,0.35) and +(0,0.35) ..  (p3)
-- (p3) -- (p4)
-- (p4) .. controls +(0,0.5) and +(0,-0.5) ..  (p5)
-- (p5) .. controls +(0,0.5) and +(0,-0.5) .. (q4)
-- (q4) .. controls +(0,0.1) and +(0,0.1) ..  (q3)
-- (q3) .. controls +(0,-0.35) and +(0,-0.35) ..  (q2)
-- (q2) .. controls +(0,0.1) and +(0,0.1) ..  (q1)
-- (q1) .. controls +(0,-0.5) and +(0,0.5) .. (p6)
-- (p6) .. controls +(0,-0.5) and +(0,0.5) ..  (p1)
;
%
\fill [orange!38] 
(p1) .. controls +(0,0.1) and +(0,0.1) ..  (p2) -- 
(p2) .. controls +(0,-0.1) and +(0,-0.1) ..  (p1)
;
%
\fill [orange!38] 
(p3) .. controls +(0,0.1) and +(0,0.1) ..  (p4) -- 
(p4) .. controls +(0,-0.1) and +(0,-0.1) ..  (p3)
;
\draw (p2) .. controls +(0,0.35) and +(0,0.35) ..  (p3); 
\draw (p4) .. controls +(0,0.5) and +(0,-0.5) .. (p5); 
\draw (p5) .. controls +(0,0.5) and +(0,-0.5) .. (q4); 
\draw (q3) .. controls +(0,-0.35) and +(0,-0.35) ..  (q2); 
\draw (q1) .. controls +(0,-0.5) and +(0,0.5) .. (p6); 
\draw (p6) .. controls +(0,-0.5) and +(0,0.5) ..  (p1); 
\draw[very thick, color=red!80!black] (p1) .. controls +(0,0.1) and +(0,0.1) ..  (p2); 
\draw[very thick, color=red!80!black] (p2) .. controls +(0,-0.1) and +(0,-0.1) ..  (p1); 
\draw[very thick, color=red!80!black] (p3) .. controls +(0,0.1) and +(0,0.1) .. (p4); 
\draw[very thick, color=red!80!black] (p4) .. controls +(0,-0.1) and +(0,-0.1) ..  (p3); 
\draw[very thick, color=red!80!black] (q1) .. controls +(0,0.1) and +(0,0.1) ..  (q2); 
\draw[very thick, color=red!80!black, opacity=0.2] (q2) .. controls +(0,-0.1) and +(0,-0.1) ..  (q1); 
\draw[very thick, color=red!80!black] (q3) .. controls +(0,0.1) and +(0,0.1) .. (q4); 
\draw[very thick, color=red!80!black, opacity=0.2] (q4) .. controls +(0,-0.1) and +(0,-0.1) ..  (q3); 
\end{tikzpicture}
} 
=
\tikzzbox
{
\begin{tikzpicture}[thick, scale=1.0, color=black, baseline=0.7cm]
\coordinate (p1) at (0.55,0);
\coordinate (p2) at (0.2,0);
\coordinate (p3) at (-0.45,0);
\coordinate (p4) at (-0.8,0);
\coordinate (q1) at (-0.2,1.6);
\coordinate (q2) at (-0.55,1.6);
\coordinate (q3) at (-1.2,1.6);
\coordinate (q4) at (-1.55,1.6);
%
\fill [orange!23] 
(p1) -- (p2)
-- (p2) .. controls +(-0.7,1.55) and +(0.2,0.85) ..  (p3)
-- (p3) -- (p4)
-- (p4) .. controls +(0,0.5) and +(0,-0.5) .. (q4)
-- (q4) .. controls +(0,0.1) and +(0,0.1) ..  (q3)
-- (q3) .. controls +(0.7,-1.55) and +(-0.2,-0.85) .. (q2)
-- (q2) .. controls +(0,0.1) and +(0,0.1) ..  (q1)
-- (q1) .. controls +(0,-0.5) and +(0,0.5) ..  (p1)
;
%
\fill [orange!38] 
(p1) .. controls +(0,0.1) and +(0,0.1) ..  (p2) -- 
(p2) .. controls +(0,-0.1) and +(0,-0.1) ..  (p1)
;
%
\fill [orange!38] 
(p3) .. controls +(0,0.1) and +(0,0.1) ..  (p4) -- 
(p4) .. controls +(0,-0.1) and +(0,-0.1) ..  (p3)
;
\draw (p2) .. controls +(-0.7,1.55) and +(0.2,0.85) ..  (p3); 
\draw (p4) .. controls +(0,0.5) and +(0,-0.5) .. (q4); 
\draw (q3) .. controls +(0.7,-1.55) and +(-0.2,-0.85) .. (q2); 
\draw (q1) .. controls +(0,-0.5) and +(0,0.5) ..  (p1); 
\draw[very thick, color=red!80!black] (p1) .. controls +(0,0.1) and +(0,0.1) ..  (p2); 
\draw[very thick, color=red!80!black] (p2) .. controls +(0,-0.1) and +(0,-0.1) ..  (p1); 
\draw[very thick, color=red!80!black] (p3) .. controls +(0,0.1) and +(0,0.1) .. (p4); 
\draw[very thick, color=red!80!black] (p4) .. controls +(0,-0.1) and +(0,-0.1) ..  (p3); 
\draw[very thick, color=red!80!black] (q1) .. controls +(0,0.1) and +(0,0.1) ..  (q2); 
\draw[very thick, color=red!80!black, opacity=0.2] (q2) .. controls +(0,-0.1) and +(0,-0.1) ..  (q1); 
\draw[very thick, color=red!80!black] (q3) .. controls +(0,0.1) and +(0,0.1) .. (q4); 
\draw[very thick, color=red!80!black, opacity=0.2] (q4) .. controls +(0,-0.1) and +(0,-0.1) ..  (q3); 
\end{tikzpicture}
} 
\, ,
\label{eq:R3}
\\
&
\label{eq:R4} 
\tikzzbox
{
\begin{tikzpicture}[thick, scale=1.0, color=black, baseline=0.3cm]
\coordinate (p1) at (-0.55,0);
\coordinate (p2) at (-0.2,0);
\coordinate (p3) at (0.2,0);
\coordinate (p4) at (0.55,0);
\coordinate (p5) at (0.175,0.8);
\coordinate (p6) at (-0.175,0.8);
\coordinate (q1) at (-0.55,-0.8);
\coordinate (q2) at (-0.2,-0.8);
\coordinate (q3) at (0.2,-0.8);
\coordinate (q4) at (0.55,-0.8);
\coordinate (q5) at (0.2,-0.8);
\coordinate (q6) at (0.55,-0.8);
%
\fill [orange!23] 
(p1) -- (p2)
-- (p2) .. controls +(0,0.35) and +(0,0.35) ..  (p3)
-- (p3) -- (p4)
-- (p4) .. controls +(0,0.5) and +(0,-0.5) ..  (p5)
-- (p5) .. controls +(0,0.1) and +(0,0.1) .. (p6)
-- (p6) .. controls +(0,-0.5) and +(0,0.5) ..  (p1)
;
\fill [orange!23] 
(q4) -- (p2) -- (p1) -- (q3);
\fill [orange!23] 
(q1) -- (p3) -- (p4) -- (q2);
%
\fill [orange!38] 
(q1) .. controls +(0,0.1) and +(0,0.1) ..  (q2) --
(q2) .. controls +(0,-0.1) and +(0,-0.1) ..  (q1);
%
\fill [orange!38] 
(q3) .. controls +(0,0.1) and +(0,0.1) ..  (q4) --
(q4) .. controls +(0,-0.1) and +(0,-0.1) ..  (q3);
\draw (q4) -- (p2);
\draw (q4) -- (p2) .. controls +(0,0.35) and +(0,0.35) ..  (p3) -- (q1); 
\draw (q2) -- (p4) .. controls +(0,0.5) and +(0,-0.5) .. (p5); 
\draw (p6) .. controls +(0,-0.5) and +(0,0.5) ..  (p1) -- (q3); 
\draw[very thick, color=red!80!black] (p5) .. controls +(0,0.1) and +(0,0.1) ..  (p6); 
\draw[very thick, color=red!80!black, opacity=0.2] (p5) .. controls +(0,-0.1) and +(0,-0.1) ..  (p6); 
\draw[very thick, color=red!80!black] (q1) .. controls +(0,0.1) and +(0,0.1) ..  (q2); 
\draw[very thick, color=red!80!black] (q1) .. controls +(0,-0.1) and +(0,-0.1) ..  (q2); 
\draw[very thick, color=red!80!black] (q3) .. controls +(0,0.1) and +(0,0.1) ..  (q4); 
\draw[very thick, color=red!80!black] (q3) .. controls +(0,-0.1) and +(0,-0.1) ..  (q4); 
\end{tikzpicture}
} 
=
\tikzzbox
{
\begin{tikzpicture}[thick, scale=1.0, color=black, baseline=0.3cm]
\coordinate (p1) at (-0.55,0);
\coordinate (p2) at (-0.2,0);
\coordinate (p3) at (0.2,0);
\coordinate (p4) at (0.55,0);
\coordinate (p5) at (0.175,0.8);
\coordinate (p6) at (-0.175,0.8);
\coordinate (q1) at (-0.55,-0.8);
\coordinate (q2) at (-0.2,-0.8);
\coordinate (q3) at (0.2,-0.8);
\coordinate (q4) at (0.55,-0.8);
%
\fill [orange!23] 
(p2) -- (q2) -- (q1) -- 
(p1) -- (p2)
-- (p2) .. controls +(0,0.35) and +(0,0.35) ..  (p3)
-- (q3)
-- (q3) -- (q4)
-- (p4) .. controls +(0,0.5) and +(0,-0.5) ..  (p5)
-- (p5) .. controls +(0,0.1) and +(0,0.1) .. (p6)
-- (p6) .. controls +(0,-0.5) and +(0,0.5) ..  (p1)
;
%
\fill [orange!38] 
(q3) .. controls +(0,0.1) and +(0,0.1) ..  (q4) --
(q4) .. controls +(0,-0.1) and +(0,-0.1) ..  (q3);
%
\fill [orange!38] 
(q1) .. controls +(0,0.1) and +(0,0.1) ..  (q2) --
(q2) .. controls +(0,-0.1) and +(0,-0.1) ..  (q1);
\draw (p2) .. controls +(0,0.35) and +(0,0.35) ..  (p3); 
\draw (p4) .. controls +(0,0.5) and +(0,-0.5) .. (p5); 
\draw (p6) .. controls +(0,-0.5) and +(0,0.5) ..  (p1); 
\draw (p1) -- (q1); 
\draw (p2) -- (q2); 
\draw (p4) -- (q4); 
\draw (p3) -- (q3); 
\draw[very thick, color=red!80!black] (p5) .. controls +(0,0.1) and +(0,0.1) ..  (p6); 
\draw[very thick, color=red!80!black, opacity=0.2] (p5) .. controls +(0,-0.1) and +(0,-0.1) ..  (p6); 
\draw[very thick, color=red!80!black] (q1) .. controls +(0,0.1) and +(0,0.1) ..  (q2); 
\draw[very thick, color=red!80!black] (q1) .. controls +(0,-0.1) and +(0,-0.1) ..  (q2); 
\draw[very thick, color=red!80!black] (q3) .. controls +(0,0.1) and +(0,0.1) ..  (q4); 
\draw[very thick, color=red!80!black] (q3) .. controls +(0,-0.1) and +(0,-0.1) ..  (q4); 
\end{tikzpicture}
}
\, .
\end{align}
Accordingly, the set of relations $G_2$ is comprised precisely of the pairs of morphisms related by an equality above. 
The key property of $\Bord_2$ now is that it satisfies the universal property in Definition~\ref{def:freesymmon}: 
	
\begin{theorem}\label{thm:Bord2-freegen}
$\Bord_2$ is freely generated as a symmetric monoidal category by $G_0, G_1, G_2$ as above. 
\end{theorem}

The reader may have noticed that \eqref{eq:R1}--\eqref{eq:R4} contains some redundant relations. 
For example, using \eqref{eq:R4} it is enough to remember only one of each of the two equalities in \eqref{eq:R2} and \eqref{eq:R3}. 
With a little more fiddling, also both equalities in \eqref{eq:R1} can be omitted. 
There is no unique set of generators and relations, not even a unique minimal one. 
The above choice is made with hindsight to fit nicely to the algebraic description we will now turn to.

\medskip

With $\Bord_2$ under control, we turn to the algebraic description of a 2-dimensional TQFT 
\be 
\zz\colon \Bord_2 \lra \Vectk \, . 
\ee 
Due to the characterisation of $\Bord_2$ via the universal property in Definition~\ref{def:freesymmon}, a symmetric monoidal functor out of $\Bord_2$ is determined uniquely up to natural monoidal isomorphism by a choice of object for each element of $G_0$, an appropriate choice of morphism for each element of $G_1$, with the condition that all relations from $G_2$ are respected. 
This amounts to identifying the algebraic structure $\textrm{comFrob}_\Bbbk$, which is analogous to the category of dual pairs $\mathcal{DP}_\Bbbk$ that solved the corresponding problem in the 1-dimensional case. 

Let us look at the data defining the symmetric monoidal functor $\zz$ out of $\Bord_2$ in more detail.
First of all, we need a $\Bbbk$-vector space for the unique generator in~$G_0$, which we will abbreviate as
\be \label{eq:equiv-to-cFA-S1}
A := \zz (S^1) \, .
\ee 
Next we need linear maps for the generators in $G_1$, and we write these as
\begin{align}
\mu &= 
\zz \bigg(
\tikzzbox
{
\begin{tikzpicture}[thick,scale=1.0,color=black, baseline=0.3cm]
\coordinate (p1) at (-0.55,0);
\coordinate (p2) at (-0.2,0);
\coordinate (p3) at (0.2,0);
\coordinate (p4) at (0.55,0);
\coordinate (p5) at (0.175,0.8);
\coordinate (p6) at (-0.175,0.8);
%
\fill [orange!23] 
(p1) .. controls +(0,-0.15) and +(0,-0.15) ..  (p2)
-- (p2) .. controls +(0,0.35) and +(0,0.35) ..  (p3)
-- (p3) .. controls +(0,-0.15) and +(0,-0.15) ..  (p4)
-- (p4) .. controls +(0,0.5) and +(0,-0.5) ..  (p5)
-- (p5) .. controls +(0,0.15) and +(0,0.15) ..  (p6)
-- (p6) .. controls +(0,-0.5) and +(0,0.5) ..  (p1)
;
\fill [orange!38] 
(p1) .. controls +(0,-0.15) and +(0,-0.15) ..  (p2)
-- (p2) .. controls +(0,0.15) and +(0,0.15) ..  (p1)
;
\fill [orange!38] 
(p3) .. controls +(0,-0.15) and +(0,-0.15) ..  (p4)
-- (p4) .. controls +(0,0.15) and +(0,0.15) ..  (p3)
;
\draw (p2) .. controls +(0,0.35) and +(0,0.35) ..  (p3); 
\draw (p4) .. controls +(0,0.5) and +(0,-0.5) ..  (p5); 
\draw (p6) .. controls +(0,-0.5) and +(0,0.5) ..  (p1); 
\draw[very thick, red!80!black] (p1) .. controls +(0,0.15) and +(0,0.15) ..  (p2); 
\draw[very thick, red!80!black] (p1) .. controls +(0,-0.15) and +(0,-0.15) ..  (p2); 
\draw[very thick, red!80!black] (p3) .. controls +(0,0.15) and +(0,0.15) ..  (p4); 
\draw[very thick, red!80!black] (p3) .. controls +(0,-0.15) and +(0,-0.15) ..  (p4); 
\draw[very thick, red!80!black] (p5) .. controls +(0,0.15) and +(0,0.15) ..  (p6); 
\draw[very thick, red!80!black, opacity=0.2] (p5) .. controls +(0,-0.15) and +(0,-0.15) ..  (p6); 
\end{tikzpicture}
}
\bigg)\colon A \otimes_\Bbbk A \lra A
\, , \quad
&\eta &= 
\zz \bigg(
\tikzzbox
{
\begin{tikzpicture}[thick,scale=1.0,color=black, baseline=-0.15cm]
\coordinate (p1) at (-0.55,0);
\coordinate (p2) at (-0.2,0);
%
\fill [orange!23] 
(p1) .. controls +(0,0.15) and +(0,0.15) ..  (p2)
-- (p2) .. controls +(0,-0.4) and +(0,-0.4) ..  (p1)
;
\draw (p1) .. controls +(0,-0.4) and +(0,-0.4) ..  (p2); 
\draw[very thick, red!80!black] (p1) .. controls +(0,0.15) and +(0,0.15) ..  (p2); 
\draw[very thick, red!80!black, opacity=0.2] (p1) .. controls +(0,-0.15) and +(0,-0.15) ..  (p2); 
\end{tikzpicture}
}
\bigg)\colon \Bbbk \lra A \, ,
\nonumber \\
\Delta &= 
\zz \bigg(
\tikzzbox
{
\begin{tikzpicture}[thick,scale=1.0,color=black, rotate=180, baseline=-0.5cm]
\coordinate (p1) at (-0.55,0);
\coordinate (p2) at (-0.2,0);
\coordinate (p3) at (0.2,0);
\coordinate (p4) at (0.55,0);
\coordinate (p5) at (0.175,0.8);
\coordinate (p6) at (-0.175,0.8);
%
\fill [orange!23] 
(p1) .. controls +(0,-0.15) and +(0,-0.15) ..  (p2)
-- (p2) .. controls +(0,0.35) and +(0,0.35) ..  (p3)
-- (p3) .. controls +(0,-0.15) and +(0,-0.15) ..  (p4)
-- (p4) .. controls +(0,0.5) and +(0,-0.5) ..  (p5)
-- (p5) .. controls +(0,0.15) and +(0,0.15) ..  (p6)
-- (p6) .. controls +(0,-0.5) and +(0,0.5) ..  (p1)
;
\fill [orange!38] 
(p6) .. controls +(0,-0.15) and +(0,-0.15) ..  (p5)
-- (p5) .. controls +(0,0.15) and +(0,0.15) ..  (p6)
;
\draw (p2) .. controls +(0,0.35) and +(0,0.35) ..  (p3); 
\draw (p4) .. controls +(0,0.5) and +(0,-0.5) ..  (p5); 
\draw (p6) .. controls +(0,-0.5) and +(0,0.5) ..  (p1); 
\draw[very thick, red!80!black, opacity=0.2] (p1) .. controls +(0,0.15) and +(0,0.15) ..  (p2); 
\draw[very thick, red!80!black] (p1) .. controls +(0,-0.15) and +(0,-0.15) ..  (p2); 
\draw[very thick, red!80!black, opacity=0.2] (p3) .. controls +(0,0.15) and +(0,0.15) ..  (p4); 
\draw[very thick, red!80!black] (p3) .. controls +(0,-0.15) and +(0,-0.15) ..  (p4); 
\draw[very thick, red!80!black] (p5) .. controls +(0,0.15) and +(0,0.15) ..  (p6); 
\draw[very thick, red!80!black] (p5) .. controls +(0,-0.15) and +(0,-0.15) ..  (p6); 
\end{tikzpicture}
}
\bigg)\colon A \lra A \otimes_\Bbbk A 
\, , \quad
& \eps &= 
\zz \bigg(
\tikzzbox
{
\begin{tikzpicture}[thick,scale=1.0,color=black, rotate=180, baseline=-0cm]
\coordinate (p1) at (-0.55,0);
\coordinate (p2) at (-0.2,0);
%
\fill [orange!23] 
(p1) .. controls +(0,0.15) and +(0,0.15) ..  (p2)
-- (p2) .. controls +(0,-0.4) and +(0,-0.4) ..  (p1)
;
\fill [orange!38] 
(p1) .. controls +(0,-0.15) and +(0,-0.15) ..  (p2)
-- (p2) .. controls +(0,0.15) and +(0,0.15) ..  (p1)
;
\draw (p1) .. controls +(0,-0.4) and +(0,-0.4) ..  (p2); 
\draw[very thick, red!80!black] (p1) .. controls +(0,0.15) and +(0,0.15) ..  (p2); 
\draw[very thick, red!80!black] (p1) .. controls +(0,-0.15) and +(0,-0.15) ..  (p2); 
\end{tikzpicture}
}
\bigg)\colon A \lra \Bbbk \ .
\label{eq:equiv-to-cFA-bord}
\end{align}

Finally, we need to impose the relations. If we think of $\mu$ as a product and write $a \cdot b$ instead of $\mu(a \otimes b)$, then by~\eqref{eq:R1} this product is associative,
\begin{align*}
(a \cdot b ) \cdot c 
 = 
\zz \!\left(\!
\tikzzbox
{
\begin{tikzpicture}[thick,scale=1.0,color=black, baseline=0.75cm]
\coordinate (p1) at (-0.55,0);
\coordinate (p2) at (-0.2,0);
\coordinate (p3) at (0.2,0);
\coordinate (p4) at (0.55,0);
\coordinate (p5) at (0.175,0.8);
\coordinate (p6) at (-0.175,0.8);
\coordinate (p7) at (0.95,0);
\coordinate (p8) at (1.3,0);
\coordinate (p9) at (0.2,1.6);
\coordinate (p10) at (0.55,1.6);
%
\fill [orange!23] 
(p1) .. controls +(0,0.1) and +(0,0.1) ..  (p2)
-- (p2) .. controls +(0,0.65) and +(0,0.65) ..  (p3)
-- (p3) -- (p4)
-- (p4) .. controls +(-0.55,1.95) and +(-0.1,0.5) ..  (p7)
-- (p7) -- (p8)
-- (p8) .. controls +(0,0.5) and +(0,-0.5) ..  (p10)
-- (p10) .. controls +(0,0.1) and +(0,0.1) ..  (p9)
-- (p9) .. controls +(0,-0.5) and +(0,0.75) ..  (p1)
;
%
\fill [orange!38] 
(p1) .. controls +(0,0.1) and +(0,0.1) ..  (p2) -- 
(p2) .. controls +(0,-0.1) and +(0,-0.1) ..  (p1)
;
%
\fill [orange!38] 
(p3) .. controls +(0,0.1) and +(0,0.1) ..  (p4) -- 
(p4) .. controls +(0,-0.1) and +(0,-0.1) ..  (p3)
;
%
\fill [orange!38] 
(p7) .. controls +(0,0.1) and +(0,0.1) ..  (p8) -- 
(p8) .. controls +(0,-0.1) and +(0,-0.1) ..  (p7)
;
\draw (p4) .. controls +(-0.55,1.95) and +(-0.1,0.5) ..  (p7);
\draw (p2) .. controls +(0,0.65) and +(0,0.65) ..  (p3); 
\draw (p8) .. controls +(0,0.5) and +(0,-0.5) ..  (p10); 
\draw (p9) .. controls +(0,-0.5) and +(0,0.75) ..  (p1); 
\draw[very thick, color=red!80!black] (p1) .. controls +(0,0.1) and +(0,0.1) ..  (p2); 
\draw[very thick, color=red!80!black] (p2) .. controls +(0,-0.1) and +(0,-0.1) ..  (p1); 
\draw[very thick, color=red!80!black] (p3) .. controls +(0,0.1) and +(0,0.1) .. (p4); 
\draw[very thick, color=red!80!black] (p4) .. controls +(0,-0.1) and +(0,-0.1) ..  (p3); 
\draw[very thick, color=red!80!black] (p7) .. controls +(0,0.1) and +(0,0.1) .. (p8); 
\draw[very thick, color=red!80!black] (p8) .. controls +(0,-0.1) and +(0,-0.1) ..  (p7); 
\draw[very thick, color=red!80!black] (p9) .. controls +(0,0.1) and +(0,0.1) .. (p10); 
\draw[very thick, color=red!80!black, opacity=0.2] (p10) .. controls +(0,-0.1) and +(0,-0.1) ..  (p9); 
\end{tikzpicture}
} \!
\right) \!(a \otimes b \otimes c)
 = 
\zz \!\left(\!
\tikzzbox
{
\begin{tikzpicture}[thick,scale=1.0,color=black, baseline=0.75cm]
\coordinate (p1) at (-0.55,0);
\coordinate (p2) at (-0.2,0);
\coordinate (p3) at (0.2,0);
\coordinate (p4) at (0.55,0);
\coordinate (p5) at (0.175,0.8);
\coordinate (p6) at (-0.175,0.8);
\coordinate (p7) at (-0.95,0);
\coordinate (p8) at (-1.3,0);
\coordinate (p9) at (-0.2,1.6);
\coordinate (p10) at (-0.55,1.6);
%
\fill[orange!23] 
(p8) -- (p7)
-- (p7) .. controls +(0.1,0.5) and +(0.55,1.95) ..  (p1)
-- (p1) -- (p2)
-- (p2) .. controls +(0,0.65) and +(0,0.65) ..  (p3)
-- (p3) -- (p4)
-- (p4) .. controls +(0,0.75) and +(0,-0.5) ..  (p9)
-- (p9) .. controls +(0,0.1) and +(0,0.1) .. (p10)
-- (p10) .. controls +(0,-0.5) and +(0,0.5) ..  (p8)
;
%
%
\fill [orange!38] 
(p1) .. controls +(0,0.1) and +(0,0.1) ..  (p2) -- 
(p2) .. controls +(0,-0.1) and +(0,-0.1) ..  (p1)
;
%
\fill [orange!38] 
(p3) .. controls +(0,0.1) and +(0,0.1) ..  (p4) -- 
(p4) .. controls +(0,-0.1) and +(0,-0.1) ..  (p3)
;
%
\fill [orange!38] 
(p7) .. controls +(0,0.1) and +(0,0.1) ..  (p8) -- 
(p8) .. controls +(0,-0.1) and +(0,-0.1) ..  (p7)
;
\draw (p7) .. controls +(0.1,0.5) and +(0.55,1.95) ..  (p1); 
\draw (p2) .. controls +(0,0.65) and +(0,0.65) ..  (p3); 
\draw (p4) .. controls +(0,0.75) and +(0,-0.5) ..  (p9); 
\draw (p10) .. controls +(0,-0.5) and +(0,0.5) ..  (p8); 
\draw[very thick, color=red!80!black] (p1) .. controls +(0,0.1) and +(0,0.1) ..  (p2); 
\draw[very thick, color=red!80!black] (p2) .. controls +(0,-0.1) and +(0,-0.1) ..  (p1); 
\draw[very thick, color=red!80!black] (p3) .. controls +(0,0.1) and +(0,0.1) .. (p4); 
\draw[very thick, color=red!80!black] (p4) .. controls +(0,-0.1) and +(0,-0.1) ..  (p3); 
\draw[very thick, color=red!80!black] (p7) .. controls +(0,0.1) and +(0,0.1) .. (p8); 
\draw[very thick, color=red!80!black] (p8) .. controls +(0,-0.1) and +(0,-0.1) ..  (p7); 
\draw[very thick, color=red!80!black] (p9) .. controls +(0,0.1) and +(0,0.1) .. (p10); 
\draw[very thick, color=red!80!black, opacity=0.2] (p10) .. controls +(0,-0.1) and +(0,-0.1) ..  (p9); 
\end{tikzpicture}
} \!
\right) \! (a \otimes b \otimes c)
 = a \cdot ( b  \cdot c ) \, ,
\end{align*}
and it has a unit, namely $u_A := \zz(\tikzzbox
{
\begin{tikzpicture}[thick,scale=0.45,color=black, baseline=-0.15cm]
\coordinate (p1) at (-0.55,0);
\coordinate (p2) at (-0.2,0);
%
\fill [orange!23] 
(p1) .. controls +(0,0.15) and +(0,0.15) ..  (p2)
-- (p2) .. controls +(0,-0.4) and +(0,-0.4) ..  (p1)
;
\draw (p1) .. controls +(0,-0.4) and +(0,-0.4) ..  (p2); 
\draw[thick, red!80!black] (p1) .. controls +(0,0.15) and +(0,0.15) ..  (p2); 
\draw[thick, red!80!black, opacity=0.2] (p1) .. controls +(0,-0.15) and +(0,-0.15) ..  (p2); 
\end{tikzpicture}
})(1)
$. Indeed, by  relation~\eqref{eq:R2} we have
\be 
u_A \cdot a = 
\zz(\tikzzbox
{
\begin{tikzpicture}[thick,scale=1.0,color=black, baseline=-0.15cm]
\coordinate (p1) at (-0.55,0);
\coordinate (p2) at (-0.2,0);
%
\fill [orange!23] 
(p1) .. controls +(0,0.15) and +(0,0.15) ..  (p2)
-- (p2) .. controls +(0,-0.4) and +(0,-0.4) ..  (p1)
;
\draw (p1) .. controls +(0,-0.4) and +(0,-0.4) ..  (p2); 
\draw[thick, red!80!black] (p1) .. controls +(0,0.15) and +(0,0.15) ..  (p2); 
\draw[thick, red!80!black, opacity=0.2] (p1) .. controls +(0,-0.15) and +(0,-0.15) ..  (p2); 
\end{tikzpicture}
})(1)
\cdot 
\zz\bigg(
\tikzzbox
{
\begin{tikzpicture}[thick,scale=1.0,color=black, baseline=0.3cm]
\coordinate (p1) at (-0.55,0);
\coordinate (p2) at (-0.2,0);
\coordinate (p3) at (-0.55,0.8);
\coordinate (p4) at (-0.2,0.8);
%
\fill [orange!23] 
(p1) -- (p2)
-- (p4)
.. controls +(0,0.1) and +(0,0.1) .. (p3)
-- (p1)
;
%
\fill [orange!38] 
(p1) .. controls +(0,0.1) and +(0,0.1) ..  (p2) -- 
(p2) .. controls +(0,-0.1) and +(0,-0.1) ..  (p1)
;
\draw (p2) -- (p4); 
\draw (p3) -- (p1); 
\draw[very thick, color=red!80!black] (p3) .. controls +(0,0.1) and +(0,0.1) ..  (p4); 
\draw[very thick, color=red!80!black, opacity=0.2] (p3) .. controls +(0,-0.1) and +(0,-0.1) ..  (p4); 
\draw[very thick, color=red!80!black] (p1) .. controls +(0,0.1) and +(0,0.1) ..  (p2); 
\draw[very thick, color=red!80!black] (p1) .. controls +(0,-0.1) and +(0,-0.1) ..  (p2); 
\end{tikzpicture}
} 
\bigg)
(a)
= 
\zz\bigg(
\tikzzbox
{
\begin{tikzpicture}[thick,scale=1.0,color=black, baseline=0.3cm]
\coordinate (p1) at (-0.55,0);
\coordinate (p2) at (-0.2,0);
\coordinate (p3) at (0.2,0);
\coordinate (p4) at (0.55,0);
\coordinate (p5) at (0.175,0.8);
\coordinate (p6) at (-0.175,0.8);
%
\fill [orange!23] 
(p1) .. controls +(0,-0.4) and +(0,-0.4) ..  (p2)
-- (p2) .. controls +(0,0.35) and +(0,0.35) ..  (p3)
-- (p3) -- (p4)
-- (p4) .. controls +(0,0.5) and +(0,-0.5) ..  (p5)
-- (p5) .. controls +(0,0.1) and +(0,0.1) .. (p6)
-- (p6) .. controls +(0,-0.5) and +(0,0.5) ..  (p1)
;
%
\fill [orange!38] 
(p3) .. controls +(0,0.1) and +(0,0.1) ..  (p4) -- 
(p4) .. controls +(0,-0.1) and +(0,-0.1) ..  (p3)
;
\draw (p1) .. controls +(0,-0.4) and +(0,-0.4) ..  (p2); 
\draw (p2) .. controls +(0,0.35) and +(0,0.35) ..  (p3); 
\draw (p4) .. controls +(0,0.5) and +(0,-0.5) ..  (p5); 
\draw (p6) .. controls +(0,-0.5) and +(0,0.5) ..  (p1); 
\draw[very thick, color=red!80!black] (p3) .. controls +(0,0.1) and +(0,0.1) ..  (p4); 
\draw[very thick, color=red!80!black] (p3) .. controls +(0,-0.1) and +(0,-0.1) ..  (p4); 
\draw[very thick, color=red!80!black] (p5) .. controls +(0,0.1) and +(0,0.1) ..  (p6); 
\draw[very thick, color=red!80!black, opacity=0.2] (p5) .. controls +(0,-0.1) and +(0,-0.1) ..  (p6); 
\end{tikzpicture}
} 
\bigg)
(a)
= 
\zz\bigg(
\tikzzbox
{
\begin{tikzpicture}[thick,scale=1.0,color=black, baseline=0.3cm]
\coordinate (p1) at (-0.55,0);
\coordinate (p2) at (-0.2,0);
\coordinate (p3) at (-0.55,0.8);
\coordinate (p4) at (-0.2,0.8);
%
\fill [orange!23] 
(p1) -- (p2)
-- (p4)
.. controls +(0,0.1) and +(0,0.1) .. (p3)
-- (p1)
;
%
\fill [orange!38] 
(p1) .. controls +(0,0.1) and +(0,0.1) ..  (p2) -- 
(p2) .. controls +(0,-0.1) and +(0,-0.1) ..  (p1)
;
\draw (p2) -- (p4); 
\draw (p3) -- (p1); 
\draw[very thick, color=red!80!black] (p3) .. controls +(0,0.1) and +(0,0.1) ..  (p4); 
\draw[very thick, color=red!80!black, opacity=0.2] (p3) .. controls +(0,-0.1) and +(0,-0.1) ..  (p4); 
\draw[very thick, color=red!80!black] (p1) .. controls +(0,0.1) and +(0,0.1) ..  (p2); 
\draw[very thick, color=red!80!black] (p1) .. controls +(0,-0.1) and +(0,-0.1) ..  (p2); 
\end{tikzpicture}
} 
\bigg)
(a)
=
a \, ,
\ee 
and $a \cdot u_A = a$ follows analogously. 
So far we only used functoriality and monoidality of $\zz$. Symmetry of $\zz$, i.\,e.\ that $\zz$ maps the symmetric braiding in $\Bord_2$ to the symmetric braiding in $\Vectk$, is used to show that \eqref{eq:R4} implies commutativity of $A$:
\be
a\cdot b
= 
\zz \!\left(
\tikzzbox
{
\begin{tikzpicture}[thick, scale=1.0, color=black, baseline=0cm]
\coordinate (p1) at (-0.55,0);
\coordinate (p2) at (-0.2,0);
\coordinate (p3) at (0.2,0);
\coordinate (p4) at (0.55,0);
\coordinate (p5) at (0.175,0.8);
\coordinate (p6) at (-0.175,0.8);
\coordinate (q1) at (-0.55,-0.8);
\coordinate (q2) at (-0.2,-0.8);
\coordinate (q3) at (0.2,-0.8);
\coordinate (q4) at (0.55,-0.8);
%
\fill [orange!23] 
(p2) -- (q2) -- (q1) -- 
(p1) -- (p2)
-- (p2) .. controls +(0,0.35) and +(0,0.35) ..  (p3)
-- (q3)
-- (q3) -- (q4)
-- (p4) .. controls +(0,0.5) and +(0,-0.5) ..  (p5)
-- (p5) .. controls +(0,0.1) and +(0,0.1) .. (p6)
-- (p6) .. controls +(0,-0.5) and +(0,0.5) ..  (p1)
;
%
\fill [orange!38] 
(q3) .. controls +(0,0.1) and +(0,0.1) ..  (q4) --
(q4) .. controls +(0,-0.1) and +(0,-0.1) ..  (q3);
%
\fill [orange!38] 
(q1) .. controls +(0,0.1) and +(0,0.1) ..  (q2) --
(q2) .. controls +(0,-0.1) and +(0,-0.1) ..  (q1);
\draw (p2) .. controls +(0,0.35) and +(0,0.35) ..  (p3); 
\draw (p4) .. controls +(0,0.5) and +(0,-0.5) .. (p5); 
\draw (p6) .. controls +(0,-0.5) and +(0,0.5) ..  (p1); 
\draw (p1) -- (q1); 
\draw (p2) -- (q2); 
\draw (p4) -- (q4); 
\draw (p3) -- (q3); 
\draw[very thick, color=red!80!black] (p5) .. controls +(0,0.1) and +(0,0.1) ..  (p6); 
\draw[very thick, color=red!80!black, opacity=0.2] (p5) .. controls +(0,-0.1) and +(0,-0.1) ..  (p6); 
\draw[very thick, color=red!80!black] (q1) .. controls +(0,0.1) and +(0,0.1) ..  (q2); 
\draw[very thick, color=red!80!black] (q1) .. controls +(0,-0.1) and +(0,-0.1) ..  (q2); 
\draw[very thick, color=red!80!black] (q3) .. controls +(0,0.1) and +(0,0.1) ..  (q4); 
\draw[very thick, color=red!80!black] (q3) .. controls +(0,-0.1) and +(0,-0.1) ..  (q4); 
\end{tikzpicture}
}
\right) \!
(a \otimes b)
= 
\zz \!\left(
\tikzzbox
{
\begin{tikzpicture}[thick, scale=1.0, color=black, baseline=0cm]
\coordinate (p1) at (-0.55,0);
\coordinate (p2) at (-0.2,0);
\coordinate (p3) at (0.2,0);
\coordinate (p4) at (0.55,0);
\coordinate (p5) at (0.175,0.8);
\coordinate (p6) at (-0.175,0.8);
\coordinate (q1) at (-0.55,-0.8);
\coordinate (q2) at (-0.2,-0.8);
\coordinate (q3) at (0.2,-0.8);
\coordinate (q4) at (0.55,-0.8);
\coordinate (q5) at (0.2,-0.8);
\coordinate (q6) at (0.55,-0.8);
%
\fill [orange!23] 
(p1) -- (p2)
-- (p2) .. controls +(0,0.35) and +(0,0.35) ..  (p3)
-- (p3) -- (p4)
-- (p4) .. controls +(0,0.5) and +(0,-0.5) ..  (p5)
-- (p5) .. controls +(0,0.1) and +(0,0.1) .. (p6)
-- (p6) .. controls +(0,-0.5) and +(0,0.5) ..  (p1)
;
\fill [orange!23] 
(q4) -- (p2) -- (p1) -- (q3);
\fill [orange!23] 
(q1) -- (p3) -- (p4) -- (q2);
%
\fill [orange!38] 
(q1) .. controls +(0,0.1) and +(0,0.1) ..  (q2) --
(q2) .. controls +(0,-0.1) and +(0,-0.1) ..  (q1);
%
\fill [orange!38] 
(q3) .. controls +(0,0.1) and +(0,0.1) ..  (q4) --
(q4) .. controls +(0,-0.1) and +(0,-0.1) ..  (q3);
\draw (q4) -- (p2);
\draw (q4) -- (p2) .. controls +(0,0.35) and +(0,0.35) ..  (p3) -- (q1); 
\draw (q2) -- (p4) .. controls +(0,0.5) and +(0,-0.5) .. (p5); 
\draw (p6) .. controls +(0,-0.5) and +(0,0.5) ..  (p1) -- (q3); 
\draw[very thick, color=red!80!black] (p5) .. controls +(0,0.1) and +(0,0.1) ..  (p6); 
\draw[very thick, color=red!80!black, opacity=0.2] (p5) .. controls +(0,-0.1) and +(0,-0.1) ..  (p6); 
\draw[very thick, color=red!80!black] (q1) .. controls +(0,0.1) and +(0,0.1) ..  (q2); 
\draw[very thick, color=red!80!black] (q1) .. controls +(0,-0.1) and +(0,-0.1) ..  (q2); 
\draw[very thick, color=red!80!black] (q3) .. controls +(0,0.1) and +(0,0.1) ..  (q4); 
\draw[very thick, color=red!80!black] (q3) .. controls +(0,-0.1) and +(0,-0.1) ..  (q4); 
\end{tikzpicture}
}
\right) \!
(a \otimes b)
= 
b \cdot a
\, . 
\ee

We can carry on with this exercise, but to avoid repetition, let us define the notion of a Frobenius algebra right away:
\begin{definition}
\label{def:Frob}
A \textsl{Frobenius algebra over~$\Bbbk$} is a $\Bbbk$-vector space $A$ with
\begin{itemize}
\item
an associative unital algebra structure $(A, \mu\colon A\otimes_\Bbbk A \to A,  \eta\colon \Bbbk \to A)$, i.\,e. 
$
\mu\circ(\mu \otimes \id) = \mu \circ (\id \otimes \mu)
$
and 
$
\mu \circ (\eta \otimes \id) = \id = \mu \circ (\id \otimes \eta)
$, 
\item
a coassociative counital coalgebra structure $(A, \Delta\colon A \to A\otimes_\Bbbk A,  \eps\colon A \to \Bbbk)$, i.\,e. 
$
(\Delta \otimes \id) \circ \Delta =  (\id \otimes \Delta) \circ \Delta
$
and 
$
(\eps \otimes \id) \circ \Delta = \id =  (\id \otimes \eps) \circ \Delta
$, 
\item
such that 
$
(\mu \otimes \id) \circ (\id \otimes \Delta) = \Delta \circ \mu = (\id \otimes \mu) \circ ( \Delta \otimes \id)
$. 
\end{itemize}
\end{definition}

A \textsl{morphism of Frobenius algebras} $\psi\colon A \to A'$ 
is a $\Bbbk$-linear map which is simultaneously an algebra map and a coalgebra map, i.\,e. 
\be 
\mu' \circ (\psi \otimes \psi ) = \psi \circ \mu
\, , \quad 
\eta' = \psi \circ \eta
\, , \quad 
(\psi \otimes \psi) \circ \Delta = \Delta' \circ \psi
\, , \quad 
\eps = \eps' \circ \psi
\, . 
\ee 
Frobenius algebras and their morphisms 
form a category, and we denote the full subcategory of commutative Frobenius algebras by $\textrm{comFrob}_\Bbbk$.
Now we can summarise our analysis of the generators and relations as follows:

\begin{proposition}\label{prop:genrel-cFA}
The assignments \eqref{eq:equiv-to-cFA-S1} and \eqref{eq:equiv-to-cFA-bord} define an equivalence of categories $\mathrm{Fun}_{\otimes,\mathrm{sym}}\big( \,\mathcal{F}(G_0,G_1,G_2) \, , \,\Vectk \,\big) \lra \textrm{comFrob}_\Bbbk$.
\end{proposition}

Combining this proposition with Theorem \ref{thm:Bord2-freegen} proves Theorem~\ref{thm:2dTQFT}. 

Note that -- as for $\mathcal{DP}_\Bbbk$ -- by Lemma \ref{lem:sym-mon-inv}, 
Proposition \ref{prop:genrel-cFA} implies that morphisms between Frobenius algebras are invertible. 

\begin{remark}\label{rem:general-target}
The proof of Theorem \ref{thm:2dTQFT} generalises to symmetric monoidal functors $\zz\colon \Bord_2 \to \mathcal C$ for any symmetric monoidal category~$\mathcal C$, giving an equivalence of groupoids of such~$\zz$ and commutative Frobenius algebras internal to~$\mathcal C$. 
\end{remark}

While the definition of a Frobenius algebra given above is nicely symmetric and fits perfectly on the generators-and-relations formulation of $\Bord_2$, it is not very economic. Indeed, the original definition of a Frobenius algebra was different:

\medskip

\noindent
\textbf{Definition \ref{def:Frob}} (economy version)\textbf{.} A \textsl{Frobenius algebra over~$\Bbbk$} is a unital associative $\Bbbk$-algebra together with a nondegenerate and invariant bilinear form $\langle -,- \rangle\colon A \times A \to \Bbbk$.
 
\medskip

Here, invariant means that $\langle a \cdot b,c \rangle = \langle a, b \cdot c \rangle$ for all $a,b,c \in A$. 
We have

\begin{proposition}
Definition \ref{def:Frob} and its economy version are equivalent.
\end{proposition}

\begin{proof}[Sketch of proof]
We write $u_A := \eta(1)$ for the unit of $A$.

To arrive from the full version at the economy version, one takes the pairing $\langle a,b \rangle := \eps(a \cdot b)$. This is clearly invariant. To see nondegeneracy, make the ansatz $c := \Delta(u_A) \in A \otimes_\Bbbk A$ and verify that it is the copairing for $\langle -,- \rangle$.

To upgrade from economy to the full version, one needs to come up with a coproduct and a counit. 
The counit is easy, just set $\eps(a) := \langle a ,u_A\rangle$. 
For the coproduct, let $c \in A \otimes_\Bbbk A$ be the (uniquely defined) copairing for $\langle -,- \rangle$. The ansatz for the coproduct is $\Delta(a) := (\mu \otimes \id)(a \otimes c)$. 
It is a straightforward but tedious calculation to verify all the properties of the full version of Definition \ref{def:Frob}.
For more details, see e.\,g.~\cite[Ch.\,2]{Kockbook}.
\end{proof}

There is one merit in the full version of Definition \ref{def:Frob}, and that is that one is immediately led to the right notion of morphisms between Frobenius algebras (right for the application to 2-dimensional TQFT, that is). The economy version would suggest algebra homomorphisms which preserve the pairing as morphisms, but this would still allow injective maps which are not bijective, and we would not get the equivalence in Proposition \ref{prop:genrel-cFA} with this notion of Frobenius algebra morphisms. 

\medskip

We conclude our discussion of 2d TQFTs with a list of examples of Frobenius algebras.

\begin{examples}
\label{ex:2dTQFT}
\begin{enumerate}
\item
Let~$G$ be a finite group. 
Its group algebra $A := \Bbbk G$ together with the pairing defined on basis elements as
\be
\big\langle g,h \big\rangle
= 
\delta_{g,h^{-1}}
\ee
and linearly extended to~$A$ is a Frobenius algebra. 
It is commutative iff~$G$ is abelian. 
\item
Let $B$ be a semisimple $\Bbbk$-algebra. By the Artin-Wedderburn Theorem, every such algebra
is isomorphic to a direct sum of matrix algebras over division algebras~$D_i$ of finite dimension over~$\Bbbk$, 
\be
B \cong \bigoplus_{i\in I}\textrm{Mat}_{n_i\times n_i}(D_i) \, . 
\ee
If $\Bbbk = \C$ then $D_i = \C$ for all $i\in I$, and 
\be
\big\langle M,M' \big\rangle
= 
\operatorname{Tr}(MM')
\ee
gives a nondegenerate pairing that makes~$B$ a Frobenius algebra. 
Its centre $A := Z(B)$ is a commutative Frobenius algebra and it describes (via Theorem \ref{thm:2dTQFT}) the ``state-sum model'' of a 2-dimensional TQFT build from $B$ \cite{BP1992, FHK, lp0602047}.
\item
Let the zero set of a polynomial $W \in \C[x_1,\ldots,x_N]$ describe an isolated singularity, i.\,e.~the quotient
\be
A:= 
\C[x_1,\ldots,x_N]/(\partial_{x_1}W, \ldots, \partial_{x_N}W)
\ee
is finite-dimensional over~$\C$. 
Multiplication of polynomials induces a commutative $\C$-algebra structure on~$A$, which is typically non-semisimple (for example if $W = x^d$ and $d\in \Z_{\geqslant 4}$). 
Together with the pairing
\be
\label{eq:res}
\big\langle \phi,\psi \big\rangle
= 
\Res \left[ \frac{\phi(x) \cdot \psi(x) \, \operatorname{d}\! x}{\partial_{x_1} W, \ldots, \partial_{x_N} W} \right] ,
\ee
$A$ is a commutative Frobenius algebra. 
The TQFTs associated to such algebras are called ``affine Landau-Ginzburg models'',
see \cite{v1991, hl0404} and \cite{cm1208.1481} which uses the same notation as in \eqref{eq:res}.
\item 
For a real compact oriented $d$-dimensional manifold~$X$, let 
\be
\label{eq:sigmamodelA}
A := \bigoplus_{p=0}^d H^p_{\textrm{dR}}(X)
\ee
denote its de Rham cohomology. 
Together with the wedge product, $A$ is a unital associative algebra. 
Furthermore, Poincar\'{e} duality says that
\be
\big\langle \alpha,\beta \big\rangle
= 
\int_X \alpha \wedge \beta
\ee
is a nondegenerate pairing, giving~$A$ the structure of a Frobenius algebra. 

However, since $\alpha \wedge \beta = (-1)^{pq} \beta \wedge \alpha$ for $\alpha \in H^p_{\textrm{dR}}(X)$ and $\beta \in H^q_{\textrm{dR}}(X)$, the algebra~$A$ is only \textsl{graded commutative} when viewed as an object in $\Vect_\R$. 
On the other hand, viewed as an object in the category of $\Z_2$-graded vector spaces $\Vect_\R^{\Z_2}$ (or of $\Z$-graded vector spaces $\Vect_\R^{\Z}$), where the braiding isomorphism comes with a Koszul sign (aka super-vector spaces)~$A$ is commutative. 

In the spirit of Remark \ref{rem:general-target}, 
the algebra~\eqref{eq:sigmamodelA} describes a closed TQFT with values in $\Vect_\R^{\Z_2}$ (or $\Vect_\R^{\Z}$). 
If~$X$ is a K\"ahler manifold, this TQFT is called the ``A-twisted sigma model'' \cite{WittenSigmaModel, mirrorbook}. 
There is also a related construction involving Calabi-Yau manifolds which give ``B-twisted sigma models''. 

\item
Not every unital associative algebra $A$ can be made into a Frobenius algebra by an appropriate choice of pairing. An example is the algebra $T$ of upper triangular $(2{\times}2)$-matrices with entries in~$\Bbbk$. One way to see this is to consider~$T$ as a left module over itself. The socle of $T$ is 1-dimensional, while the socle of the dual $T$-module $T^*$ is 2-dimensional. If there were an invariant nondegenerate pairing, we would have $T \cong T^*$ as $T$-modules, which cannot be.
\end{enumerate}
\end{examples}

\subsection{Three-dimensional TQFTs}
\label{sec:extendedTQFTs}

To treat 3-dimensional TQFTs properly with only a \textsl{finite} amount of data, one is lead to pass from phrasing everything in terms of symmetric monoidal categories to using ``symmetric monoidal 2-categories''. Instead of introducing the formalism thoroughly, we will merely attempt to motivate why this more elaborate setting is forced on us. Precision and details will be two early victims of our approach.

\subsubsection*{Bordisms for extended TQFTs}

We start by noticing that from $\Bord_3$ onwards, the objects are no longer tensor-generated by a finite set. Indeed, since the tensor product in $\Bord_n$ is disjoint union, we need one generator for each diffeomorphism class of connected (oriented, closed, compact) $(n-1)$-manifolds. For $n \geqslant 3$, there are infinitely many of these. In the language of Section \ref{sec:gen+rel_def}, already the set $G_0$ is infinite. 
For example, in the case of $\Bord_3$ the objects are tensor-generated by connected, oriented, closed, compact surfaces of genus~$g$ for all $g\in\N$. 

Therefore, in a generators-and-relations approach to symmetric monoidal functors $\Bord_3 \to \Vectk$ one has to do battle with an infinite amount of data subject to an infinite amount of conditions. 
This is not to say that such an algebraic description cannot be given. Indeed, the paper \cite{Juhasz2014} does just that (in arbitrary dimension~$n$).

\medskip

One way to arrive at a finite set of generators subject to a finite set of conditions is to introduce another categorical layer to 3-dimensional bordisms. 
Roughly, we start with objects already in dimension one, 2-bordisms between them as morphisms, and ``morphisms between morphisms''  are 3-dimensional manifolds with corners. 

More generally, TQFTs in dimension $n$ which also work with manifolds of dimensions other than $n$ and $n-1$ are called \textsl{extended TQFTs}. One way to remember which manifolds occur in a given bordism category is to rename $\Bord_n$ to $\Bord_{n,n-1}$, and to just add consecutive numbers to this list for extended TQFTs. For example, we will soon look at $\Bord_{3,2,1}$. 
If one considers all dimensions between~$0$ and~$n$, i.\,e.~the higher bordism category is $\Bord_{3,2,1,0}$ for $n=3$, the TQFT is called ``extended down to points''. 

Let us stay in general dimension $n$ for a little while before restricting ourselves again to $n=3$.
Since $\Bord_{n,n-1,n-2}$ is a three-tiered structure, it will be a bicategory. 
A little more precisely (but still wrong, see \cite{spthesis} for precision), it is given by
\begin{enumerate}
\item objects: oriented closed compact $(n-2)$-manifolds, 
\item 1-morphisms: $(n-1)$-dimensional compact oriented bordisms between two given such $(n-2)$-manifolds (not just up to diffeomorphism), 
\item 2-morphisms: $n$-dimensional compact oriented bordisms (which may now have corners) between two such $(n-1)$-manifolds, now considered up to diffeomorphisms compatible with boundary parametrisations.
\end{enumerate}

The reason this does not quite work was already alluded to in Remark \ref{rem:luck-with-gluing}: in order to have a well-defined gluing one should work with collars. 
That is, one should really take the $(n-2)$- and $(n-1)$-manifolds above to sit inside a little neighbourhood of surrounding $n$-manifold. 
The correct definition of the bicategory $\Bord_{n,n-1,n-2}$ is quite technical and can be found in \cite[Sect.\,3.1]{spthesis}. 
If one wants to do more with extended TQFTs than oversimplifying other people's mathematical results, one will have to wade through the details. 
In our discussion we will stick to oversimplification.

\subsubsection*{Targets for extended TQFTs}

Let us take a step back even from our familiar symmetric monoidal functors $\Bord_{n,n-1} \to \Vectk$. 	Namely, suppose we are merely interested in invariants of (closed, oriented, compact) $n$-manifolds. For example, we could look for a prescription which assigns a number $\zz(M) \in \Bbbk$ to each such $n$-manifold~$M$.
In keeping with the systematics of the notation just introduced, but risking total confusion, we now write $\Bord_n$ for the 0-category (= set) of diffeomorphism classes of closed oriented compact $n$-manifolds. Our invariants define a function
\be
	\zz\colon \Bord_n \lra \Bbbk \, .
\ee

One may then have the idea that such invariants can be computed from cutting $M$ along $(n-1)$-manifolds into simpler building blocks. In this way one would be lead to the category $\Bord_{n,n-1}$ and to consider functors $\zz$ out of it. It remains to find a good target category.

Let us write $\emptyset_k$ for the empty set viewed as a $k$-manifold. 
Then $\Bord_{n,n-1}$ contains all the closed manifolds as endomorphisms of $\emptyset_{n-1}$:
\be
	\Bord_n = \End_{\Bord_{n,n-1}}(\emptyset_{n-1}) \, .
\ee
We thus posit that whatever the target of a functor $\zz$ out of $\Bord_{n,n-1}$ is, $\zz(\emptyset_{n-1} \xrightarrow{M} \emptyset_{n-1})$ should be an element of $\Bbbk$.

Hence one needs to find a category $\mathcal{C}$ with a preferred object $*$ such that the endomorphisms of $*$ are $\Bbbk$. From this point of view $\Vectk$ with $* = \Bbbk$ is but one option. Super-vector spaces, representations of finite groups, or in fact any $\Bbbk$-linear symmetric monoidal category with $*$ chosen as the tensor unit (which is then assumed to be absolutely simple) will do just as well.
Nonetheless, we will stick with $\Vectk$ as target.

\medskip

Next we add a categorical layer and look at $\Bord_{n,n-1,n-2}$. As before,
\be
	\Bord_{n,n-1} = \End_{\Bord_{n,n-1,n-2}}(\emptyset_{n-2}) \, .
\ee
This suggests that
we need to find a bicategory $\mathcal{C}$ with a preferred object~$\star$ such that the endomorphism category of~$\star$ is $\Vectk$.

Again there are many options, and we refer to \cite[App.\,A]{bdspv1509.06811} for a discussion of one class of target bicategories collected under the heading ``2-vector spaces'' (which is a bit like calling vector spaces ``1-numbers''). 
One possible target for a 2-functor out of $\Bord_{n,n-1,n-2}$ is $\mathrm{LinCat}^{\mathrm{ss}}_{\Bbbk}$ which has:
\begin{itemize}
\item \textsl{objects}: finitely semisimple $\Bbbk$-linear abelian categories. Each such category is equivalent to $\mathrm{rep}_\mathrm{fin}(A)$, the category of finite-dimensional representations of a finite-dimensional
semisimple $\Bbbk$-algebra~$A$. 
\item \textsl{1-morphisms}: $\Bbbk$-linear functors. 
For semisimple algebras, all such functors $\mathrm{rep}_\mathrm{fin}(A) \to \mathrm{rep}_\mathrm{fin}(B)$ are given by tensoring with a $B$-$A$-bimodule.
\item \textsl{2-morphisms}: natural transformations.
For semisimple algebras, these are given by bimodule homomorphisms.
\end{itemize}
The distinguished object~$\star$ would be the category $\vectk := \mathrm{rep}_\mathrm{fin}(\Bbbk)$ of finite-dimensional $\Bbbk$-vector spaces. Indeed, the category of $\Bbbk$-linear additive endofunctors of $\vectk $ is again equivalent to $\vectk $, as desired.

There is a product operation on $\mathrm{LinCat}^{\mathrm{ss}}_{\Bbbk}$ which takes two categories $\mathcal{C},\mathcal{D} \in \Obj(\mathrm{LinCat}^{\mathrm{ss}}_{\Bbbk})$ and produces a new category $\mathcal{C} \boxtimes \mathcal{D}$ which is again an object in $\mathrm{LinCat}^{\mathrm{ss}}_{\Bbbk}$. We will give three equivalent descriptions of this operation (see e.\,g.~\cite[Def.\,1.1.15]{BakalovKirillov} and \cite[Sect.\,1.1]{EGNO}).
\begin{itemize}
\item 
The objects of $\mathcal{C} \boxtimes \mathcal{D}$ are formal direct sums of pairs $(X,Y)$ with $X \in \mathcal{C}$ and $Y \in \mathcal{D}$. We will write $X \boxtimes Y$ instead of $(X,Y)$. Since $\mathcal{C}$, $\mathcal{D}$ are $\Bbbk$-linear, their Hom-spaces are $\Bbbk$-vector spaces. The space of morphisms $X \boxtimes Y \to X' \boxtimes Y'$ in $\mathcal{C} \boxtimes \mathcal{D}$ is defined to be $\Hom_{\mathcal{C}}(X,X') \otimes_\Bbbk \Hom_{\mathcal{D}}(Y,Y')$. 
For direct sums of objects, we take the corresponding direct sum of Hom-spaces.

By construction we have a canonical $\Bbbk$-bilinear functor $\boxtimes\colon \mathcal{C} \times \mathcal{D} \to \mathcal{C} \boxtimes \mathcal{D}$, which maps $(X,Y)$ to $X \boxtimes Y$ and $(f,g)$ to $f \otimes_\Bbbk g$.
\item 
If $\mathcal{C} = \mathrm{rep}_{\mathrm{fin}}(A)$ and $\mathcal{D} = \mathrm{rep}_{\mathrm{fin}}(B)$ for semisimple $\Bbbk$-algebras $A,B$, then $\mathcal{C} \boxtimes \mathcal{D} = \mathrm{rep}_{\mathrm{fin}}(A \otimes_\Bbbk B)$.

There is a canonical functor $\boxtimes\colon \mathcal{C} \times \mathcal{D} \to \mathcal{C} \boxtimes \mathcal{D}$ which sends a pair $(M,N)$ of an $A$-module $M$ and a $B$-module $N$ to $M \otimes_{\Bbbk} N$, and dito for module homomorphisms.

\item
The category $\mathcal{C} \boxtimes \mathcal{D}$ in $\mathrm{LinCat}^{\mathrm{ss}}_{\Bbbk}$ together with the $\Bbbk$-bilinear functor $\boxtimes\colon \mathcal{C} \times \mathcal{D} \to \mathcal{C} \boxtimes \mathcal{D}$ satisfies the following universal property: For every $\mathcal{E} \in \Obj(\mathrm{LinCat}^{\mathrm{ss}}_{\Bbbk})$ the functor
\be
	\mathrm{Fun}_{\text{$\Bbbk$-lin}}(\mathcal{C} \boxtimes \mathcal{D},\mathcal{E})
	\xrightarrow{(-) \circ \boxtimes}
	\mathrm{Fun}_{\text{$\Bbbk$-bilin}}(\mathcal{C} \times \mathcal{D},\mathcal{E})
\ee
is an equivalence of categories.
\end{itemize}
Of these equivalent characterisations of $\mathcal{C} \boxtimes \mathcal{D}$ the first one is limited to the present setting, the second generalises to 
finite 
non-semisimple categories, and the last one -- also known as the Deligne product \cite{DeligneCT} -- goes far beyond this, see e.\,g.\ \cite[Sect.\,3]{Ben-Zvi:2015}.

\medskip

Now we should explain how $\Bord_{n,n-1,n-2}$ and $\mathrm{LinCat}^{\mathrm{ss}}_{\Bbbk}$ are symmetric monoidal bicategories (and what such a bicategory is in the first place), what the appropriate notion of symmetric monoidal 2-functor $\Bord_{n,n-1,n-2} \to \mathrm{LinCat}^{\mathrm{ss}}_{\Bbbk}$ is, etc. 
But in keeping with merrily oversimplifying we will not 
(see, however, \cite{spthesis}). 
Instead we now return to three dimensions and will sketch why the statement
\begin{quote}
``An extended TQFT defined on $\Bord_{3,2,1}$ and taking values in $\mathrm{LinCat}^{\mathrm{ss}}_{\Bbbk}$ assigns a braided monoidal category to the circle $S^1$.''
\end{quote}
might be true.

\subsubsection*{Three-dimensional extended TQFTs and braided monoidal categories}

Let $\zz\colon \Bord_{3,2,1} \to \mathrm{LinCat}^{\mathrm{ss}}_{\Bbbk}$ be a symmetric monoidal 2-functor (we will not need to know what this is in detail). Here is how $\zz$  produces a braided monoidal category $\mathcal{C}$:
\begin{itemize}
\item
\textsl{underlying category:} The 2-functor maps objects to objects, hence $\mathcal{C} := \zz(S^1)$ will be an object in $\mathrm{LinCat}^{\mathrm{ss}}_{\Bbbk}$, that is, a 
finite semisimple $\Bbbk$-linear abelian category.
\item
\textsl{tensor product functor:} We are looking for a functor $\mathcal{C} \times \mathcal{C} \to \mathcal{C}$. 
This functor should be $\Bbbk$-bilinear.
By the third characterisation of $\mathcal{C} \boxtimes \mathcal{C}$ given above, this functor therefore factors uniquely as
\be
	\mathcal{C} \times \mathcal{C}
	\stackrel{ \boxtimes }{\lra}
	\mathcal{C} \boxtimes \mathcal{C}
	\stackrel{T}{\lra}
	\mathcal{C}
\ee
for some $T$, which is now $\Bbbk$-linear, i.\,e.\ a 1-morphism in $\mathrm{LinCat}^{\mathrm{ss}}_{\Bbbk}$.

The 2-functor $\zz$ will map 1-morphisms in $\Bord_{3,2,1}$, i.\,e.\ surfaces, to 1-morphisms in $\mathrm{LinCat}^{\mathrm{ss}}_{\Bbbk}$. So the natural place to look for $T$ is a surface with two incoming boundary circles and one outgoing boundary circle. The simplest such surface is a pair-of-pants
\be
\Phi = 
{
\begin{tikzpicture}[scale=0.1,color=red!65!black, baseline=-0.1cm]
%
\fill [orange!23] (0,0) circle (6);
%
\fill[color=white] (2.5,0) circle (1.5);
\fill[color=white] (-2.5,0) circle (1.5);
%
\draw[very thick] (0,0) circle (6);
%
\draw[very thick, color=red!30!orange] (2.5,0) circle (1.5);
%
\draw[very thick, color=red!30!orange] (-2.5,0) circle (1.5);
\end{tikzpicture} 
} 
\,
= 
{
\begin{tikzpicture}[thick,scale=1.0,color=black, baseline=0.3cm]
\coordinate (p1) at (-0.55,0);
\coordinate (p2) at (-0.2,0);
\coordinate (p3) at (0.2,0);
\coordinate (p4) at (0.55,0);
\coordinate (p5) at (0.175,0.8);
\coordinate (p6) at (-0.175,0.8);
%
\fill [orange!23] 
(p1) .. controls +(0,-0.15) and +(0,-0.15) ..  (p2)
-- (p2) .. controls +(0,0.35) and +(0,0.35) ..  (p3)
-- (p3) .. controls +(0,-0.15) and +(0,-0.15) ..  (p4)
-- (p4) .. controls +(0,0.5) and +(0,-0.5) ..  (p5)
-- (p5) .. controls +(0,0.15) and +(0,0.15) ..  (p6)
-- (p6) .. controls +(0,-0.5) and +(0,0.5) ..  (p1)
;
\fill [orange!38] 
(p1) .. controls +(0,-0.15) and +(0,-0.15) ..  (p2)
-- (p2) .. controls +(0,0.15) and +(0,0.15) ..  (p1)
;
\fill [orange!38] 
(p3) .. controls +(0,-0.15) and +(0,-0.15) ..  (p4)
-- (p4) .. controls +(0,0.15) and +(0,0.15) ..  (p3)
;
\draw (p2) .. controls +(0,0.35) and +(0,0.35) ..  (p3); 
\draw (p4) .. controls +(0,0.5) and +(0,-0.5) ..  (p5); 
\draw (p6) .. controls +(0,-0.5) and +(0,0.5) ..  (p1); 
\draw[very thick, color=red!30!orange] (p1) .. controls +(0,0.15) and +(0,0.15) ..  (p2); 
\draw[very thick, color=red!30!orange] (p1) .. controls +(0,-0.15) and +(0,-0.15) ..  (p2); 
\draw[very thick, color=red!30!orange] (p3) .. controls +(0,0.15) and +(0,0.15) ..  (p4); 
\draw[very thick, color=red!30!orange] (p3) .. controls +(0,-0.15) and +(0,-0.15) ..  (p4); 
\draw[very thick, red!65!black] (p5) .. controls +(0,0.15) and +(0,0.15) ..  (p6); 
\draw[very thick, red!65!black, opacity=0.2] (p5) .. controls +(0,-0.15) and +(0,-0.15) ..  (p6); 
\end{tikzpicture}
} 
\ee
of which we depict two representative bordisms. 
We obtain a functor $\zz(\Phi)\colon \zz(S^1 \sqcup S^1) \to \zz(S^1)$. 
Furthermore, the monoidality of $\zz$ (which we did not explain) provides us with an equivalence of categories $\zz(S^1 \sqcup S^1) \cong \zz(S^1) \boxtimes \zz(S^1)$. 
Altogether we obtain a functor 
\be
	T = \Big[ \,
	\mathcal{C} \boxtimes \mathcal{C}
	\stackrel{\sim}{\lra}
	\zz(S^1 \sqcup S^1)
	\xrightarrow{\zz(\Phi)}
	\zz(S^1) = \mathcal{C} \, \Big] \, .
\ee
Now one should find the associator and verify the pentagon. For this one needs to look at spheres with 3 and 4 incoming holes and one outgoing hole, and bordisms between them, but the details descend too far into the inner workings of symmetric monoidal bicategories and their 2-functors.
\item
\textsl{unit object:}
We need to identify an object $\one \in \mathcal{C}$ which serves as the unit under the tensor product functor $T$, that is, $T(\one \boxtimes -) \cong \id \cong T(- \boxtimes \one)$ as $\Bbbk$-linear functors (subject to coherence conditions). 
Looking at the pair-of-pants defining $T$, we note that gluing a disc~$D$ into one of the holes produces a cylinder, whose image under~$\zz$ is equivalent to the identity functor.
So we consider $\zz(D)\colon \zz(\emptyset_1) \to \zz(S^1)$. But $\zz(\emptyset_1) \cong \vectk $, as by monoidality~$\zz$ should map the monoidal unit to something equivalent to the monoidal unit, and a $\Bbbk$-linear functor $\vectk  \to \mathcal{C}$ is uniquely determined (up to natural isomorphism) by what it does to $\Bbbk$. 
We choose $\one \in \mathcal{C}$ to be the image of $\Bbbk$ under the functor $\zz(D)$. 
The unit isomorphisms and their coherence conditions again require  too many of the  details we skipped earlier.
\item
\textsl{braiding isomorphism:} 
We need to find, for each pair $X,Y \in \mathcal{C}$, an isomorphism 
$\beta_{X,Y}\colon T(X \boxtimes Y) \to T(Y \boxtimes X)$ natural in~$X$ and~$Y$. 
To view the collection $\{\beta_{X,Y}\}$ as a 2-morphism in $\mathrm{LinCat}^{\mathrm{ss}}_{\Bbbk}$ we write it as a natural isomorphism making the triangle
\be\label{eq:braiding-as-natiso}
\begin{tikzpicture}[
			     baseline=(current bounding box.base), 
			     >=stealth,
			     descr/.style={fill=white,inner sep=3.5pt}, 
			     normal line/.style={->}
			     ] 
\matrix (m) [matrix of math nodes, row sep=2.0em, column sep=1.0em, text height=1.1ex, text depth=0.1ex] {%
\mathcal{C} \boxtimes \mathcal{C} && \mathcal{C} \boxtimes \mathcal{C}
\\
& \vspace{-1cm}
{}^{\stackrel{\normalsize{\beta}}{\rotatebox[origin=c]{214}{$\implies$}}} &
\\
& \mathcal{C} &
\\
};
\path[font=\footnotesize] (m-1-1) edge[->] node[auto] {\text{swap}} (m-1-3);
\path[font=\footnotesize] (m-1-1) edge[->] node[below] {$T\;\;\;$} (m-3-2);
\path[font=\footnotesize] (m-1-3) edge[->] node[below] {$\;\;\;T$} (m-3-2);
\end{tikzpicture}
\ee
commute. The top arrow maps the ``pure tensor object'' $X \boxtimes Y$ to $Y \boxtimes X$, and similarly for morphisms. That this indeed defines an endofunctor of $\mathcal{C} \boxtimes \mathcal{C}$ follows most readily from the third characterisation of $\boxtimes$ given above. The functor ``swap'' is actually part of the symmetric structure on $\mathrm{LinCat}^{\mathrm{ss}}_{\Bbbk}$.

A natural isomorphism making \eqref{eq:braiding-as-natiso} commute is a 2-isomorphism in $\mathrm{LinCat}^{\mathrm{ss}}_{\Bbbk}$. The 2-morphisms in $\Bord_{3,2,1}$ are 3-dimensional bordisms (with corners). The symmetric structure on $\Bord_{3,2,1}$ acts on objects by reordering the components of a disjoint union. We arrive at the following candidate for the bordism giving the braiding: 
\be
\label{eq:sausagetwist}
B = 
%
\begin{tikzpicture}[thick,scale=2.0,color=blue!50!black, baseline=0.0cm, >=stealth, 
				style={x={(-0.9cm,-0.4cm)},y={(0.8cm,-0.4cm)},z={(0cm,0.9cm)}}]
%
\foreach \z in {0, 0.5, ..., 100}
{
	\draw[very thick, color=orange!23, opacity=0.9] (0.5, 0.5, \z/100) circle (0.5);
}
%
\draw[very thick, color=red!80!black] (0.5, 0.5, 0) circle (0.5);
%
\foreach \z in {0, 0.45, ..., 180}
{
	\draw[very thin, color=red!65!black] (0.5 - 0.3 * sin \z, 0.5 - 0.3 * cos \z  ,\z/180) circle (0.15);
}
\draw[very thick, color=red!65!black] (0.5 , 0.2 ,0) circle (0.15);
\foreach \z in {0, 0.45, ..., 180}
{
	\draw[very thin, color=red!30!orange] (0.5 + 0.3 * sin \z , 0.5 + 0.3 * cos \z ,\z/180) circle (0.15);
}
%
\draw[very thick, color=red!65!black] (0.5, 0.8 ,1) circle (0.15);
\draw[very thick, color=red!30!orange] (0.5 , 0.8, 0) circle (0.15);
\draw[very thick, color=red!30!orange] (0.5 , 0.2, 1) circle (0.15);
%
\draw[very thick, color=red!80!black] (0.5, 0.5, 1) circle (0.5);
\end{tikzpicture}
\ee
The corresponding natural transformation $\zz(B)$ between two functors $\zz(S^1 \sqcup S^1) \to \zz(S^1)$ provides the family of isomorphisms $\beta_{X,Y}$ we are after. 
These isomorphisms have to satisfy the hexagon conditions, the details of which require the coherence properties of symmetric monoidal 2-functors, and so we omit this argument.
However, we would like to stress that 
-- as suggested by the bordism in \eqref{eq:sausagetwist} -- 
in general the braiding will not be symmetric: $\beta_{X,Y} \neq \beta^{-1}_{Y,X}$.
\end{itemize}

The above sketch illustrates how an extended TQFT $\zz\colon \Bord_{n,n-1,n-2} \to \mathrm{LinCat}^{\mathrm{ss}}_{\Bbbk}$ produces a braided monoidal category. However, braided monoidal categories do not provide an algebraic description of extended TQFTs. Indeed, not every braided monoidal category in $\mathrm{LinCat}^{\mathrm{ss}}_{\Bbbk}$ arises as $\zz(S^1)$ for some extended TQFT, and -- even worse -- inequivalent TQFTs may produce equivalent braided monoidal categories.

With a heroic amount of work it is possible to describe $\Bord_{3,2,1}$ in terms of a finite number of generators and relations (but now as a symmetric monoidal bicategory, which is one categorical level more involved than Section \ref{sec:gen+rel_def}), and to prove:

\begin{theorem}[\cite{bdspv1509.06811}]\label{thm:class123}
Assume $\Bbbk$ is algebraically closed.
Extended TQFTs $\zz\colon \Bord_{3,2,1} \to \mathrm{LinCat}^{\mathrm{ss}}_{\Bbbk}$ are classified by anomaly-free modular tensor categories over $\Bbbk$.
\end{theorem}

A modular tensor category is an object in $\mathrm{LinCat}^{\mathrm{ss}}_{\Bbbk}$ which is braided monoidal with simple tensor unit, and in addition has two-sided duals and a ribbon twist. Furthermore, the braiding has to satisfy a nondegeneracy condition. 
Anomaly-freeness is a condition on the ribbon twists and quantum dimensions. See e.\,g.~\cite[Sect.\,8.13--8.15]{EGNO}, \cite[Sect.\,3.1\&\,4.4]{BakalovKirillov} or indeed \cite{bdspv1509.06811} for more details.

\subsection[Dimensional reduction from 3 to 2 along $S^1$]{Dimensional reduction from 3 to 2 along $\boldsymbol{S^1}$}
\label{subsec:dimred3to2}

In Proposition \ref{prop:dimred} we described the general procedure of dimensional reduction, 
taking a TQFT of some dimension and producing lower-dimensional TQFTs. 
In this section we would like to take an extended three-dimensional TQFT as described in the previous section and see which two-dimensional TQFT it produces when dimensionally reducing along a circle. 

\medskip

Before we sketch the derivation, let us guess (correctly) what the answer is.
Assume the field $\Bbbk$ to be algebraically closed.
By Theorem \ref{thm:2dTQFT}, to give a two-dimensional TQFT with values in $\Vectk$ is the same as to give a commutative Frobenius algebra over $\Bbbk$. 
By Theorem \ref{thm:class123}, to give an extended three-dimensional TQFT is the same as to give an anomaly-free modular tensor category over $\Bbbk$.
In particular, we need to produce a commutative algebra out of a braided monoidal category. This can be done by taking the Grothendieck ring.

The additive \textsl{Grothendieck group} of a category $\mathcal{C} \in \mathrm{LinCat}^{\mathrm{ss}}_{\Bbbk}$ can be defined as 
\be
\mathrm{Gr}(\mathcal{C}) 
:= \bigoplus_{U \in \mathrm{Irr}(\mathcal{C})} 
\mathbb{Z} [U] \, .
\ee
Here, $\mathrm{Irr}(\mathcal{C})$ stands for a choice of simple objects in $\mathcal{C}$ which contains precisely one representative for each isomorphism class of simple objects. For $U \in \mathrm{Irr}(\mathcal{C})$, $[U]$ denotes the isomorphism class it represents. So $\mathrm{Gr}(\mathcal{C})$ is a free abelian group of 
rank $|\mathrm{Irr}(\mathcal{C})|$. 

If $\mathcal{C}$ is in addition monoidal 
with tensor product induced by $T\colon \mathcal{C} \boxtimes \mathcal{C} \to \mathcal{C}$, then 
$\mathrm{Gr}(\mathcal{C})$ 
is a unital associative ring with product given on generators by
\be
	[U] \cdot [V] := [T(U \boxtimes V)] \, .
\ee
If (but not only if) $\mathcal{C}$ is braided, this product is commutative. 
Hence, to a braided tensor category $\mathcal C \in \mathrm{LinCat}^{\mathrm{ss}}_{\Bbbk}$ we can associate the commutative $\Bbbk$-algebra $\mathrm{Gr}(\mathcal{C}) \otimes_\mathbb{Z} \Bbbk$.

It remains to find a nondegenerate and invariant pairing. 
Here we use that a modular tensor category -- in addition to being braided monoidal -- also has duals~$U^*$ for every object~$U$. 
Now for $U,V \in \mathrm{Irr}(\mathcal{C})$, 
the tensor unit occurs in $U \otimes V$ with multiplicity one if $V \cong U^*$ and with multiplicity zero else 
(as $\Hom_{\mathcal{C}}(U \otimes V,\one) \cong \Hom_{\mathcal{C}}(U,V^*)$). 
Thus the pairing 
\be
\label{eq:pairingguess}
\big\langle [U],[V] \big\rangle 
:= 
\dim_\Bbbk \Hom_{\mathcal{C}}(U \otimes V,\one)
\ee
turns $\mathrm{Gr}(\mathcal{C}) \otimes_\mathbb{Z} \Bbbk$ into a commutative Frobenius algebra.

\medskip

We now verify that the TQFT calculation indeed produces $\mathrm{Gr}(\mathcal{C}) \otimes_\mathbb{Z} \Bbbk$ with the above pairing. 
In the course of this verification, we will have to pull two facts out of our hats. 

Let $\zz$ be 
a 3-dimensional 
extended TQFT.
Denote by $S^2_{0 \to 2}$ a 2-sphere with two outgoing holes, and by $S^2_{m \to 0}$ a 2-sphere with $m$~ingoing holes. To these, the extended TQFT $\zz$ assigns functors. Here is the first out-of-the-hat fact: on objects these functors act as
\begin{align}
\zz\big(S^2_{0 \to 2}\big)\colon \vectk  &\lra \mathcal{C} \boxtimes \mathcal{C} \, , 
\nonumber 
\\
\Bbbk &\lmt \bigoplus_{U \in \mathrm{Irr}(\mathcal{C})} U^* \boxtimes U \, , 
\nonumber 
\\
\zz\big(S^2_{n \to 0}\big) \colon \mathcal{C}^{\boxtimes n} &\lra \vectk  \, , 
\nonumber 
\\
X_1 \boxtimes \cdots \boxtimes X_n &\lmt \Hom_{\mathcal{C}}(X_1 \otimes \cdots \otimes X_n,\one) \, . 
\end{align}
With this, we can compute the state space on $S^1$ of the dimensionally reduced 2-dimensional TQFT $\zz^\mathrm{red}$:
\begin{align}
	A := \zz^\mathrm{red}\big(S^1\big)
	&=
	\zz\big(S^1 \times S^1\big)
	\nonumber \\
	&=
	\zz\big(S^2_{2 \to 0}) \circ \zz(S^2_{0 \to 2}\big)
	\nonumber \\
	&=
	\bigoplus_{U \in \mathrm{Irr}(\mathcal{C})} \Hom_{\mathcal{C}}(U \otimes U^* , \one)
	\nonumber \\
	&\cong
	\bigoplus_{U \in \mathrm{Irr}(\mathcal{C})} \Hom_{\mathcal{C}}(U,U) \, .
\end{align}
Here we used implicitly the identification of $\Bbbk$-linear functors $\vectk\to\vectk$ with vector spaces (by evaluating on $\Bbbk$).
Since $\Hom_{\mathcal{C}}(U,U) \cong \Bbbk$ this is indeed canonically isomorphic to $\mathrm{Gr}(\mathcal{C}) \otimes_\mathbb{Z} \Bbbk$ as a $\Bbbk$-vector space (even as one with preferred basis).

The second hat-fact is the value of $\zz$ on $S^2_{m \to 0} \times S^1$, which by definition is a linear map from $A^{\otimes_\Bbbk m}$ to $\Bbbk$. A basis of $A^{\otimes_\Bbbk m}$ is given by the elements $1_{U_1} \otimes_{\Bbbk} \cdots \otimes_{\Bbbk} 1_{U_m}$ for $U_1,\dots,U_m \in \mathrm{Irr}(\mathcal{C})$. 
On this basis, we have
\begin{align}
\zz^{\mathrm{red}}\big(S^2_{m \to 0}\big) (1_{U_1} \otimes_{\Bbbk} \cdots \otimes_{\Bbbk} 1_{U_m})
&= 
\zz\big(S^2_{m \to 0} \times S^1\big)(1_{U_1} \otimes_{\Bbbk} \cdots \otimes_{\Bbbk} 1_{U_m})
\nonumber \\
&=
\dim_{\Bbbk} \Hom_{\mathcal{C}}(U_1 \otimes \cdots \otimes U_m,\one) \, .
\label{eq:ZredS2nS1}
\end{align}

For $m=2$ we obtain a pairing on $A$ which is precisely the pairing $\langle-,-\rangle$ from our earlier guess~\eqref{eq:pairingguess}. 
Since this pairing is nondegenerate, the product on $A$ is determined uniquely by the corresponding map $A^{\otimes_\Bbbk 3} \to \Bbbk$. 
For $\mathrm{Gr}(\mathcal{C}) \otimes_\mathbb{Z} \Bbbk$ this map is given by
\be
	[U] \otimes_\Bbbk [V] \otimes_\Bbbk [W]
	\lmt \big\langle[U]\cdot[V],[W]\big\rangle 
	=
	\dim_\Bbbk \Hom_{\mathcal{C}}(U\otimes V\otimes W,\one)\, ,
\ee
which agrees with \eqref{eq:ZredS2nS1}.

\bigskip

This ends our expository notes on TQFT.

\appendix

\section{Appendix}

\subsection{Alternative description of symmetric monoidal functors}
\label{app:alt-sym-mon-fun}

In this appendix we assume familiarity with symmetric monoidal functors and dualities and how to manipulate them. 
Let~$\mathcal{C}$ and~$\mathcal{D}$ be symmetric monoidal categories and let $\mathcal{C}$ have left duals (and hence also right duals). 
Let $\gamma$ denote the following collection of data:
\begin{itemize}
\item for all $X \in \mathcal{C}$ an object $\gamma(X) \in \mathcal{D}$.
\item for all morphisms $f\colon \one_{\mathcal{C}} \to X$ in $\mathcal{C}$ a morphism $\gamma(f)\colon \gamma(\one_{\mathcal{C}}) \to \gamma(X)$ in $\mathcal{D}$,
\item an isomorphism $\gamma^0\colon \one_{\mathcal{D}} \to \gamma(\one_{\mathcal{C}})$ and a family of isomorphisms $\gamma^2_{X,Y}\colon \gamma(X) \otimes_{\mathcal{D}} \gamma(Y) \to \gamma(X \otimes_{\mathcal{C}} Y)$, where $X,Y \in \mathcal{C}$.
\end{itemize}

In the next lemma we will lighten notation by dropping tensor products between objects in longer expressions, and by writing $\xrightarrow{\sim}$ when it is clear how to obtain the corresponding isomorphism out of the coherence isomorphisms of $\mathcal{C}$, $\mathcal{D}$ and $\gamma$.

\begin{lemma}\label{lem:sym-mon-fun-via-Y}
The following are equivalent.
\begin{enumerate}
\item $\gamma$ extends to a symmetric monoidal functor $\mathcal{C}\to\mathcal{D}$.
\item $\gamma$ satisfies:
\begin{enumerate}
\item \textsl{(Existence of dual pairing)} for each $X \in \mathcal{C}$ there is a morphism $d_X\colon \gamma(X^*) \otimes \gamma(X) \to \one_{\mathcal D}$ which is dual to $\gamma(\coev_X)$ in the sense that
\begin{align*}
1_{\gamma(X)} ~=~ 
\Big[&
\gamma(X) \stackrel{\sim}{\lra} \gamma(\one_{\mathcal C})\gamma(X)
\xrightarrow{\gamma(\coev_X) \otimes \id} \gamma(X X^*) \gamma(X)
\\ &
\stackrel{\sim}{\lra} \gamma(X)(\gamma(X^*)\gamma(X))
\xrightarrow{\id \otimes d_X} \gamma(X) \one_{\mathcal D} \stackrel{\sim}{\lra} \gamma(X)
\Big] \, ,
\\
1_{\gamma(X^*)} ~=~ 
\Big[&
\gamma(X^*) \stackrel{\sim}{\lra} \gamma(X^*)\gamma(\one_{\mathcal C})
\xrightarrow{\id \otimes \gamma(\coev_X)}  \gamma(X^*)\gamma(X X^*)
\\ &
\stackrel{\sim}{\lra} ( \gamma(X^*)\gamma(X) ) \gamma(X^*)
\xrightarrow{d_X \otimes \id} \one_{\mathcal D}  \gamma(X^*) \stackrel{\sim}{\lra} \gamma(X^*)
\Big] \, .
\end{align*}
\item \textsl{(Compatibility with gluing)} for all $X,M \in \mathcal{C}$ and all $f\colon \one \to M \otimes (X^* \otimes X)$,
\begin{align*}
&
\Big[\,
\gamma(\one_{\mathcal C}) \xrightarrow{\gamma(f)} \gamma\big(M(X^*X)\big)
\stackrel{\sim}{\lra}
\gamma(M)(\gamma(X^*)\gamma(X))
\xrightarrow{\id \otimes d_X} \gamma(M)\one_{\mathcal D}
\stackrel{\sim}{\lra}\gamma(M) \,\Big]
\\
&=
\gamma\Big(\Big[\,\one_{\mathcal C} \xrightarrow{f} M(X^*X)
\xrightarrow{\id \otimes \ev_X} M \one_{\mathcal D} \stackrel{\sim}{\lra} M \,\Big]\Big) \, .
\end{align*}
\item \textsl{(Compatibility with tensor products)} for all $f: \one_{\mathcal C} \to X$ and $g\colon \one_{\mathcal C} \to Y$ in $\mathcal{C}$,
\begin{align*}
&\gamma\Big(\Big[ \one_{\mathcal C} \stackrel{\sim}{\lra} \one_{\mathcal C}\one_{\mathcal C} \xrightarrow{f \otimes g} XY \Big] \Big)
\\&= 
\Big[
\gamma(\one_{\mathcal C}) \stackrel{\sim}{\lra}\gamma(\one_{\mathcal C})\gamma(\one_{\mathcal C})
\xrightarrow{\gamma(f) \otimes \gamma(g)} \gamma(X)\gamma(Y)
\stackrel{\sim}{\lra} \gamma(XY) \Big] 
\, . 
\end{align*}
\item 
\textsl{(Compatibility with symmetric monoidal structure)}
Let $X_1,\dots,X_n \in \mathcal{C}$ and let $W$ be the tensor product of $X_1,\dots,X_n$, in this order, with any choice of bracketing. Let $\pi \in S_n$ be a permutation and write $W_\pi$ for the tensor product of $X_{\pi1},\dots,X_{\pi n}$, again in this order and with any choice of bracketing (possibly different from that of $W$). There is a unique isomorphism $\hat\pi_{\mathcal{C}}\colon W \to W_\pi$ which is built from the symmetric braiding and the associator of $\mathcal{C}$ (there are many ways to write this isomorphism, but they are all equal by coherence). 

Finally, let $W^\gamma$ be the tensor product of $\gamma(X_1),\dots,\gamma(X_n)$ in this order with the same bracketing as $W$ and dito for $W_\pi^\gamma$ and $\gamma(X_{\pi1}),\dots,\gamma(X_{\pi n})$. Write $\hat\pi_{\mathcal{D}}\colon W^\gamma \to W_\pi^\gamma$ for the corresponding unique isomorphism.

Then, for all $f\colon \one_{\mathcal C} \to W$,
\begin{align*}
&\gamma\Big(\Big[ \one_{\mathcal C} \stackrel{f}{\lra} W \stackrel{\hat\pi_{\mathcal{C}}}{\lra} W_\pi \Big] \Big)
\\&= 
\Big[
\gamma(\one_{\mathcal C})
\xrightarrow{\gamma(f)} \gamma(W) 
\stackrel{\sim}{\lra} W^\gamma 
\stackrel{\hat\pi_{\mathcal{D}}}{\lra} W_\pi^\gamma  
\stackrel{\sim}{\lra}
\gamma(W_\pi) \Big] \, .
\end{align*} 
\item
\textsl{(Compatibility with units)}
The identity
$$
	\gamma(1_{\one_{\mathcal C}}) = 1_{\gamma(\one_{\mathcal C})}
$$
holds, and the squares
$$
	\xymatrix{ 
	\one\gamma(\one_{\mathcal D}) 
		\ar[r]^{\lambda^{\mathcal{D}}_{\gamma(\one)}} 
		\ar[d]_{\gamma^0 \otimes \id}
	&
	\gamma(\one_{\mathcal C})
		\ar[d]^{\gamma(\lambda^{\mathcal{C}}_{\one})}
	\\
	\gamma(\one_{\mathcal C})\gamma(\one_{\mathcal C})
	\ar[r]_{\gamma^2_{\one,\one}}
	&
	\gamma(\one_{\mathcal C}\one_{\mathcal C})
	}
	\quad , \quad
	\xymatrix{ 
	\gamma(\one_{\mathcal C}) \one_{\mathcal D}
		\ar[r]^{\rho^{\mathcal{D}}_{\gamma(\one)}} 
		\ar[d]_{\id \otimes \gamma^0}
	&
	\gamma(\one_{\mathcal C})
		\ar[d]^{\gamma(\rho^{\mathcal{C}}_{\one})}
	\\
	\gamma(\one_{\mathcal C})\gamma(\one_{\mathcal C})
	\ar[r]_{\gamma^2_{\one,\one}}
	&
	\gamma(\one_{\mathcal C}\one_{\mathcal C})
	}
$$
commute. 
\end{enumerate}
\end{enumerate}
If $\gamma$ satisfies the conditions in (ii), the extension to a symmetric monoidal functor in (i) is unique.
\end{lemma}

\begin{proof}[Sketch of proof.]
The detailed proof is surprisingly tedious, and more an exercise in endurance than in ingenuity. We will only indicate the rough steps.

\medskip

\noindent
{}From 
(i) to (ii) there is really nothing to do. The value of $d_X$ is
$$
d_X = \Big[\,
\gamma(X^*) \gamma(X) \xrightarrow{\gamma^2_{X^*,X}} \gamma(X^*X) 
\xrightarrow{\gamma(\ev_X)} \gamma(\one_{\mathcal C}) \xrightarrow{(\gamma^0)^{-1}} \one_{\mathcal D} \,\Big] \, .
$$

\medskip

\noindent
{}From (ii) to (i): The first step is to extend the definition of $\gamma$ from morphisms $\one_{\mathcal C} \to Y$ to morphisms $X \to Y$ for all $X \in \mathcal{C}$. This is done by the following ansatz. Let $f\colon X \to Y$ be  a morphism in $\mathcal{C}$. Abbreviate
$$
	\tilde f = \Big[\, \one_{\mathcal C} \xrightarrow{\coev_X} XX^* \xrightarrow{f \otimes \id} Y X^* \,\Big] \, .
$$
We set
\begin{align*}
	\gamma(f) := 
	\Big[\, &\gamma(X) \stackrel{\sim}{\lra} \one_{\mathcal D} \gamma(X) \xrightarrow{\gamma^0 \otimes \id} \gamma(\one_{\mathcal C})\gamma(X)
	\xrightarrow{\gamma(\tilde f) \otimes \id}
	\gamma(Y X^*) \gamma(X)
	\\&
	\stackrel{\sim}{\lra}
	\gamma(Y) (\gamma(X^*) \gamma(X))
	\xrightarrow{\id \otimes d_X} \gamma(Y) \one_{\mathcal D} 
	\stackrel{\sim}{\lra} \gamma(Y) \, \Big] \, .
\end{align*}
One now has to check that this notation actually makes sense, i.\,e.~that it produces the original $\gamma$ when evaluated on $f\colon \one_{\mathcal C} \to Y$. In verifying this, one already needs all the properties a)--e) of $\gamma$.

Now that $\gamma$ is defined on all objects and morphisms of $\mathcal{C}$, one can set about verifying that the data $(\gamma, \gamma^0, \gamma^2)$ is a symmetric monoidal functor $\mathcal{C} \to \mathcal{D}$.
\begin{itemize}
\item
$\gamma(1_X) = 1_{\gamma(X)}$: Since $\tilde 1_X = \coev_X$, this boils down to the two zig-zag identities in a). 
\item
$\gamma(f) \circ \gamma(g) = \gamma(f \circ g)$: This is already a more lengthy calculation. At some point one wants to use b) to move a $d_X$ into $\gamma(\dots)$, but the $d_X$ does not quite sit in the right place and one needs to use d) to shuffle the arguments around.
\item \textsl{$\gamma^2_{X,Y}$ is natural in $X$ and $Y$:} One needs to verify that
$$
\xymatrix{
\gamma(U) \gamma(X) 
\ar[rr]^{\gamma(f) \otimes \gamma(g)}
\ar[d]_{\gamma^2_{U,X}}
&& \gamma(V) \gamma(Y)
\ar[d]^{\gamma^2_{V,Y}}
\\
\gamma(U X) 
\ar[rr]_{\gamma(f \otimes g)}
&&\gamma(V Y) 
}
$$
commutes for all $f\colon U \to V$ and $g\colon X \to Y$. In doing so, one is quickly lead to using c), but is then left with $d_U$ and $d_X$ instead of $d_{U \otimes X}$. At this point the uniqueness of the dual pairing, which follows from the zig-zag identities in a), can be used. One verifies that
\begin{align*}
	&\gamma((UX)^*) \gamma(UX)
	\stackrel{\sim}{\lra}\gamma(X^*U^*)\gamma(UX)
	\stackrel{\sim}{\lra}
	\gamma(X^*)((\gamma(U^*)\gamma(U))\gamma(X))
\\
	& \xrightarrow{\id \otimes d_U \otimes \id}
	\gamma(X^*)(\one_{\mathcal D} \gamma(X))
	\stackrel{\sim}{\lra}
	\gamma(X^*)\gamma(X)
	\xrightarrow{d_X} 
	\one_{\mathcal D}
\end{align*}
equally satisfies the zig-zag-identities for $\gamma(\coev_{U \otimes X})$.
\item \textsl{$\gamma^2$ is compatible with the associators and symmetric braiding:} From d), the definition of $\gamma$ and the zig-zag identities in a) one quickly finds, in the notation of d),
$$
	\gamma\Big( \Big[ W \xrightarrow{\hat\pi_{\mathcal{C}}} W_\pi \Big] \Big)
= 
\Big[
\gamma(W) 
\stackrel{\sim}{\lra} W^\gamma 
\stackrel{\hat\pi_{\mathcal{D}}}{\lra} W_\pi^\gamma  
\stackrel{\sim}{\lra}
\gamma(W_\pi) \Big] \, .
$$
This implies compatibility with associators and braidings.
\item \textsl{$\gamma^0$, $\gamma^2$ are compatible with units:} immediate from e).
\end{itemize}
\end{proof}

\subsection{Monoidal natural transformations and duals}\label{app:mon-nat-x-groupoid}

Here we provide the abstract reason why categories of monoidal functors form groupoids in the presence of duals. The same disclaimer as in the beginning of Appendix \ref{app:alt-sym-mon-fun} applies: familiarity with monoidal functors and duality morphisms is assumed.

\begin{lemma}\label{lem:sym-mon-inv}
Let $\mathcal{C}$, $\mathcal{D}$ be monoidal categories, $F,G\colon \mathcal{C} \to \mathcal{D}$ monoidal functors and $\phi\colon F \to G$ a natural monoidal transformation. If $\mathcal{C}$ has left duals (or right duals), $\phi$ is invertible.
\end{lemma}

\begin{proof}[Sketch of proof.]
The proof amounts to plugging together three observations.
\begin{enumerate}
\item Everything in the image of $F$ also has a left dual. For example, the  coevaluation is given by
$$
\coev^F_{F(X)} := \Big[ 
\one
\xrightarrow{F^0} F(\one)
\xrightarrow{F(\coev_X)} F(XX^*)
\xrightarrow{(F^2_{X,X^*})^{-1}}
F(X)F(X^*)
\Big] \, .
$$
The same of course holds for $G$.
\item $\phi$ relates the duality maps arising from $F$ and $G$. For example, the coevaluations satisfy
$$
(\phi_X \otimes \phi_{X^*}) \circ \coev^F_{F(X)} ~=~ \coev^G_{G(X)} \, .
$$
\item The natural transformation $\tilde\phi\colon G \to F$ given by
\begin{align*}
\tilde\phi_X :=
\Big[
& G(X)
\stackrel{\sim}{\lra} \one G(X) 
\xrightarrow{\coev^F_{F(X)} \otimes \id} (F(X)F(X^*))G(X)
\nonumber \\ & 
\stackrel{\sim}{\lra} F(X)(F(X^*)G(X))
\xrightarrow{\id \otimes \phi_{X^*} \otimes \id} F(X)(G(X^*)G(X))
\nonumber \\ & 
\xrightarrow{\id \otimes \ev^G_{G(X)}} F(X) \one
\stackrel{\sim}{\lra} F(X) \Big]
\end{align*}
is the two-sided inverse to $\phi$.
\end{enumerate}
\end{proof}

\subsection{Alternative description of monoidal natural isomorphisms}\label{app:mon-nat-iso-alt}

After describing symmetric monoidal functors via  a different-from-usual set of data and conditions in Lemma \ref{lem:sym-mon-fun-via-Y}, we now turn to describing monoidal natural isomorphisms (and only those) in that language. 

\medskip

Let $\mathcal{C}$ and $\mathcal{D}$ be symmetric monoidal categories, let $\mathcal{C}$ have left duals, and let $\gamma,\delta$ be as in Appendix \ref{app:alt-sym-mon-fun} satisfying the conditions in Lemma \ref{lem:sym-mon-fun-via-Y}\,(ii). Let
$$
	\phi_X\colon \gamma(X) \lra \delta(X)
	\, , \quad X \in \mathcal{C} \, ,
$$
be a family of isomorphisms in $\mathcal{D}$. 

\begin{lemma}\label{lem:sym-mon-iso-via-Y}
The following are equivalent.
\begin{enumerate}
\item $(\phi_X)_{X \in \mathcal{C}}$ is a monoidal natural isomorphism $\gamma \to \delta$ between the symmetric monoidal functors $\mathcal{C} \to \mathcal{D}$ obtained by extending $\gamma,\delta$ via Lemma \ref{lem:sym-mon-fun-via-Y}.
\item The family $\phi_X$ satisfies:
\begin{enumerate}
\item \textsl{(Naturality)} the diagram
$$
	\xymatrix{
	\gamma(X) \ar[rr]^{\phi_X} && \delta(X)
	\\
	\gamma(\one_{\mathcal{C}})  \ar[u]^{\gamma(f)} 
	& \one_{\mathcal{D}} \ar[r]_{\delta^0} \ar[l]^{\gamma^0} & 
	\delta(\one_{\mathcal{C}})\ar[u]_{\delta(f)} 
	}
$$
commutes for all $X \in \mathcal{C}$ and $f\colon \one \to X$.
\item \textsl{(Monoidality)} the diagram
$$
	\xymatrix{
	\gamma(X) \otimes_{\mathcal{D}} \gamma(Y) 
	\ar[rr]^{\gamma^2_{X,Y}} 
	\ar[d]_{\phi_X \otimes \phi_Y}
	&& \gamma(X \otimes_{\mathcal{C}} Y)
	\ar[d]^{\phi_{X \otimes Y}}
	\\
	\delta(X) \otimes_{\mathcal{D}} \delta(Y) 
	\ar[rr]_{\delta^2_{X,Y}} 
	&& \delta(X \otimes_{\mathcal{C}} Y)
	}
$$
commutes for all $X,Y \in \mathcal{C}$.
\end{enumerate}
\end{enumerate}
\end{lemma}

\begin{proof}[Sketch of proof.]
By Lemma \ref{lem:sym-mon-fun-via-Y} we may assume $\gamma$ and $\delta$ to be symmetric monoidal functors $\mathcal{C} \to \mathcal{D}$ for both directions of the proof. 

\medskip

\noindent
{}From (i) to (ii) is the easy direction: Condition a) is a combination of naturality and the unit condition of a natural monoidal transformation, and b) is just the same as for a natural monoidal transformation.

\medskip

\noindent
{}From (ii) to (i): Monoidality follows from b) together with a) evaluated on $X = \one_{\mathcal{C}}$ and $f=1_\one$.
It remains to verify that $\phi_X$ is natural for all morphisms $f\colon X \to Y$ in $\mathcal{C}$, not just those with $X = \one$, which is a little more tiresome. Here are the main steps:

\medskip

\noindent
\textsl{Step 1:} Check the identity
\begin{align*}
&\Big[\,
\one
\stackrel{\sim}{\lra}
\gamma(\one)
\xrightarrow{\gamma(\coev_X)}
\gamma(X X^*)
\stackrel{\sim}{\lra}
\gamma(X) \gamma(X^*)
\xrightarrow{\phi_{X} \otimes \phi_{X^*}}
\delta(X) \delta(X^*)
\,\Big]
\\
&=~
\Big[\,
\one
\stackrel{\sim}{\lra}
\delta(\one)
\xrightarrow{\delta(\coev_X)}
\delta(X X^*)
\stackrel{\sim}{\lra}
\delta(X) \delta(X^*)
\,\Big] \, .
\end{align*}
Note that the corresponding identity for $\ev_X$ does a priori not hold as condition~a) only imposes naturality for morphisms out of $\one_{\mathcal{C}}$, not into $\one_{\mathcal{C}}$.

\medskip

\noindent
\textsl{Step 2:} 
The Zorro moves 
in Lemma \ref{lem:sym-mon-fun-via-Y}\,(ii\,a) fix $d^\gamma_X$ uniquely. Define
$$
\tilde d^\gamma_X
~=~
\Big[\,
\gamma(X) \gamma(X^*)
\xrightarrow{\phi_{X} \otimes \phi_{X^*}}
\delta(X) \delta(X^*)
\xrightarrow{d^\delta_X}
\one
\,\Big] 
$$
and verify that the properties of $\gamma$ and the result of step 1 give the Zorro moves for $\tilde d^\gamma_X$, but postcomposed with $\phi_X$, respectively $\phi_{X^*}$. At this point in the proof we need to make use of invertibility of $\phi_X$ (and $\phi_{X^*}$) in order to obtain the Zorro moves for $\tilde d^\gamma_X$ and hence the identity
$$
	\tilde d^\gamma_X ~=~ d^\gamma_X \, .
$$

\smallskip

\noindent
\textsl{Step 3:} Recall the definition of $\gamma(f)\colon \gamma(X) \to \gamma(Y)$ for $f\colon X\to Y$ in terms of $\gamma(\tilde f)$ for $\tilde f\colon \one \to Y\otimes X^*$ from the proof of Lemma \ref{lem:sym-mon-fun-via-Y}.
Substituting this into the naturality square
$$
\xymatrix{
	\gamma(X) \ar[r]^{\gamma(f)} \ar[d]_{\phi_X} & \gamma(Y)\ar[d]^{\phi_Y}
	\\
	\delta(X) \ar[r]_{\delta(f)} & \delta(Y)
	}
$$
and using step 2, one checks that the square commutes.
\end{proof}

\end{document}